\documentclass[11pt]{article}
\usepackage{latexsym}
\usepackage{amsmath}
\usepackage{amssymb}
\usepackage{amsthm}
\usepackage{epsfig}
\usepackage{float}
\usepackage{xcolor}
\usepackage{graphicx}
\usepackage{authblk}
\usepackage{bm}
\usepackage{accents}
\usepackage{enumitem}
\usepackage{mathtools}
\usepackage{graphicx}
\usepackage{tikz}
\usepackage{ mathrsfs }
\usepackage[toc,page]{appendix}
\usepackage{graphicx}
\usepackage{bbm}
\usepackage[top=1in, bottom=1in, left=1in, right=1in]{geometry}
\usepackage{hyperref}
\usepackage[normalem]{ulem}
\usepackage{subcaption}
\hypersetup{
    colorlinks,
    linkcolor={blue!80!black},
    citecolor={green!50!black},
}
\colorlet{linkequation}{blue}

\usepackage{comment}

\renewcommand{\P}{\mathbb{P}}
\newcommand{\E}{\mathbb{E}}

\newcommand{\R}{\mathbb{R}}
\newcommand{\C}{\mathbb{C}}

\renewcommand{\S}{\mathbb{S}}

\newcommand{\eps}{\varepsilon} 

\def\bI{{\mathbf I}}

\newcommand{\<}{\langle}
\renewcommand{\>}{\rangle}

\newcommand{\diag}{\text{diag}}
\newcommand{\tr}{\mathrm{tr}}

\newcommand{\ones}{\bm{1}}

\def\bzero{{\boldsymbol 0}}

\newtheorem{theorem}{Theorem}[section]
\newtheorem*{theorem*}{Theorem}
\newtheorem{lemma}[theorem]{Lemma}

\newtheorem{corollary}[theorem]{Corollary}

\theoremstyle{definition}

\newtheoremstyle{myremark} 
    {\topsep}                    
    {\topsep}                    
    {\rm}                        
    {}                           
    {\bf}                        
    {.}                          
    {.5em}                       
    {}  

\theoremstyle{remark}
\newtheorem{remark}{Remark}[section]

\DeclareSymbolFont{rsfs}{U}{rsfs}{m}{n}
\DeclareSymbolFontAlphabet{\mathscrsfs}{rsfs}

\newcommand{\ol}[1]{\overline{\mkern-1mu #1 \mkern-1mu}}

\def\bA{{\boldsymbol A}}
\def\bB{{\boldsymbol B}}

\def\bD{{\boldsymbol D}}

\def \bK{\textit{\scalebox{.98}[1]{\textbf K}}}

\def\bN{{\boldsymbol N}}

\def\bQ{{\boldsymbol Q}}
\def\bR{{\boldsymbol R}}
\def\bS{{\boldsymbol S}}

\def\bV{{\boldsymbol V}}

\def\bW{{\boldsymbol W}}
\def\bX{{\boldsymbol X}}
\def\bY{{\boldsymbol Y}}
\def\bZ{{\boldsymbol Z}}

\def\ba{{\boldsymbol a}}
\def\bb{{\boldsymbol b}}
\def\be{{\boldsymbol e}}

\def\bj{{\boldsymbol j}}

\def\bu{{\boldsymbol u}}
\def\bv{{\boldsymbol v}}
\def\bw{{\boldsymbol w}}
\def\bx{{\boldsymbol x}}
\def\by{{\boldsymbol y}}
\def\bz{{\boldsymbol z}}

\def\bxi{{\boldsymbol \xi}}

\def\bDelta{{\boldsymbol \Delta}}
\def\bLambda{{\boldsymbol \Lambda}}

\def\bSigma{{\boldsymbol \Sigma}}

\def\cR{\mathcal{R}}

\def\tr{{\rm tr}}

\def\cV{{\mathcal V}}
\def\cG{{\mathcal G}}

\def\cP{{\mathcal P}}

\def\cQ{{\mathcal Q}}
\def\cL{{\mathcal L}}
\def\cF{{\mathcal F}}

\def\cV{{\mathcal V}}
\def\cG{{\mathcal G}}

\def\cP{{\mathcal P}}

\def\normal{{\sf N}}

\def\normal{{\sf N}}

\def\normal{{\sf N}}

\def\cF{{\mathcal F}}

\def\normal{{\sf N}}

\def\bDelta{{\boldsymbol \Delta}}

\def\bA{{\boldsymbol A}}

\def\bLambda{{\boldsymbol \Lambda}}

\def\cM{{\mathcal M}}

\def\cV{{\mathcal V}}

\def\diag{{\rm diag}}
\def\bS{{\boldsymbol S}}

\def\bD{{\boldsymbol D}}

\def\bb{{\boldsymbol b}}

\def\spa{\hspace{-.14em}}
\def\spb{\hspace{-.08em}}
\def\spc{\hspace{-.05em}}

\def\bR{{\boldsymbol R}}

\def\Im{{\rm Im}}

\def\cJ{\mathcal{J}}

\def\Im{{\rm Im}}

\def\normal{{\mathcal{N}}}

\newcommand{\edit}[1]{\textcolor{olive}{#1}}

\usepackage{graphicx, amssymb,amsmath,amsthm,xcolor,enumitem,mathtools}
\numberwithin{equation}{section}

\title{Sparse PCA: Phase Transitions in the  Critical Regime
}
\author{Michael J.\ Feldman\footnote{Stern School of Business, New York University 
}, \;Theodor Misiakiewicz\footnote{Department of Statistics and Data Science, Yale University}, \;and Elad Romanov\footnote{Department of Mathematics, Texas A\&M University}}

\date{}
\begin{document}

\maketitle

\begin{abstract}

This work studies sparse principal component analysis (PCA) in high dimensions. Given $n$ independent $p$-dimensional Gaussian samples with covariance $ \bSigma \coloneqq (\lambda -1 ) \bv \bv^\top + \bI_p,$ our goal is to estimate $\bv$ under the assumption of sparsity. On the one hand, if the sparsity level $m \coloneqq \|\bv\|_0$ satisfies $m  \lesssim \sqrt{n}$, algorithms such as  covariance thresholding (Krauthgamer et al., 2015) consistently outperform PCA. On the other hand, if $m \gg \sqrt{n}$, it is conjectured that no polynomial-time algorithm can recover $\bv$ below the detection threshold of PCA. 
We investigate the ``critical" high-dimensional  regime, where $n, p, m \rightarrow \infty$ with $m/\sqrt{n} \rightarrow \beta$ and $p/n \rightarrow \gamma$, and study estimators based on kernel PCA, generalizing covariance thresholding.

Within this framework, we achieve a fine-grained understanding of signal detection and recovery. Our main result establishes a detection phase transition, analogous to the Baik--Ben Arous--P\'ech\'e  (BBP) transition for PCA: above a  signal strength threshold---depending on the kernel function, $\gamma$, and $\beta$---kernel PCA is informative. Conversely, below the threshold, kernel principal components are asymptotically orthogonal to the signal.  
Notably, (1) above this threshold, consistent support recovery is possible with high probability, (2) for all $\beta \in (0, \infty)$, kernel PCA strictly outperforms PCA, and (3) as $\beta \rightarrow \infty$, kernel PCA and PCA coincide. We identify optimal kernel functions for detection and support recovery, and numerical calculations suggest that soft thresholding is nearly optimal. Our key technical contribution is approximation guarantees for deterministic equivalents of kernel random matrices, which enable sharp estimates of  coordinate fluctuations of kernel principal components.

\end{abstract}

\section{Introduction} \label{sec:1}

From factor analysis to covariance estimation to matrix factorization, principal component analysis (PCA) is a standard tool for dimensionality reduction and low-rank signal recovery. Given data with $n$ observations and $p$ variables, the principal components are the eigenvectors of the $p \times p$ sample covariance matrix. PCA, however, has two key drawbacks: first, the principal components are generally non-sparse linear combinations of all $p$ variables. This  limits interpretability, and fails to capture the fact that  signal matrices are often simultaneously low-rank and sparse---such matrices arise in diverse fields, including genetics, computer vision, imaging, and neuroscience \cite{Asp1, Asp2, Friedman, Candes, MingYu}. Second, in high dimensions (when $p$ is comparable to or greater than $n$), PCA is inconsistent: the principal components are inconsistent estimators of the eigenvectors of the population covariance \cite{Paul07}. 

These issues have spurred the development of numerous alternatives to classical PCA that produce sparse principal components, collectively referred to as sparse PCA methods \cite{jolliffe2003modified,d2004direct,zou2006sparse, Asp1, cai2013sparse,vu2013fantope, Krauthgamer}. In this literature, the assumption of a low-rank and sparse ground truth is often implicit in applied and methodological papers, and explicit in theoretical analyses \cite{Lei15}. Whether or not there is evidence in a particular problem that the ground truth is sparse, Friedman et al. advise ``betting on sparsity":  ``Use a procedure that does well in sparse problems, since no procedure does well in dense problems'' \cite{Friedman, FriedmanBook}.

PCA has been extensively studied in high dimensions, revealing several phenomena that are absent in classical fixed-$p$ statistics. A prototypical model in the theoretical literature is
Johnstone's spiked covariance model (henceforth referred to as the ``spiked model'') \cite{JohnstoneLu}. Under this model, $n$ independent samples are observed from a $p$-dimensional Gaussian distribution $\mathcal{N}(0,\bSigma)$, where the population covariance $\bSigma \in \mathbb{R}^{p \times p}$ is a rank-one perturbation of identity: 
\begin{align} \label{spiked_model}    \bSigma \coloneqq (\lambda -1) \bv \bv^\top  + \bI_p . \end{align} 
Here, $\lambda>1$ is the maximum eigenvalue of $\bSigma$ and our goal is to estimate the corresponding eigenvector or ``spike" $\bv\in \mathbb{S}^{p-1}$. Let $\bZ\bSigma^{1/2} \in \R^{n\times p}$ denote the data matrix with the samples as rows, where $\bZ\in \R^{n\times p}$ contains i.i.d.\ standard Gaussian elements. 

PCA estimates $\bv$ by the leading eigenvector $\bu_1$ of the sample covariance \[\bY \coloneqq n^{-1} \bSigma^{1/2} \bZ^\top \spb \bZ \bSigma^{1/2}\]
(which is a sufficient statistic for $\bv$).
In low dimensions, where $n \to \infty$ and $p/n \to 0$, $\bu_1$ and $\bY$ are consistent estimators of $\bv$ (in terms of cosine similarity) and $\bSigma$ (in operator norm), respectively: $\langle \bu_1, \bv\rangle^2 \xrightarrow{a.s.} 1$ and $\|\bY-\bSigma\|_2\xrightarrow{a.s.} 0$. In high dimensions, however, these consistency results no longer hold. 
  As $n,p \rightarrow \infty$ with $p/n \rightarrow \gamma \in (0,\infty)$ 
  (and $\lambda$ fixed),
the empirical spectral distribution (ESD) of $\bY$ converges weakly almost surely to the Marchenko--Pastur law, supported on $[(1-\sqrt{\gamma})^2,(1+\sqrt{\gamma})^2]$. Moreover, the maximum eigenvalue $\lambda_1$ of $\bY$ and $\bu_1$ satisfy 
\begin{gather} \label{asdfzxcv}
     \lambda_1\xrightarrow{a.s.} \begin{dcases}
        \lambda + \frac{\lambda \gamma}{\lambda-1}  & \lambda  > 1 + \sqrt{\gamma} , \\  (1+\sqrt{\gamma})^2 & \lambda \leq 1 + \sqrt{\gamma}, 
    \end{dcases} \\
    \langle \bu_1, \bv \rangle^2 \xrightarrow{a.s.} \begin{dcases}
        \Big(1 - \frac{\gamma}{(\lambda-1)^2}\Big)\Big/ \Big(1 + \frac{\gamma}{\lambda-1}\Big)  & \lambda  > 1 + \sqrt{\gamma} , \\  0 & \lambda \leq 1 + \sqrt{\gamma}. \label{asdfzxcv-1}
    \end{dcases}
\end{gather}    

Remarkably, PCA experiences a {\it phase transition}: 
if $\lambda \leq 1+\sqrt{\gamma}$, $\lambda_1$ converges to the upper edge of the support of the Marchenko--Pastur law, and $\bu_1$ is asymptotically orthogonal to $\bv$. On the other hand, if $\lambda > 1 + \sqrt{\gamma}$, $\lambda_1$ is an {\it outlier eigenvalue}, occurring outside the Marchenko--Pastur support, and the cosine similarity between $\bu_1$ and $\bv$ is non-trivial; that is, $\bu_1$ contains (partial) information  about $\bv$.  
Key papers in the development of (\ref{asdfzxcv}) and (\ref{asdfzxcv-1}) include works of Baik, Ben Arous, and P\'ech\'e \cite{BBP}, Baik and Silverstein \cite{Baik}, and Paul \cite{Paul07}, and the threshold $1+\sqrt{\gamma}$ is known as the Baik-Ben Arous-P\'ech\'e (BBP) transition.  For a full list of references, see the survey \cite{PaulPCA}.

When $\bv$ is uniformly distributed on $\mathbb{S}^{p-1}$, consistent detection (testing between $\lambda = 1$ and $\lambda > 1$) and non-trivial estimation are impossible for weak signal strengths below the BBP transition  \cite{Perry}. However, if $\bv$ is {\it sufficiently} sparse, these key objectives  become possible for such weak signals  
through alternative methods to PCA. 
In a seminal work, Johnstone and Lu \cite{JohnstoneLu} proposed the diagonal thresholding (DT) algorithm, which estimates the support of $\bv$ by the indices of the largest-magnitude diagonal elements of $\bY$.\footnote{Given the support of $\bv$, PCA applied to the corresponding submatrix of $\bY$ consistently estimates $\bv$.}
Assuming knowledge of the sparsity level $m \coloneqq \|\bv\|_0$ and that the entries of $\bv$ satisfy $v_i\in \{0,\pm 1/\sqrt{m})$, 
Amini and Wainwright \cite{amini} proved that DT recovers the support of $\bv$ with high probability if $m \leq C(\lambda) \sqrt{n/\log p}$. 

 Building upon DT,  Krauthgamer, Nadler, and Vilenchik \cite{Krauthgamer} proposed covariance thresholding (CT). Informally, CT (1) hard-thresholds the entries of $\bY$, (2) computes the principal eigenvector $\bu_1$ of the thresholded matrix, and (3) estimates the support of $\bv$ from the largest-magnitude elements of  $\bu_1$. Deshpande and Montanari \cite{deshp2016sparse} subsequently suggested using soft thresholding  and setting the diagonal of $\bY$ to zero in step (1).  They proved that this modified algorithm (also referred to as CT) 
recovers the support of $\bv$ with high probability if $m \leq C(\lambda) \sqrt{n}$, thereby confirming a conjecture of \cite{Krauthgamer}. In particular, 
below the BBP transition where PCA fails, CT succeeds if $m/\sqrt{n}$ is sufficiently small. Semidefinite programming approaches were unable to improve on this result \cite{amini,Krauthgamer} (see the discussion in \cite{deshp2016sparse}).

This body of research raises the following question: is the condition $m \lesssim \sqrt{n}$ necessary for computationally efficient estimation below the BBP transition? Ignoring computational considerations, a weaker condition suffices: by exhaustively searching over $m \times m$ principal submatrices of $\bY$, exponential-time algorithms can recover the support of $\bv$ if $m \log p \lesssim n$ (for any fixed $\lambda > 1$), which is information-theoretically optimal \cite{deshp2016sparse, Yunzi}. Conversely, work of Berthet and Rigollet \cite{Rigollet} and ensuing papers strongly suggest that no polynomial-time algorithm can non-trivially estimate $\bv$ below the BBP transition if $m \gg \sqrt{n}$.  In summary, ``the sparse PCA problem demonstrates a fascinating interplay between computational and statistical barriers" \cite{deshp2016sparse}.



This work focuses on the ``critical" regime $m  \asymp \sqrt{n}$;  as discussed above, this is the coarsest sparsity level at which $\bv$ is believed recoverable below the BBP transition in polynomial time. 
Accordingly, we assume $m/\sqrt{n}\rightarrow \beta \in (0, \infty)$ and, as in \cite{amini, Krauthgamer}, work under the spiked model (\ref{spiked_model}) 
with $v_i\in \{0,\pm1/\sqrt{m}\}$. We consider generalized covariance thresholding (GCT), which computes the maximum eigenvalue $\lambda_1$ and corresponding eigenvector $\bu_1$ of $\bK(f) \in \mathbb{R}^{p \times p}$, defined as follows:
\begin{align}    \label{0} \hspace{3cm}
(\bK(f))_{ij} \coloneqq \begin{dcases} \frac{1}{\sqrt{n}} f( \sqrt{n}  \hspace{.05em} \bY_{ij}) & i \neq j , \\ 0 & i = j     ,
\end{dcases}  & \hspace{2cm} i,j \in [p] ,
\end{align}
where $f:\mathbb{R} \rightarrow \mathbb{R}$ is a kernel function.  {That is, $\bK(f)$ is an entrywise transformation of $\bY$; we do not assume that $\bK(f)$ is positive semidefinite or admits a feature map representation.} 
We estimate the support of $\bv$ by hard thresholding the entries of $\bu_1$. CT corresponds to the hard-thresholding kernel $\eta_h(x,t) \coloneqq x \cdot {\bf 1}_{|x| \geq t}$ in \cite{Krauthgamer} and soft-thresholding $\eta_s(x, t) \coloneqq \mathrm{sign}(x) \cdot(|x|-t)_+$ in \cite{deshp2016sparse}.

\subsection{Summary of Contributions} \label{sec:contributions}

The objective of this paper is to develop a comprehensive theory of GCT. Specifically, we answer in the affirmative the following question: does the high-dimensional PCA theory outlined above  (namely, (\ref{asdfzxcv}) and (\ref{asdfzxcv-1})) extend to GCT?
Moreover, we resolve an open question of
Desphande and Montanari \cite{deshp2016sparse}:
under what range of  sparsity levels $\beta$ does CT (or GCT)  (1) outperform PCA by producing a non-trivial estimate of $\bv$ below the BBP transition, and (2) consistently recover the support of $\bv$?   We now summarize our contributions: 
\begin{enumerate} 
    \item 
    For a general class of kernel functions, we establish a phase transition, analogous to the BBP transition.\footnote{For brevity, we describe here a slightly simplified version of our results.} Above a signal strength threshold $\lambda_*(f,\gamma,\beta)$---depending on $f$, the limiting aspect ratio $\gamma$, and the limiting sparsity level $\beta$---GCT is informative: the spectrum of $\bK(f)$ contains an outlier eigenvalue and the cosine similarity between $\bu_1$ and $\bv$ is non-trivial. Below $\lambda_*(f,\gamma,\beta),$ $\bu_1$ is asymptotically orthogonal to $\bv$. We provide limiting formulas for eigenvalue and eigenvector inconsistency,  analogous to (\ref{asdfzxcv}) and (\ref{asdfzxcv-1}):
\begin{align}  \label{cvbn}
     \lambda_1&\xrightarrow{a.s.} \begin{dcases}
       \psi(f, \gamma,\beta, \lambda)  & \lambda  > \lambda_* (f,\gamma,\beta) , \\  \lambda_+ & \lambda \leq \lambda_* (f,\gamma,\beta), 
    \end{dcases} \\ \label{cvbn2}
    \langle \bu_1, \bv \rangle^2 &\xrightarrow{a.s.} \begin{dcases}
        \theta^2 (f, \gamma, \beta, \lambda)  & \lambda  > \lambda_* (f,\gamma,\beta) , \\  0 & \lambda \leq\lambda_* (f,\gamma,\beta).
    \end{dcases}
\end{align} 
Expressions for $\psi(f, \gamma,\beta, \lambda)$ and $\theta^2(f,\gamma,\beta, \lambda)$ are provided in Section \ref{sec:results}.  

Furthermore, above $\lambda_*(f,\gamma,\beta)$, the coordinates of $\bu_1$ concentrate tightly around those of $\theta(f,\gamma,\beta)\bv$. Consequently, appropriate hard-thresholding of $\bu_1$  recovers the support of $\bv$ {\it exactly}, with high probability. 

  \item We  prove the existence of and characterize an optimal kernel function (depending on $\beta$) with phase transition 
  \[
  \lambda_*(\gamma,\beta) \coloneqq \inf_{f \in L^2(\phi)} \lambda_*(f,\gamma,\beta) , 
  \]
  where $L^2(\phi)$ is the set of odd $L^2$ functions with respect to the Gaussian measure. With this kernel choice, GCT strictly outperforms  PCA for all values of $\beta \in (0,\infty)$, and as $\beta \rightarrow \infty$, GCT and PCA coincide (see Figure \ref{fig:enter-label}). 
  Numerical calculations suggest that CT is nearly optimal, and we propose to adaptively select the thresholding level to maximize the (normalized) spectral gap. That is, we select the thresholding level that  empirically produces the most ``clear" outlier eigenvalue.

    
    \item To prove (1), we develop sharp bounds on the coordinate fluctuations of kernel principal components. These bounds are based on novel approximation guarantees for deterministic equivalents of kernel matrices---Theorems \ref{lem:quad_forms} and \ref{lem:quad_formsX}---which achieve optimal rates of convergence and may be of independent interest. General-purpose entrywise eigenvector perturbation bounds (such as those in \cite{abbe}) are insufficient for our purposes. 
     
\end{enumerate}

\noindent 

The paper is organized as follows: Section \ref{sec:results} presents our main results, Section \ref{sec:2.4} contains numerical calculations and simulations, Section \ref{sec:iso_law} states deterministic equivalents for kernel matrices, and Section \ref{sec:proofs} and Appendices \ref{sec:appendix1}--\ref{sec:poly approx} contain proofs.


\begin{figure}[]
 \begin{center}
 \def\myLength{5}
 \begin{tikzpicture}
    \node[scale=.9] (scaled) {
\begin{tikzpicture}[scale=2]

\fill[green!20, opacity=0.5] (0,2) -- (\myLength,2) -- (\myLength,3) -- (0,3) -- cycle;

\fill[blue!20, opacity=0.5] (0,1) -- plot[domain=0:\myLength,samples=100] (\x,{0. + 2*(1 - 1/(\x+1)^1.25)})  -- (\myLength,2) -- (0,2) -- cycle;

\fill[red!20, opacity=0.5] (0,0)  -- plot[domain=0:\myLength,samples=100] (\x,{0 + 2*(1 - 1/(\x+1)^1.25)}) -- (\myLength,0) -- cycle;

\draw[->] (0,0) -- (\myLength,0) node[right] {};
\draw[->] (0,0) -- (0,3) node[above] {};

\node at (\myLength,-0.15) {$\infty$};
\node at (0,-0.15) {$0$};
\node at (\myLength/2,-0.25) {{\large $\beta$}};

\node at (-0.15,0) {$1$};
\node at (-0.15,3) {$\infty$};
\node at (-0.25,1.5) {{\large $\lambda$}};
\draw[thick, line width=1.2pt, green!60!black] (0,2) -- (\myLength,2) node[right] {};
\node[green!60!black] at (4.25,2.15) {{\large $\lambda = 1 + \sqrt{\gamma}$}};
\node[green!60!black] at (1,2.4) {{\large $({\rm I})$ PCA succeeds}};
\draw[thick, line width=1.2pt, blue!100!black] plot[domain=0:\myLength,samples=100] (\x,{0 + 2*(1 - 1/(\x+1)^1.25)}) 
node[right] {};
\node[blue!100!black, rotate = 7] at (2.85,1.77) {{\large $\lambda = \lambda_*(\gamma,\beta)$}};
\node[blue!100!black] at (1.,1.7) {{\large $({\rm II})$ GCT succeeds}};
\node[red!60!black] at (2.8,0.7) {{\large $({\rm III})$ GCT fails}};
\end{tikzpicture} };
\end{tikzpicture}
    \caption{Stylized  illustration of our results as $p/n \rightarrow \gamma$. Area (I) is the recovery region of PCA, the union of areas (I) and (II) is the recovery region of  GCT, and (III) is the GCT impossibility region. The boundary between regions (II) and (III), the curve $\lambda_*(\gamma, \beta)$, 
    is the phase transition location of the optimal kernel function, characterized in Section \ref{sec:2.3}. As $\beta \rightarrow \infty$, $\lambda_*(\gamma,\beta) \rightarrow 1+ \sqrt{\gamma}$ (the BBP transition).  
    }
    \label{fig:enter-label}
\end{center}
\end{figure}

\subsection{Related Work} \label{sec:RW}

Deshpande and Montanari \cite{deshp2016sparse} proved that CT recovers the support of $\bv$ with high probability if $m \leq C(\lambda)\sqrt{n}$, thereby confirming a conjecture of \cite{Krauthgamer}. Their analysis is non-asymptotic and applies to a more general spiked model than  we consider here. 
El Amine Seddik, Tamaazousti, and Couillet \cite{CouilletSeddik} studied a variant of GCT in the sub-critical sparsity regime $m \ll \sqrt{n}$. They do not establish a phase transition or consider support recovery. 

Recently, Novikov \cite{Novikov} proposed a novel sparse PCA algorithm building on ideas from the planted clique problem \cite{alon1998finding}. Given an integer $k$, the algorithm recovers $\bv$ with high probability if 
\[\lambda \geq \frac{Cm}{\sqrt{kn}}\sqrt{\log(2+kp/m^2)} \] and runs in time $np^{O(k)}$ (GCT runs in time $O(np^2)$). In the critical regime where $p/n \rightarrow \gamma$ and $m/ \sqrt{n} \rightarrow \beta$,  for any fixed $\lambda > 1$ and $\beta \in (0, \infty)$, recovery is possible in polynomial time by choosing a sufficiently large constant $k$. Novikov does not establish a phase transition or derive exact limits as in (\ref{cvbn}) and (\ref{cvbn2}). We remark that the algorithm is highly tailored to the spiked model \eqref{spiked_model}, making it primarily of theoretical interest.


A related body of work, including \cite{rangan2012iterative,deshp2014sparse,lesieur2015phase,Flo1,Perry,Flo3,Venkat}, considers estimation in the spiked model \eqref{spiked_model} where 
$v_1, \ldots, v_p$ are i.i.d.\ samples from a  (fixed) prior distribution.  By assuming a prior with a positive point mass on zero, this framework covers sparse PCA in the \emph{linear} sparsity regime where $m \asymp n$. Deshpande and Montanari \cite{deshp2014sparse}  considered a Bernoulli prior: $v_1, \ldots, v_p \sim \mathrm{Ber}(\rho)$.
While non-trivial estimation of $\bv$ is impossible below the BBP transition, PCA becomes sub-optimal above the transition when $\rho$ is sufficiently small. Specifically, there exists a critical value $\rho_c \leq 0.05$, such that if $\rho<\rho_c$, an efficient approximate message passing (AMP) algorithm is Bayes optimal \cite{deshp2014sparse} {for recovery}. 
{
Complementary results are obtained in \cite{Perry}, which considers the prior where $\P(v_i = 0) = 1-\rho$ and $\P(v_i = \pm 1) = \rho/2$: for all sufficiently large $\rho$, PCA is optimal for detection, in the sense that detection is impossible below the BBP transition.}

DT and CT are connected to the works of Bickel and Levina \cite{bickel08}, Ma \cite{zongming}, and Cai, Ma, and Wu \cite{Tony}, all of which use thresholding to induce sparsity.  In contrast to the spiked model, the first paper studies estimation of covariance matrices with $L^q$ quasi-norm sparse rows, where $q \in [0,1)$. 
The third paper considers minimax estimation of the principal subspace of spiked models of the form $\bSigma = \bV \bLambda \bV^\top + \bI_p$, where $\bV$ has $L^q$-sparse rows. They develop a data-driven algorithm based on DT which runs in polynomial time and achieves the minimax rate of convergence over a subset of the parameter space. Specializing their results to our model and $L^0$-sparsity, their  estimator is suboptimal by a logarithmic factor: it requires $m \leq C(\lambda)\sqrt{n/\log p}$ to provably outperform PCA (see Theorem 7 and Proposition 1 of \cite{Tony}). 

Rather than thresholding, many sparse PCA methods induce sparsity through $L_1$-penalized regression. We intend to study the theoretical properties of the elastic-net-based method developed by Zou, Hastie, and Tibshirani \cite{zou2006sparse} in future work. 


Key studies of high-dimensional kernel matrices include the works of El Karoui \cite{karoui2010spectrum}, Cheng and Singer, \cite{cheng2013spectrum}, and Fan and Montanari \cite{fan2019spectral}.
El Karoui considered kernel matrices with off-diagonal elements $f(\bY_{ij})$ (instead of $n^{-1/2} f(\sqrt{n} \hspace{.05em}\bY_{ij})$ as in (\ref{0})), in which case the kernel matrix acts as a spiked covariance matrix, plus (1) a non-informative low-rank term and (2) a multiple of the identity. Consequently, the leading kernel principal components are asymptotically orthogonal to  $\bv$ below the BBP transition. 
The second and third papers pertain to kernel matrices of isotropic data (that is, $\bSigma = \bI_p$). Cheng and Singer characterized the limiting spectral distribution (LSD) of $\bK(f)$ as defined in (\ref{0}), which is the additive free convolution of the (scaled) Marchenko--Pastur and semicircle laws; see Section \ref{sec:ChenSinger} for details. Subsequently, Fan and Montanari proved that the maximum eigenvalue of $\bK(f)$ converges to the supremum or ``upper edge" of this LSD's support. 


Nonlinear transformations of spiked matrices is an active area of research not limited to sparse PCA, with recent contributions from  Liao, Couillet, and Mahoney \cite{liao2020sparse}, Guionnet et al.\ \cite{Guionnet_nonlin}, Feldman \cite{feldman23}, Wang, Wu, and Fan \cite{wang2024nonlinear}, and Mergny et al.\ \cite{flo4}.
The first paper assumes delocalized (dense) spikes and investigates thresholding and quantization kernels (which reduce PCA's computational cost), finding that their application minimally impacts estimation. The second and third papers study a related model in which a transform $f$ is applied elementwise to the data $\bZ \bSigma^{1/2}$, rather than to the sample covariance $\bY$. Similarly to \cite{liao2020sparse}, spikes are delocalized.  For a more detailed discussion of the technical distinctions between our work and \cite{CouilletSeddik, liao2020sparse, Guionnet_nonlin, feldman23}, see (\ref{eq:taylor}) and the following comments.

\subsection{Model and Notations}

In this section, we summarize the definitions introduced above for ease of reference and state additional notations. 
As in \cite{amini, Krauthgamer}, let $ \bSigma \coloneqq (\lambda -1) \bv \bv^\top  + \bI_p $, where $\lambda \geq 1$ is constant, $m \coloneqq \|\bv\|_0$, and $\bv$ is deterministic (or independent of $\bZ$) and of the form 
\begin{align} \label{eq:v_form}
 &   \sqrt{m} v_1, \sqrt{m} v_2, \ldots, \sqrt{m} v_m \in \{-1, 1\},  & v_{m+1} = v_{m+2} = \cdots = v_p = 0 .
\end{align}
 Let $\bZ \in \mathbb{R}^{n \times p}$ have rows  $\bz_1, \ldots, \bz_n \stackrel{i.i.d.}{\sim} \mathcal{N}(0, \bI_p)$ and \begin{align*}
&  \bS \coloneqq n^{-1}      \bZ^\top \spb \bZ
 , &  \bY \coloneqq \bSigma^{1/2} \bS \bSigma^{1/2} , && \bX \coloneqq \bY - \bS .  \end{align*} 
{This model can be viewed as a second-moment analog of the sparse Gaussian sequence model with a Rademacher prior \cite{Johnstone04}: there the signal is a sparse mean vector corrupted by Gaussian noise, and here the signal is a sparse rank-one perturbation of the identity matrix.}
The assumptions that the non-zero entries of $\bv$ are the first $m$ and that $\bv$ is sparse in the canonical basis are without loss of generality. 

Define kernel matrices $\bK(f), \bK_0(f) \in \mathbb{R}^{p \times p}$ as follows: for $i, j \in [p]$, 
\begin{align*} & (\bK(f))_{ij} \coloneqq \begin{dcases} \frac{1}{\sqrt{n}} f( \sqrt{n}  \hspace{.05em} \bY_{ij}) & i \neq j , \\ 0 & i = j     ,
\end{dcases} & \quad  
(\bK_0(f))_{ij} \coloneqq \begin{dcases} \frac{1}{\sqrt{n}} f( \sqrt{n}  \hspace{.05em} \bS_{ij}) & i \neq j , \\ 0 & i = j     .
\end{dcases}  
\end{align*}
The kernel matrix of noise, $\bK_0(f)$, is unobserved.  We work within a standard asymptotic framework of random matrix theory, where
$n,p,m \rightarrow \infty$ and
\begin{align} \label{asdf2}
&    \frac{p}{n} \rightarrow  \gamma \in (0, \infty)  , && \frac{m}{\sqrt{n}} \rightarrow \beta \in (0, \infty). 
 \end{align}

Let $\lambda_1 \geq \lambda_2 \geq \cdots \geq \lambda_p$ denote the eigenvalues of $\bK(f)$ and $\bu_1, \bu_2, \ldots, \bu_p$ the corresponding eigenvectors. Non-bold symbols will be used to denote the elements of certain matrices and vectors. For example, $\bu_1 = (u_{11}, u_{12}, \ldots, u_{1p})$ and $\bZ = (z_{ij}:1 \leq i \leq n, 1 \leq j \leq p)$.  
The soft- and hard-thresholding operators are  $\eta_s(x, t) \coloneqq \mathrm{sign}(x) \cdot(|x|-t)_+$ and $\eta_h(x,t) \coloneqq x \cdot {\bf 1}_{|x| \geq t}$. 
Let $\phi(z) \coloneqq (2\pi)^{-1/2} e^{-z^2/2}$ denote the Gaussian density and define the inner product on the space of real-valued functions
\[
 \langle f, g \rangle_\phi \coloneqq \int_{\mathbb{R}} f(z) g(z) \phi(z) dz . 
\]
Let $\{h_\ell\}_{\ell \in \mathbb{N}}$ be the Hermite polynomials (normalized such that $\langle h_k, h_\ell \rangle_\phi = {\bf 1}_{k=\ell}$), $\|f\|_\phi^2 \coloneqq \langle f, f \rangle_\phi$, and $a_\ell \coloneqq \langle f, h_\ell \rangle_\phi$; $\{a_\ell\}_{\ell \in \mathbb{N}}$ are  the Hermite coefficients of $f$. We will use $\odot$  to denote the Hadamard (elementwise) product. 

Results for polynomial kernel functions will be stated using the notion of {\it stochastic domination} from \cite{bloe2014isotropic}:  for two sequences of nonnegative random variables $\xi_n$ and $\zeta_n$, we say $\xi_n$ is stochastically dominated by $\zeta_n$ and write $\xi_n \prec \zeta_n$ if for all $\eps, D > 0$,  there exists $n_{\eps,D} \in \mathbb{N}$ such that
\begin{align*}
  \P ( \xi_n > n^\eps \zeta_n ) \leq n^{-D}
\end{align*}
for all $n \geq n_{\eps, D}$. 
If $\xi_n$ is not assumed nonnegative and $|\xi_n| \prec \zeta_n$, we may write $\xi_n = O_\prec(\zeta_n)$. 

\subsection{The Spectrum of $\bK_0(f)$
} \label{sec:ChenSinger}

The following result, Theorem 3.4 of Cheng and Singer \cite{cheng2013spectrum}, characterizes the LSD of the kernel matrix of noise $\bK_0(f)$.\footnote{For brevity, we do not state Theorem 3.4 of \cite{cheng2013spectrum} in its full generality.} Recall that $a_\ell \coloneqq \langle f, h_\ell \rangle_\phi$ is the $\ell$-th Hermite coefficient of $f$.

\begin{theorem}\label{thrm:stj_trans} Let $a_0 = 0$, $\|f\|_\phi^2< \infty$, and $f$ be bounded on compact sets. The ESD of $\bK_0(f)$ converges weakly almost surely to a continuous probability measure $\mu$ on $\mathbb{R}$.   The Stieltjes transform $s(z)$ of $\mu$ solves the equation
\begin{align} \label{eq:stj_trans}    
-\frac{1}{s(z)} = z + a_1 \Big( 1 - \frac{1}{1+a_1 \gamma s(z)}\Big) + \gamma (\|f\|_\phi^2 - a_1^2) s(z) .
\end{align}
For $z \in \mathbb{C}^+$, (\ref{eq:stj_trans}) has a unique solution $s(z)$ with $\Im(s(z)) > 0$. 
\end{theorem}


\begin{corollary}\label{cor:stj_trans}(see Corollary 3 of \cite{liao2020sparse})
   In the setting of Theorem \ref{thrm:stj_trans}, assume $a_1 > 0$ and let $\psi(s)$ denote the functional inverse of $s(z)$:
\begin{align} \label{eq:stj_trans2}
\psi(s) \coloneqq -\frac{1}{s} - a_1 \Big( 1 - \frac{1}{1+a_1 \gamma s}\Big) - \gamma (\|f\|_\phi^2 - a_1^2) s
\end{align}
The supremum or upper edge of $\mathrm{supp}(\mu)$ is given by \[\lambda_+ \coloneqq \psi(s_0),\] where $s_0$ is the unique solution of $\psi'(s) = 0$ in the interval $(-1/(a_1 \gamma), 0)$. 
\end{corollary}

\section{Results} \label{sec:results}

Our results are organized into three sections. Sections \ref{sec:poly} and \ref{sec:Nonpoly} present our core technical findings, with Section \ref{sec:poly} covering polynomial kernels and Section \ref{sec:Nonpoly} non-polynomial kernels. Section \ref{sec:2.3} discusses the  statistical implications of these results.

Before proceeding, we define a key function:
\begin{align} \label{eq:tau_def}
    \tau(f,\beta,\lambda) \coloneqq \sum_{\ell=3}^\infty \frac{a_\ell (\lambda -1)^\ell}{\sqrt{\ell !} \beta^{\ell-1}} .
\end{align}
We shall see that the spectral properties of $\bK(f)$ are determined by $\lambda, a_1, \|f\|_\phi^2$, and $\tau$; in particular, the performance of GCT depends on $\beta$ only through $\tau(f,\beta,\lambda)$.  

\subsection{Polynomial Kernel Functions}\label{sec:poly}

Throughout this section, we assume $f$ is an odd polynomial {(see Remark \ref{Rem2.3} for a discussion of this assumption)},  $L \coloneqq \mathrm{deg}(f)$, and in addition to (\ref{asdf2}),
\begin{align} \label{asdf}
&    \frac{p}{n} = \gamma + O\big(n^{-1/2}\big) 
 \end{align} 
 Let $\beta_n \coloneqq  m/\sqrt{n} = \beta + o(1) .  $ For brevity, we will write $\tau$ in place of $\tau(f,\beta_n,\lambda)$ (note that $\tau$ contains at most $L$ non-zero terms).

 Our first result is that $\bK(f)$ is approximately the sum of three components: a (diagonal centered) sample covariance matrix $a_1 (\bY - \bI_p)$, a rank-one signal term proportional to $\bv \bv^\top $, and a noise matrix $\bK_0(f-a_1 h_1)$. 
 This third term acts as an (asymptotically) independent Wigner matrix with semicircular LSD. 
 
\begin{theorem}\label{thrm:poly1}
Let $f$ be an odd polynomial and (\ref{asdf}) hold. The matrix  \[ \bA(f) \coloneqq a_1 (\bY -  \bI_p) + \tau \cdot \bv \bv^\top  + \bK_0(f-a_1 h_1) \]
approximates $\bK(f)$ in operator norm:
\[
\|\bK(f) - \bA(f)\|_2 \prec n^{-1/4} . 
\]
\end{theorem}

In our view, this is a rather unexpected result. 
Consider the decomposition
\begin{align} \label{eq:taylor} 
    \bK(f) &= \sum_{\ell = 0}^L \frac{1}{\ell!} (\sqrt{n} \bX)^{\odot \ell} \odot \bK_0(f^{(\ell)}) .  
\end{align} 
{which is the Taylor expansion of $f(\sqrt{n}\bY)$ around $\sqrt{n}\bS$, taken entrywise.} 
In related studies such as \cite{liao2020sparse, Guionnet_nonlin, feldman23}, where the signal $\bv$ is dense (for example, if $\bv$ is generated uniformly on $\mathbb{S}^{p-1}$), Hadamard powers of  $\bX$ have vanishing  operator norm: for $\ell > 1$, $\|(\sqrt{n} \bX)^{\odot \ell}\|_2 \rightarrow 0$. Therefore, $\bK(f)$ is well approximated by $\bK_0(f) +\sqrt{n} \bX \odot \bK_0(f')$. In our setting, such powers of $\sqrt{n} \bX$ are no longer negligible.  Representing $f$ in the basis of Hermite polynomials and using the identity $h^{(\ell)}_k(z) = \sqrt{k!/(k-\ell)!} h_{k-\ell}(z)$, (\ref{eq:taylor}) becomes
\begin{equation}
\begin{aligned} 
    \bK(f) &= \sum_{\ell = 0}^L \sum_{k=\ell}^L \frac{a_k}{\ell!} \sqrt{\frac{k!}{(k-\ell)!}} \big( (\sqrt{n} \bX)^{\odot \ell} \odot   \bK_0(h_{k-\ell})\big) .
    \label{eq:taylor_f_L_2}
\end{aligned}
\end{equation} 
Theorem \ref{thrm:poly1} is a consequence of the convergence  
\[\big\|(\sqrt{n} \bX)^{\odot \ell} \odot \bK_0(h_k)\big\|_2 \xrightarrow{a.s.} 0\]
for $\ell, k > 0$---in short, this occurs  because $ \bX$ sparsifies $\bK_0(h_k)$, which is an array of weakly dependent elements with means converging to zero.\footnote{Let $i \neq j$. Since $\sqrt{n} \bS_{ij} \xrightarrow[]{d} \mathcal{N}(0,1)$ and $h_k$  orthogonal to $h_0$, the continuous mapping theorem implies  $\sqrt{n} \hspace{.08em}\E (\bK_0(h_k))_{ij} = \E \hspace{.08em} h_k(\sqrt{n}\bS_{ij}) \rightarrow \langle h_k, h_0 \rangle_\phi = 0$.}

The operator-norm convergence established in Theorem \ref{thrm:poly1} implies that the spectral properties of $\bK(f)$ and $\bA(f)$ are closely related (see Lemma 2.1 of \cite{karoui2010spectrum}). In particular, (1) the ESDs of $\bK(f)$ and $\bA(f)$ converge weakly almost surely to a common limit by Cauchy's interlacing inequality and (2) for every fixed $k \geq 1$, the $k$-th largest eigenvalues of $\bK(f)$ and $\bA(f)$ are asymptotically equal. 

\begin{theorem} \label{thrm:poly2} In the setting of Theorem \ref{thrm:poly1}, assume $a_1 > 0$, $\tau \neq 0$, and define
\begin{align*} s_+ & \coloneqq  \frac{-a_1(\lambda + \gamma-1) - \tau + \sqrt{(a_1(\lambda + \gamma-1) + \tau)^2 - 4 a_1 \gamma \tau}}{2 a_1 \gamma \tau} , \\   \theta^2(x) & \coloneqq -\frac{1+\gamma x(a_1(2 + a_1 \gamma x)(1+a_1^2\gamma x^2) -x(1+a_1 \gamma x)^2 \|f\|_\phi^2)}{x(1+a_1\gamma x)(\tau + a_1(\lambda+ \gamma+2 \tau \gamma x -1))}  .
\end{align*}

\noindent  If $\psi'(s_+) > 0$, 
\begin{align} \label{asdfzxcv2}
 &   |\lambda_1 - \psi(s_+)| \prec n^{-1/2} , & \langle \bu_1, \bv \rangle^2  = \theta^2(s_+)  + O_\prec(n^{-1/2}) ,
\end{align}
and for $i \in [p]$,  assuming $\langle \bu_1, \bv \rangle \geq 0$ without loss of generality, 
\begin{align} \label{asdfzxcv3}
  u_{1i} = \theta(s_+) v_i + O_\prec(n^{-1/2}) .\end{align}
If $\psi'(s_+) \leq 0$, 
\begin{align}
&    \lambda_1 \xrightarrow{a.s.} \lambda_+ , 
& \langle \bu_1, \bv \rangle  \xrightarrow{a.s.} 0 . 
\end{align}
\end{theorem}
We note that  $\psi'(s_+) > 0$ implies $\psi(s_+) > \lambda_+$ and $\theta^2(s_+)  > 0$---see the proofs of (\ref{q6j}) and (\ref{q7j}) in Section \ref{sec:proofs}. Thus, if $\psi'(s_+) > 0$, GCT is informative: $\lambda_1$ is an outlier eigenvalue and the cosine similarity between $\bu_1$ and $\bv$ is non-trivial. Moreover, according to (\ref{asdfzxcv3}), the coordinates of $\bu_1$ concentrate tightly around those of $\theta(s_+)\bv$. Consequently, appropriate hard-thresholding of $\bu_1$ recovers the support of $\bv$ {\it exactly}---see Corollary \ref{thrm:SR}. On the other hand, if $\psi'(s_+) \leq 0$, GCT is non-informative: $\lambda_1$ converges to $\lambda_+$, the supremum of $\mathrm{supp}(\mu)$, and $\bu_1$ is asymptotically orthogonal to $\bv$. For comments on the  cases $a_1 \leq 0$ and $\tau = 0$, which Theorem \ref{thrm:poly2} excludes, see Remark \ref{rem1}.

\begin{remark} \label{rem:bloe}
    We believe the convergence rates in Theorem \ref{thrm:poly2} are optimal as they match results for the spiked model given in \cite{bloe2014isotropic, bloepca} (for example,  (\ref{asdfzxcv2}) and (\ref{asdfzxcv3}) compare to Theorems 2.3 and  2.16 of \cite{bloepca}). In contrast, relating $\lambda_1$ to the maximum eigenvalue of $\bA(f)$ using Weyl's inequality and Theorem \ref{thrm:poly1} yields 
\begin{align*}
|\lambda_1  - \psi(s_+)| & \leq \big|\lambda_1  - \lambda_{\max}(\bA(f)) \big|  + \big| \lambda_{\max}(\bA(f)) - \psi(s_+)\big| \\ &  \lesssim  \|\bK(f) - \bA(f)\|_2 \prec n^{-1/4}. \end{align*}
Rather, we prove that $|\lambda_1 - \psi(s_+)|$ is bounded by quadratic forms such as
\begin{align*} 
\bv^\top (\bK_0(f) - \bz \bI_p)^{-1} (\bK(f) - \bA(f))(\bK_0(f) - \bz \bI_p)^{-1} \bv , 
\end{align*}
which have fluctuations of order $O_\prec(n^{-1/2})$---see Section \ref{sec:iso_law}. 
\end{remark}

\subsection{Non-polynomial Kernel Functions}\label{sec:Nonpoly}

Throughout this section, we write $\tau$ in place of $\tau(f,\beta,\lambda)$. The following results are analogs of Theorems \ref{thrm:poly1} and \ref{thrm:poly2}:
\begin{theorem} \label{thrm:A}
    Let $f(x)$ be odd, everywhere continuous, and twice differentiable except at finitely many points. Assume there is some $c > 0$ such that $|f(x)|,|f'(x)|,|f''(x)| \leq c e^{c|x|}$ whenever they exist. The matrix
\begin{align*}  \bA(f) \coloneqq a_1 (\bY -  \bI_p) + \tau \cdot \bv \bv^\top  + \bK_0(f-a_1 h_1) ,
\end{align*}
approximates $\bK(f)$ in operator norm:
\[
\|\bK(f) - \bA(f)\|_2 \xrightarrow{a.s.} 0 . 
\]
\end{theorem}

\begin{theorem} \label{thrm:B} In the setting of Theorem \ref{thrm:A}, assume $a_1 > 0$ and $\tau \neq 0$, and let $s_+$ and $\theta^2(x)$ be as in Theorem \ref{thrm:poly2}. Then, 
\begin{align} \nonumber 
&    \lambda_1 \xrightarrow{a.s.} \begin{dcases}
        \psi(s_+) & \psi'(s_+) > 0,  \\
        \lambda_+ & \psi'(s_+) \leq 0,  
    \end{dcases} 
& \langle \bu_1, \bv \rangle^2  \xrightarrow{a.s.} \begin{dcases}
        \theta^2(s_+) &\psi'(s_+) > 0,  \\
        0 & \psi'(s_+) \leq 0. 
    \end{dcases} 
\end{align}
\end{theorem}

To  prove Theorems \ref{thrm:A} and \ref{thrm:B} from Theorems \ref{thrm:poly1} and \ref{thrm:poly2}, we construct in Lemma \ref{lem:polynomial_approx} a sequence of odd polynomials  $\{f_\ell\}_{\ell \in \mathbb{N}}$ such that $\|f - f_\ell\|_\phi \rightarrow 0$ and  
 \begin{align} \label{m,./}
     \lim_{\ell \rightarrow \infty} \limsup_{n \rightarrow \infty} \|\bK(f) - \bK(f_\ell)\|_2 = \lim_{\ell \rightarrow \infty} \limsup_{n \rightarrow \infty} \|\bK_0(f) - \bK_0(f_\ell)\|_2 \stackrel{a.s.}{=} 0.
 \end{align}
 This result builds upon \cite{fan2019spectral} in that (1)  Theorems 1.4 and 1.6 of \cite{fan2019spectral} pertain only to the kernel matrix of noise $\bK_0(f)$  and (2) Theorem 1.4 assumes $f$ is continuously differentiable. 
 We stress that the specific conditions Theorem \ref{thrm:A} places on $f$ are not the focus of this paper and are likely improvable. 
 Rather, we developed in (\ref{m,./}) a minimal extension of \cite{fan2019spectral} that accommodates soft thresholding.

\begin{remark} \label{rem1}
    Theorems \ref{thrm:poly2} and \ref{thrm:B} exclude the cases $\tau = 0$ or $a_1 \leq 0$. The former case arises under linear kernel functions (the limits of $\lambda_1$ and $\langle \bu_1, \bv \rangle^2$ are then given by \cite{BBP}) or $\lambda = 1$ ($\bv$ is then unidentifiable). If $a_1 < 0$, our results apply to $\bK(-f)$. If $a_1 = 0$, $\bK(f)$ is approximately a spiked Wigner matrix, and the limits of interest are derived by applying \cite{benaych2011eigenvalues} to $\bA(f)$:\footnote{The argument of \cite{benaych2011eigenvalues} requires that quadratic forms such as $\bv^\top (\bK(f) - z \bI_p)^{-1} \bv$ concentrate around their expectations; we prove the necessary technical conditions in Theorem \ref{lem:quad_forms}.}
    \begin{align} 
&    \lambda_1 \xrightarrow{a.s.} \begin{dcases}
        \tau + \frac{\gamma\|f\|_\phi^2}{\tau}& \tau > \sqrt{\gamma}\|f\|_\phi,  \\
        2 \sqrt{\gamma}\|f\|_\phi & \tau \leq \sqrt{\gamma}\|f\|_\phi,  
    \end{dcases} 
& \langle \bu_1, \bv \rangle^2  \xrightarrow{a.s.} \begin{dcases}
        1- \frac{\gamma \|f\|_\phi^2}{\tau^2}& \tau > \sqrt{\gamma}\|f\|_\phi,  \\
        0 & \tau \leq \sqrt{\gamma}\|f\|_\phi. 
    \end{dcases} 
\end{align}
These are the standard formulas (appropriate scaled) for limiting eigenvalue bias and eigenvector inconsistency under the spiked Wigner model; see Example 3.1 of \cite{benaych2011eigenvalues}.
\end{remark}

\begin{remark}\label{Rem2.3}
   We assume $f$ is odd for two reasons. First, the hard and soft thresholding kernels used in \cite{Krauthgamer,deshp2016sparse} are odd. Second, an even component to $f$ introduces additional spikes into the spectrum of $\bK(f)$, complicating the spectral structure. Specifically, if $f$ is not odd, $\bA(f)$ in Theorem \ref{thrm:A} becomes
    \[
    \bA (f) = a_1 (\bY - \bI_p) + \tau_{e} \cdot \bw \bw^\top + \tau_{o} \cdot \bv \bv^\top  + \bK_0(f-a_1 h_1) ,
    \]
    where $\bw := [\ones_m /\sqrt{m},\bzero_{p-m}]$ and 
    \[
    \tau_{e}  := \sum_{k =1}^{\infty} \frac{a_{2k} (\lambda - 1)^{2k}}{\sqrt{(2k)!} \beta^{2k - 1}}, \qquad  \tau_{o} := \sum_{k =1}^{\infty} \frac{a_{2k+1} (\lambda - 1)^{2k+1}}{\sqrt{(2k+1)!} \beta^{2k}} .
    \]
   Moreover, the spectrum of $\bK_0 (f - a_1 h_1)$ contains a spurious outlier eigenvalue with eigenvector $\ones_p$ if $|a_2|$ is sufficiently large (see \cite{fan2019spectral} and \cite{liao2020sparse}). 
\end{remark}

\subsection{Detection, Support Recovery, and Optimal Kernels}
\label{sec:2.3}
This section investigates the statistical implications of the results in Sections \ref{sec:poly}  and \ref{sec:Nonpoly}. 

First, we introduce a hypothesis test based on the maximum eigenvalue $\lambda_1$ of $\bK(f)$ to detect the presence of a low-rank component in $\bSigma$ (testing the hypothesis $\lambda = 1$ versus $\lambda \geq \lambda_0$). For polynomial kernels, if a spike is detected, appropriate hard-thresholding of the corresponding eigenvector $\bu_1$ recovers the support of $\bv$ {\it exactly}, with high probability. Importantly, this test is fully powerful if $\psi'(s_+) > 0$---as shown in Theorems \ref{thrm:poly2} and \ref{thrm:B}, $\psi'(s_+) > 0$ ensures that $\lambda_1$ is an outlier eigenvalue and that the cosine similarity between $\bu_1$ and $\bv$ is non-trivial. Thus, $\lambda_1, \bu_1$, and support recovery undergo an {\it identical} phase transition. This support recovery phase transition was empirically observed by Krauthgamer et al.\ (Figure 3 of \cite{Krauthgamer}), though it has not been rigorously established. 

Second, we prove the existence of and characterize an optimal kernel  $f^*(\hspace{.05em}\cdot\,;\gamma,\beta)$ with the broadest recovery region.
Third, we demonstrate that as $\beta \to \infty$, GCT fails below the BBP transition. This supports the conjecture that if $m \gg \sqrt{n}$, the detection threshold of PCA is optimal among polynomial-time algorithms---see the discussion in Section \ref{sec:1}. 
Proofs are deferred to Section \ref{sec:2.3proofs}.


Let $L^2_o(\phi) \coloneqq \{f: \|f\|_\phi^2 < \infty, f(x) = -f(-x)\}$. 
In light of Theorems \ref{thrm:poly2} and \ref{thrm:B}, we define a region $\cR$ in which GCT is informative, 
\begin{align} \cR \coloneqq \Big\{(f,\gamma,\beta,\lambda) :  \psi'(s_+) > 0,  f \in L_o^2(\phi), 
\gamma \in (0,\infty), \beta \in (0,\infty) , \lambda \geq 1 \Big\} , \end{align}
and  $\cR(f,\gamma,\beta, \cdot) \coloneqq  \{\lambda : (f, \gamma, \beta, \lambda) \in \cR \}$. 
Within $\cR(f,\gamma,\beta,\cdot)$, the following test consistently detects the presence of a spike:
\begin{corollary} \label{cor:detect}
   Let $f\in L_o^2(\phi)$ satisfy the assumptions of Theorem \ref{thrm:B} and $\lambda_0 \in \cR(f,\gamma,\beta,\cdot)$. 
   Consider the test of $H_0:\lambda = 1$ versus $H_1$: $\lambda \geq\lambda_0$ which rejects $H_0$ if 
    \[
    \lambda_1 > \lambda_+ + \varepsilon ,
    \]
    where $\varepsilon\in (0,\psi(s_+)-\lambda_+)$. This test is asymptotically correct and fully powerful.
\end{corollary}

As a consequence of (\ref{asdfzxcv3}), which states that the coordinates of $\bu_1$ concentrate tightly around those of $\theta(s_+)$ when $\psi'(s_+) > 0$, hard thresholding of $\bu_1$ recovers the support of $\bv$ exactly. The thresholding level is independent of $\lambda$ and $\beta$, which are generally unknown: 
\begin{corollary} \label{thrm:SR}

Let $f\in L_o^2(\phi)$ be a polynomial, $\lambda\in \cR(f,\gamma,\beta,\cdot)$, $\varepsilon \in (1/4,1/2)$, and define the estimator
\[ \widehat \bv \coloneqq \frac{\mathrm{sign}((\eta_h(\bu_1, n^{-\varepsilon}))}{\big\|\mathrm{sign}((\eta_h(\bu_1, n^{-\varepsilon}))\big\|_2} . 
\] For any $D > 0$, there exists $n_{\varepsilon, D} \in \mathbb{N}$ such that 
\[
\P( \widehat \bv = \bv)\geq 1 - n^{-D},  \qquad \qquad  \forall n \geq n_{\varepsilon, D} .  \]

\end{corollary} 

 Corollaries \ref{cor:detect} and \ref{thrm:SR} follow directly from Theorems \ref{thrm:poly2} and \ref{thrm:B}. In Corollary \ref{thrm:SR}, we assume that $f$ is polynomial, since we only establish entrywise eigenvector bounds for polynomial kernels. While we believe this assumption is not strictly necessary, we leave this technicality for future work. Notwithstanding, recovery throughout $\cR$ is \emph{formally} possible by polynomial-kernel GCT:
\begin{lemma}\label{lem:cR-poly-approx}
    For $f\in L^2_o(\phi)$, let $f_\ell \coloneqq \sum_{k=1}^\ell a_k h_k$ denote its degree-$\ell$ Hermite approximation. If $\lambda\in \cR(f,\gamma,\beta,\cdot)$, there exists $L \in \mathbb{N}$ such that  $\lambda\in \cR(f_\ell,\gamma,\beta,\cdot)$ for all $ \ell \geq L$. 
\end{lemma}

For many kernels of interest, the recovery region admits a simple form. Specifically, if $\tau(f, \beta,\lambda)$ is non-decreasing in $\lambda$, a ``standard'' phase transition occurs: GCT is informative for signal strengths exceeding a critical value  $\lambda_*(f,\gamma,\beta)$.  
\begin{lemma} \label{prop:non_dec}
If $\tau(f, \beta,\lambda)$ is non-decreasing in $\lambda$, then there exists $\lambda_*(f,\gamma,\beta) > 1$ such that   \[ \cR(f,\gamma,\beta,\cdot) = (\lambda_*(f,\gamma,\beta),\infty) . \] 
\end{lemma}

\noindent  Lemma \ref{prop:non_dec} applies to kernels with non-negative Hermite coefficients, as well as to the soft thresholding operator $\eta_s(\cdot,t)$, which has coefficients $a_0 = 0$, $a_1 = 1 - \mathrm{erf}(t/\sqrt{2})$, and 
    \[ \hspace{3.3cm}
    a_k = \begin{dcases}
        \sqrt{\frac{2}{k(k-1)\pi}}\hspace{.05em}e^{-t^2/2} \hspace{.05em} h_{k-2}(t) & k \text{ odd},\\
        0 & k \text{ even},  
    \end{dcases} \hspace{2.3cm} k > 1 .
    \]



We next prove there exists an optimal kernel  $f^*(\hspace{.05em} \cdot \, ; \gamma,\beta)$ with the broadest recovery region.
Let $g(\hspace{.05em} \cdot \, ; \gamma,\beta) \in L_o^2(\phi)$ denote the kernel function with Hermite coefficients
\[ 
    \big\langle g(\hspace{.05em} \cdot \, ; \gamma,\beta), h_k \big\rangle_\phi = \begin{dcases}
        \frac{(\lambda_*(\gamma,\beta)-1)^k}{\sqrt{k!}\beta^{k}} & k \text{ odd}, \hspace{.1em} k > 2\\
        0 & k \text{ otherwise}. 
    \end{dcases} 
\]
From the exponential generating function of the Hermite polynomials, 
\[
e^{t x - t^2/2} = \sum_{k=0}^\infty \frac{t^k h_k(x) }{\sqrt{k!}},
\]
it follows that 
\begin{equation*}
    g(x ; \gamma,\beta) =  e^{-t^2/2}\sinh(t x) - t x,\qquad t := \frac{\lambda_*(\gamma,\beta)-1}{\beta}. 
\end{equation*}

\begin{corollary}\label{cor:opt_transform}
    There exists an optimal kernel function $f^*(\hspace{.05em} \cdot \, ; \gamma,\beta)$ satisfying the conditions of Theorem \ref{thrm:A} and a critical value $\lambda_*(\gamma,\beta)$ such that
    \[  \bigcup_{f \in L_o^2(\phi)} \cR(f,\gamma,\beta,\cdot) = \cR\big(f^*(\hspace{.05em} \cdot \, ; \gamma,\beta),\gamma,\beta,\cdot\big) = (\lambda_*(\gamma,\beta), \infty) . \]
    The kernel $f^*(\hspace{.05em} \cdot \, ; \gamma,\beta)$ is of the form
    \begin{align} \label{eq:optf*}
    f^*(x ; \gamma,\beta) = a_1^*(\gamma,\beta) x + \beta g(x ; \gamma, \beta) ,
    \end{align}
    {where $a_1^*(\gamma,\beta)$ denotes the coefficient of $h_1$ in the Hermite expansion of $f^*(\cdot;\gamma,\beta)$.}
\end{corollary}
\noindent   
In Section \ref{sec:2.4}, we  numerically calculate $ \lambda_*(\gamma,\beta)$. {For fixed $\gamma$ and $\beta$, Corollary \ref{cor:opt_transform}  reduces this problem to a one-dimensional search for $a_1^*(\gamma,\beta)$ in  (\ref{eq:optf*}). For each candidate value of $a_1$ in a grid search,  we compute the minimum  value of $\lambda$ such that $\psi'(s_+) = 0$. }

Our final result demonstrates that as $\beta\to \infty$, GCT fails  below the BBP transition. This supports the conjecture that if $m \gg \sqrt{n}$, the detection threshold of PCA is optimal among polynomial-time algorithms. 
\begin{corollary} \label{prop_betainf}
    For any $\gamma\in (0,\infty)$, the function $\beta\mapsto \lambda_*(\gamma,\beta)$ is increasing and
    \begin{equation}
        \lim_{\beta\to\infty} \lambda_*(\gamma,\beta) = 1+\sqrt{\gamma} .
    \end{equation}
\end{corollary}

\section{Numerical Calculations and Simulations}\label{sec:2.4}

This section presents numerical calculations and simulations.

Figure \ref{fig:soft2} compares GCT with the kernel $\eta_s(\cdot,2)$ (soft thresholding) in blue to standard PCA in orange in terms of cosine similarity, illustrating the phase transition of Theorem \ref{thrm:B} and the BBP transition. As $\beta$ decreases, the performance gap between GCT and PCA increases. 
{Green markers indicate an empirical support recovery score for GCT, which we measure by
\begin{align} \label{asdfqwer1}
    \frac{1}{m} \Big( 
    \big|\mathrm{supp}(\widehat \bv) \cap \mathrm{supp}(\bv)\big|
    -
    \big|\mathrm{supp}(\widehat \bv) \cap ([p] \setminus \mathrm{supp}(\bv))\big|
    \Big), 
\end{align}
where $\widehat \bv$ is defined in Corollary \ref{thrm:SR}. 
Empirically, GCT recovers the support of $\bv$ with high probability above the detection threshold.  For example, in the left plot, the average recovery score is $99.84\%$ at $\lambda = 1.65$. For $\lambda \geq 1.75$, GCT recovered $\bv$ exactly in all simulations.
}

Figure \ref{fig:opt_pts}  depicts the optimal phase transitions of GCT  (right) and soft thresholding
(left), for $\gamma \in \{.5,1,1.5\}$ and $\beta \in [.1,2.5]$. From Corollary \ref{cor:opt_transform}, recall the optimal transition of GCT is
\[
\lambda_*(\gamma,\beta) = \inf \bigcup_{f \in L_o^2(\phi)} \cR(f,\gamma,\beta,\cdot) .
\]
We define the optimal transition of soft thresholding analogously: 
\[
\lambda_{s,*}(\gamma,\beta) \coloneqq \inf \bigcup_{t \geq 0 } \cR(\eta_s(\cdot, t),\gamma,\beta,\cdot)  . \vspace{-.2cm}
\]
Similarly to Corollary \ref{cor:opt_transform},
the infimum is achieved for a specific thresholding level $t_*(\gamma,\beta)$. 

Surprisingly, soft thresholding is close to optimal: the left and right plots of Figure \ref{fig:opt_pts} are visually nearly indistinguishable, and the discrepancy between the plotted curves is less than $0.05$.\footnote{We compute $\lambda_*(\gamma,\beta)$ as described after Corollary \ref{cor:opt_transform}. To compute  $\lambda_{s,*}(\gamma,\beta)$, we perform a grid search over the thresholding level $t$. For each level, we compute the minimum value of  $\lambda$ such that 
$\psi'(s_+) = 0$. In both cases, we approximate $s_+$ by replacing $\tau(f,\beta,\lambda)$ with $\tau(f_{21},\beta,\lambda)$, where $f_{21}$ is the degree-${21}$ Hermite approximation of $f$re.  
}

Based on these calculations, we advocate soft thresholding; it is likely more robust to model misspecification than $f^*$, which is tailored to the Rademacher spike structure assumed in  (\ref{eq:v_form}). Figure \ref{fig:opt_pts} also  illustrates  Corollary \ref{prop_betainf}: as $\beta \rightarrow \infty$, $\lambda_{*}(\gamma,\beta)$ converges to the BBP transition. {The left panel of Figure \ref{fig:adapt} plots the optimal nonlinearity $f^*(\cdot;\gamma,\beta)$ for $\gamma =1$ and several values of $\beta$. 
For $\beta \in \{0.1,0.3,1\}$, the optimal kernel is nearly zero over the interval $[-1,1]$, helping explain the strong performance of soft thresholding. 
For $\beta=3$, the optimal kernel is approximately linear near the origin.  
} 

As the optimal soft thresholding level depends on $\lambda$ and $\beta$, which are generally unknown, we propose to adaptively select the threshold to maximize the normalized spectral gap $(\lambda_1 - \lambda_2)/\lambda_2$. 
This procedure is motivated by Corollaries \ref{cor:detect} and \ref{thrm:SR}: since the presence of an outlier eigenvalue indicates that $\psi'(s_+) > 0$, it is natural to choose the threshold that produces the most distinct outlier. This procedure compares well to optimal fixed-level thresholding, GCT with the kernel $\eta_s(\cdot, t_*(\gamma,\beta))$---see  the right panel of Figure \ref{fig:adapt}. 

Empirically, the phenomena we uncover are not restricted to 
$\sqrt{m} v_1, \ldots, \sqrt{m} v_m \in \{-1,+1\}$ as in (\ref{eq:v_form}). In Figure \ref{fig:unif}, we generate $\bxi \in \mathbb{R}^m$ according to \begin{align} \label{asdfqwer}
    \xi_1, \ldots, \xi_m \stackrel{i.i.d.}{\sim} \mathrm{unif}\big([-2,-1]\cup[1,2]\big) , \end{align}
    and set $\bv_{[1:m]} = \bxi/\|\bxi\|_2$.
The left plot depicts the cosine similarity of GCT with the kernel $\eta_s(\cdot,2)$, which experiences a phase transition. The right plot depicts support recovery, measured as follows: 

That is, (\ref{asdfqwer1}) counts the number of correctly identified entries minus the number of false positives, normalized by $m$. The signal detection and recovery thresholds seem aligned, as we expect from Section \ref{sec:2.3}. 


\begin{figure}[h]
\begin{center} 
    \includegraphics[scale=.51]{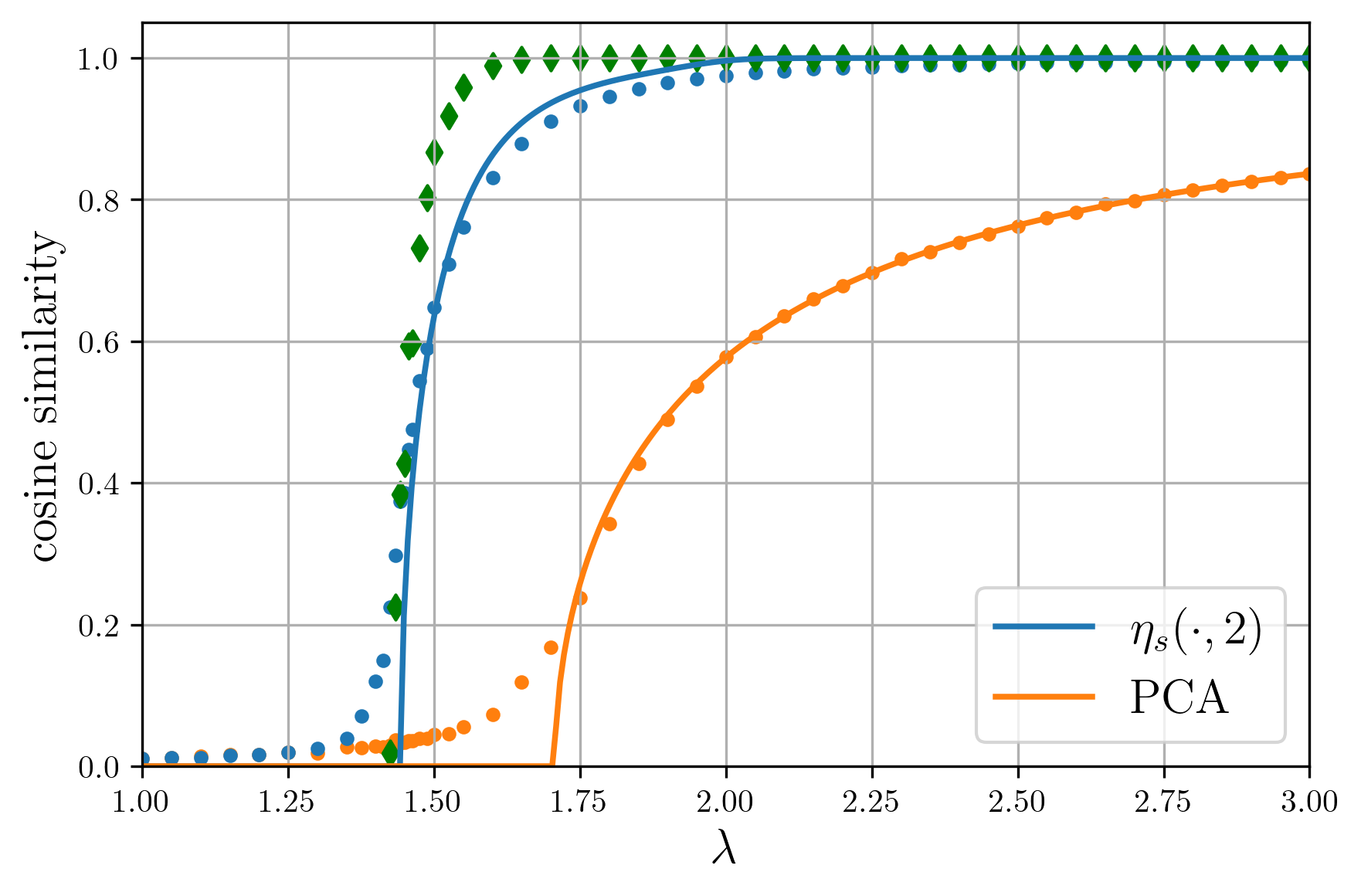} \includegraphics[scale=.51]{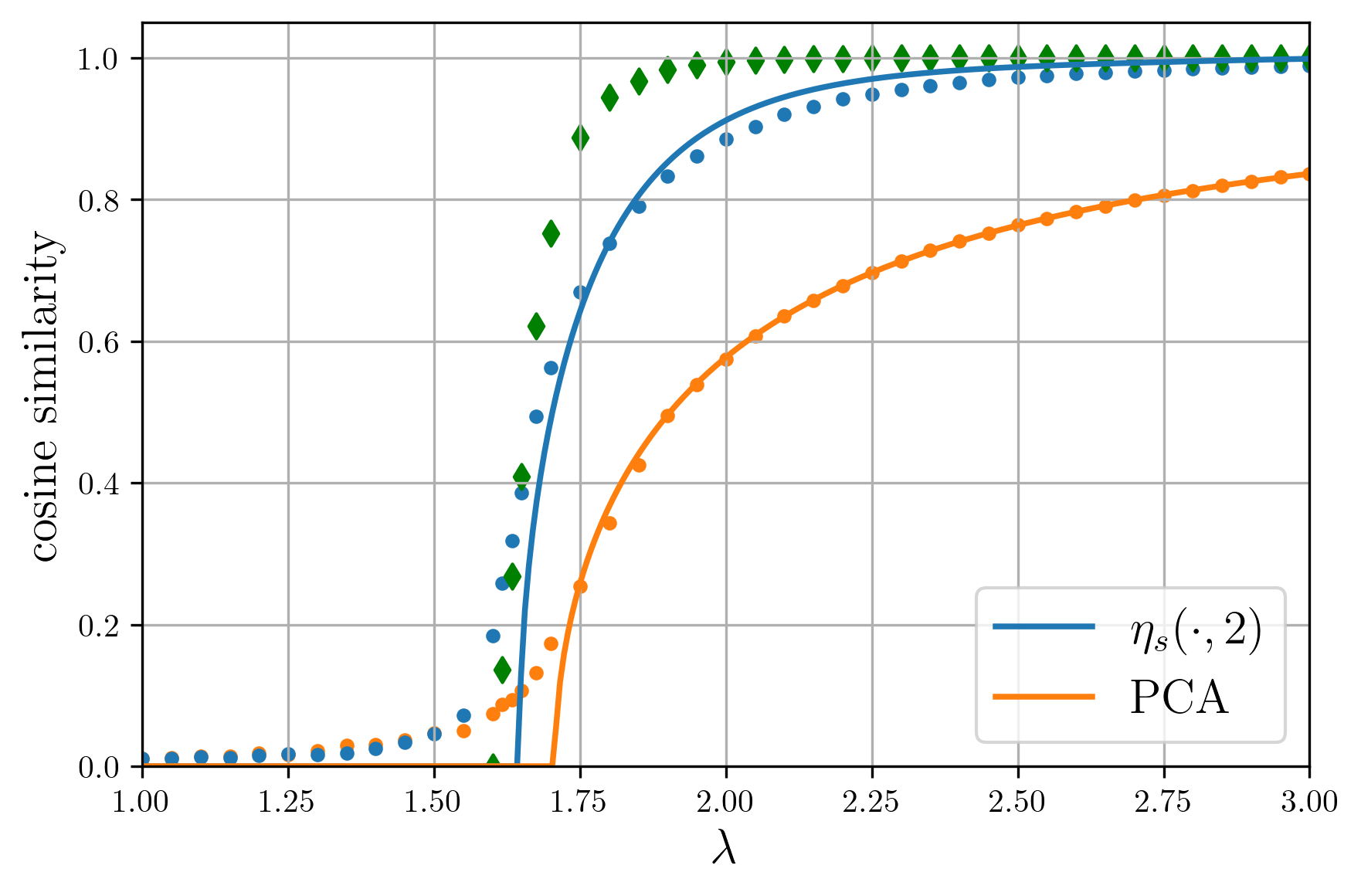}
\end{center} \vspace{-.5cm} 
\caption{Cosine similarities for GCT with the kernel $\eta_s(\cdot, 2)$ (blue) and standard PCA (orange). There is close agreement between Theorem \ref{thrm:B} (solid lines) and simulations (points, each representing the average over 100 simulations). {Green markers indicate the support recovery score of GCT, defined in (\ref{asdfqwer1})} On the left, $n = $  10,000\unskip, $p = $ 5,000\unskip, and $m = 25$, so $\beta = 1/4$.  On the right, $n = $ 10,000\unskip, $p =$ 5,000\unskip, and $m = 50$, so $\beta = 1/2$. Observe that the phase transition of GCT decreases as $\beta$ decreases.} \label{fig:soft2}
\end{figure}


\begin{figure}[H]
\begin{center} 
    \includegraphics[scale=.505]{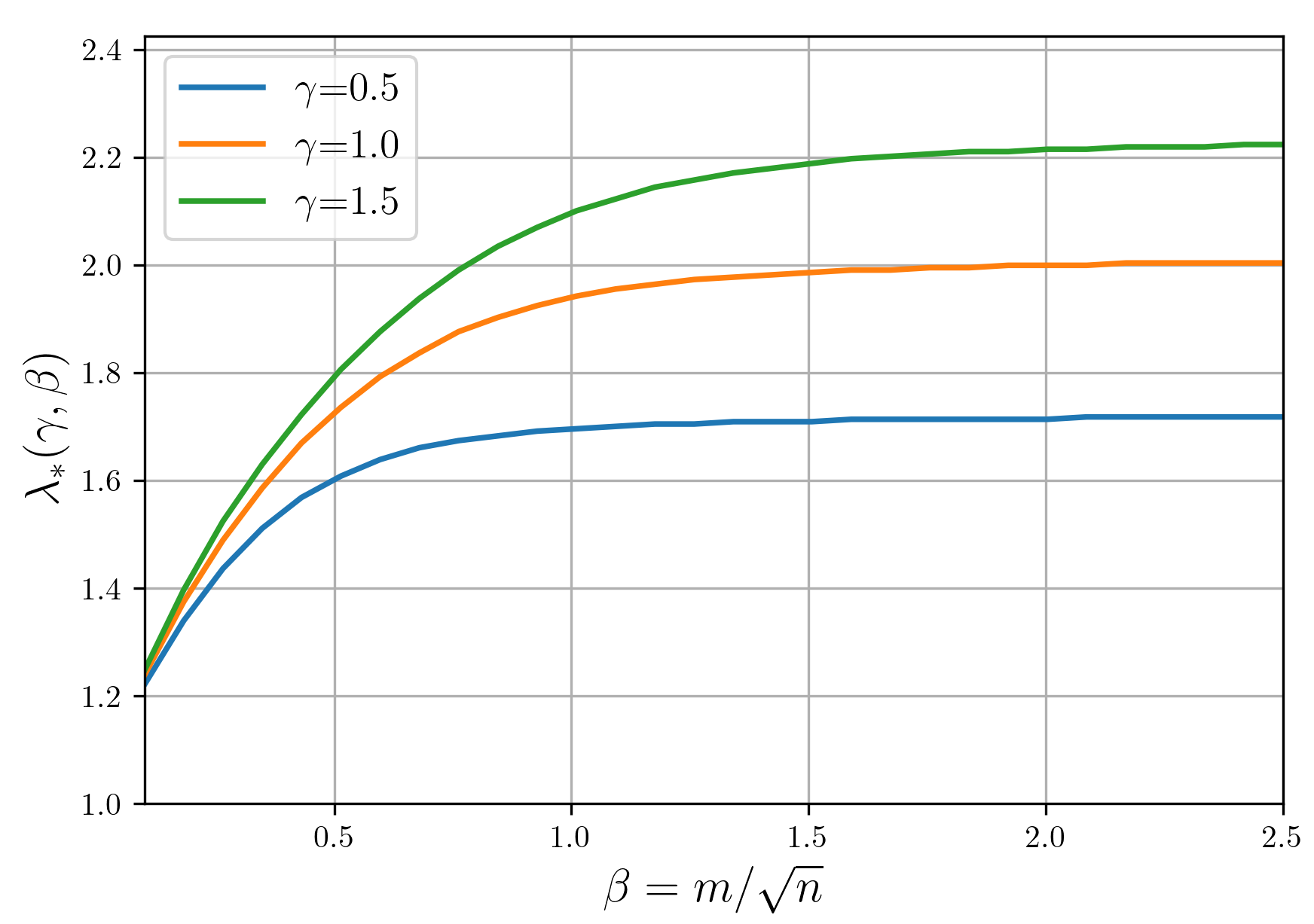} \includegraphics[scale=.505]{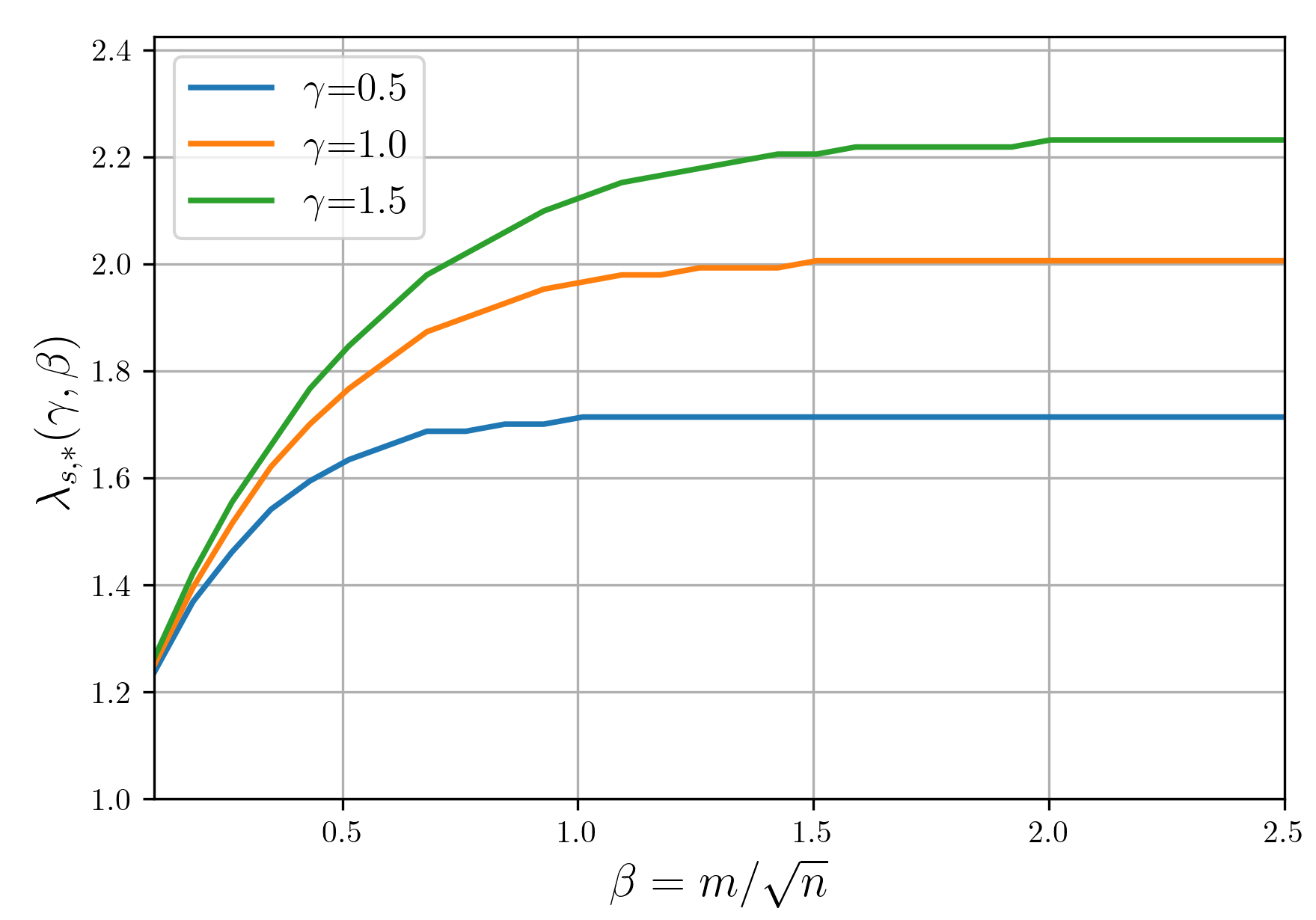}
\end{center} \vspace{-.5cm} \caption{The optimal phase transition of GCT (left) and soft thresholding (right), for $\gamma \in \{.5,1,1.5\}$ and $\beta \in [.1,2.5]$. Notice that (1) $\lambda_{s,*}(\gamma,\beta)$ and $\lambda_{*}(\gamma,\beta)$ are visually nearly indistinguishable, suggesting that soft thresholding is close to optimal, and (2)  $\lambda_{*}(\gamma,\beta) \rightarrow 1+ \sqrt{\gamma}$ as $\beta \rightarrow \infty$, supporting Lemma \ref{prop_betainf}. } \label{fig:opt_pts} 
\end{figure}

\begin{figure}[H]
\begin{center} 
\includegraphics[scale=.48]{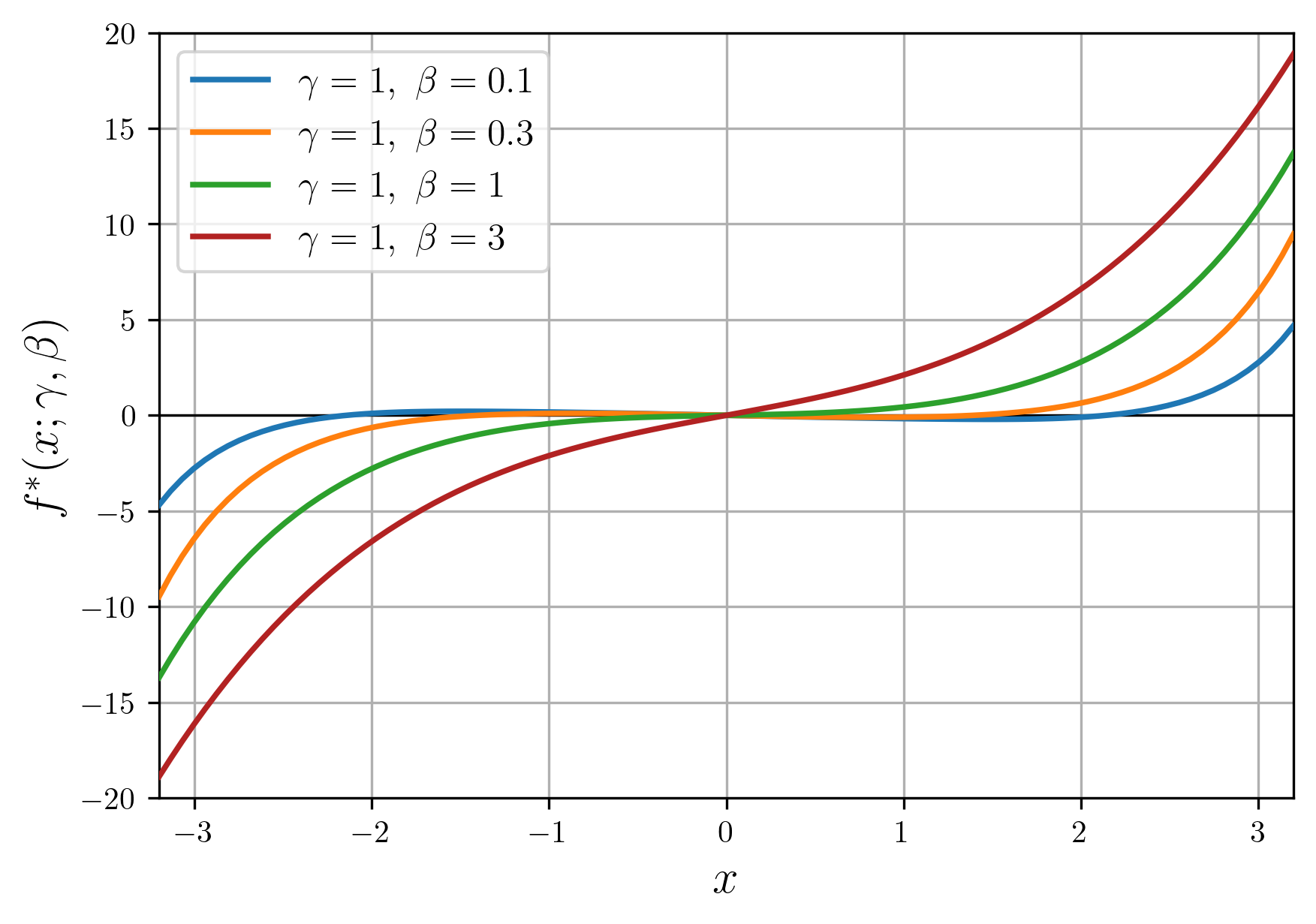}  
    \includegraphics[scale=.51]{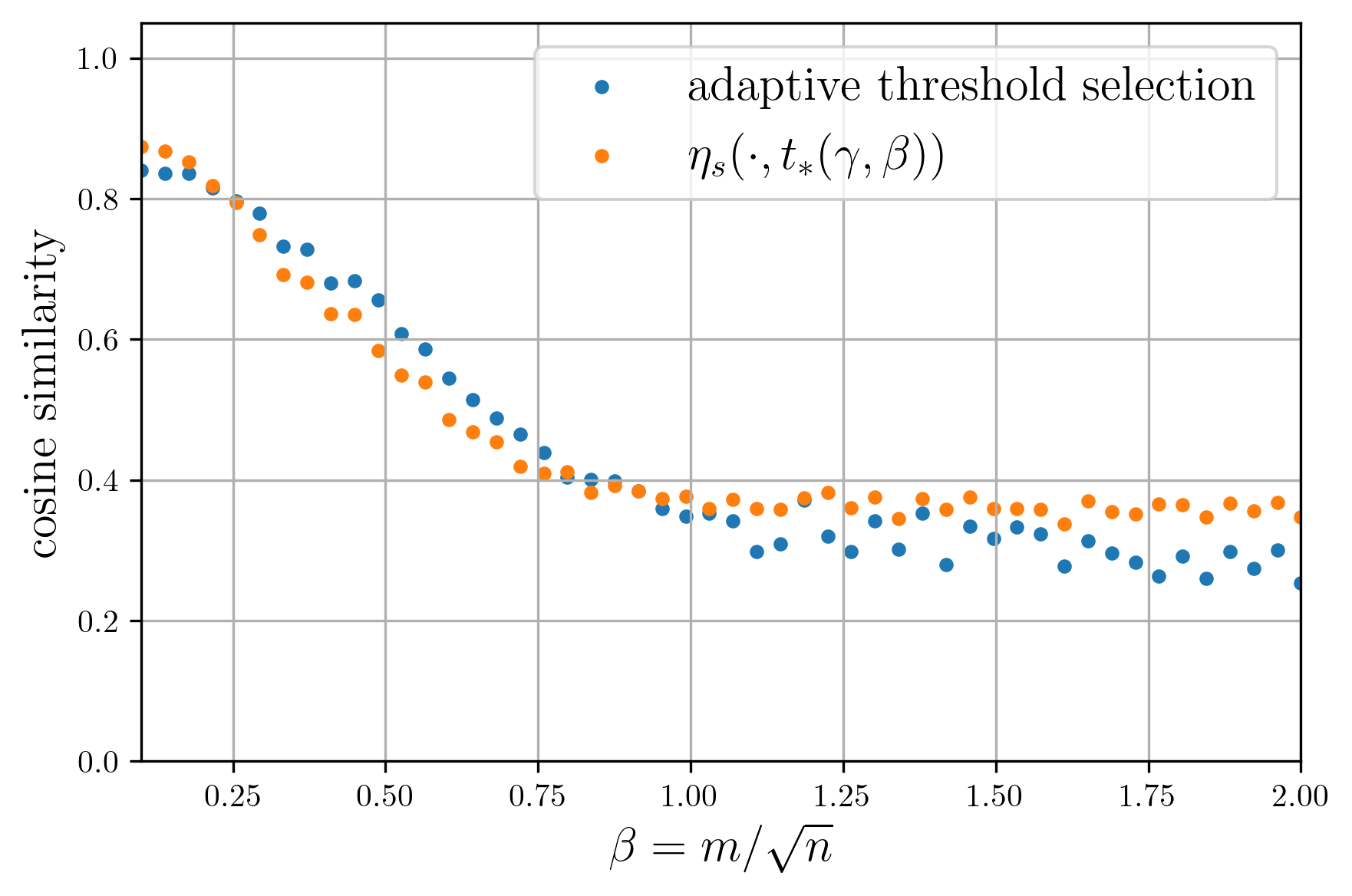} 
\end{center} \vspace{-.5cm} \caption{Optimal nonlinearities $f^*(\cdot;\gamma,\beta)$ for $\gamma=1$ and several values of $\beta$ (left), and cosine similarities for soft thresholding with adaptive threshold selection (right). 
For $\beta \in \{0.1,0.3,1\}$, the optimal kernel is nearly zero over the interval $[-2,2]$, helping explain the strong performance of soft thresholding. For $\beta = 3$, the optimal kernel is approximately linear near the origin. In the right panel, we compare  adaptive threshold selection (blue) with the optimal fixed threshold $t_*(\gamma,\beta)$ (orange). Here $n=10{,}000$ and $p=5{,}000$, so $\gamma=0.5$. Each point averages 50 simulations. Adaptive thresholding outperforms the optimal fixed level for intermediate $\beta$.
} \label{fig:adapt}
\end{figure}

\begin{figure}[H]
\begin{center} 
    \includegraphics[scale=.51]{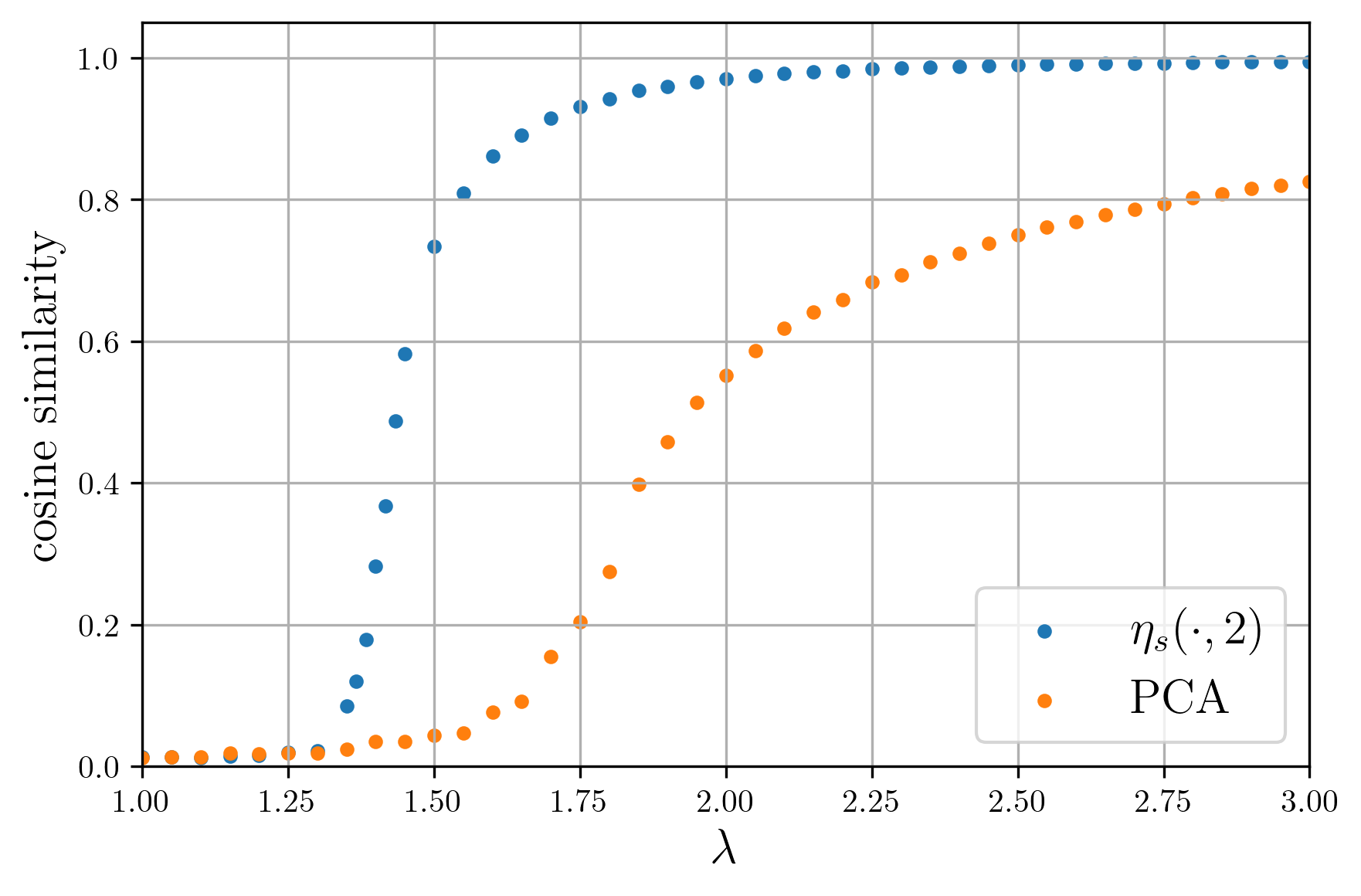} \includegraphics[scale=.51]{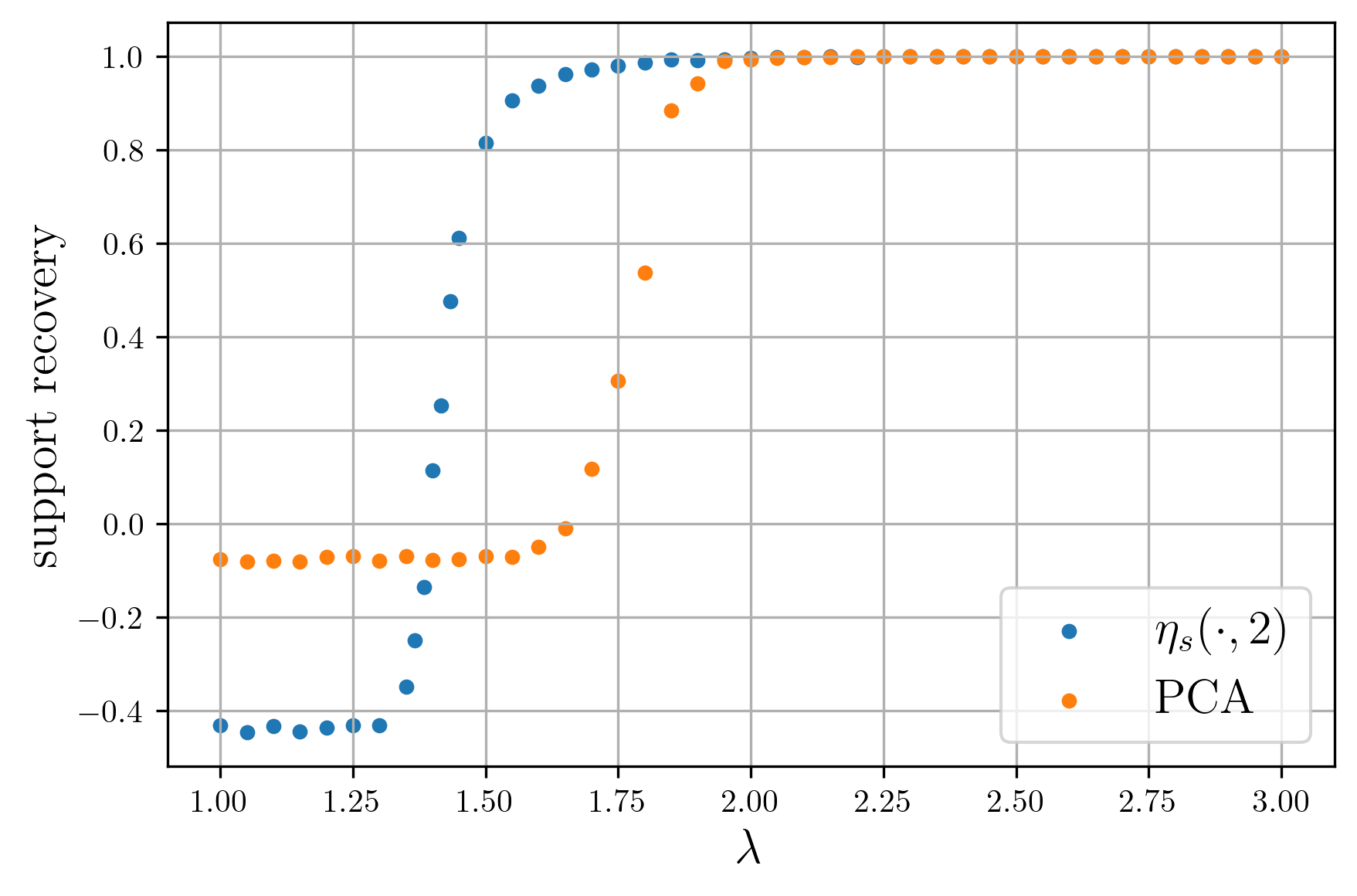}
\end{center} \vspace{-.5cm} \caption{Simulations of GCT with the kernel $\eta(\cdot, 2)$  (blue) and standard PCA (orange). The non-zero elements of $\bv$ are generated from a uniform prior (\ref{asdfqwer}).  The left plot displays cosine similarities; on the right, we plot support recovery, measured according to (\ref{asdfqwer1}). Empirically, the signal detection and recovery thresholds seem aligned. Here, $n = $ 10,000\unskip, $p = $ 5,000\unskip, and $m = 25$, so $\beta = 1/4$. Each point represents the average of 50 simulations.
} \label{fig:unif}
\end{figure}


\section{Deterministic Equivalents} \label{sec:iso_law}

In this section, we state our main technical contribution. Let $\bR_0(f,z)$ denote the resolvents of $\bK_0(f)$, defined as
\begin{align*} 
 \bR_0(f,z) \coloneqq (\bK_0(f) - z \bI_p)^{-1} .
\end{align*}
The resolvent is holomorphic at all points $z \in \mathbb{C}$ not equal to an eigenvalue of $\bK_0(f)$  (Theorem 1.5 of \cite{kato}). 
Recalling the Stieltjes transform $s(z)$ of $\mu$, we define the related quantities
\begin{align*}
& \breve s(z) \coloneqq \frac{s(z)}{1 + a_1 \gamma s(z)} , &   \accentset{\circ}{s}(z) \coloneqq     \breve s(z) ( 1 + \gamma - a_1 \gamma  \breve s(z) ).
\end{align*}

\begin{theorem} \label{lem:quad_forms} Let  $f$ be an odd polynomial and (\ref{asdf}) hold. For any $z \in \mathbb{C}^+$ and  deterministic  vectors $\bu, \bw \in \mathbb{S}^{p-1}$,
\begin{align}
   &\bu^\top \spa \bR_0(f,z)\bw - \langle \bu, \bw \rangle s(z) = O_\prec(n^{-1/2}) ,   \label{nhy1}\\
    &\bu^\top \spa \bS\bR_0(f,z) \bw - \langle \bu, \bw \rangle  \breve s(z) = O_\prec(n^{-1/2}) ,  \label{nhy2} \\ & \bu^\top \spa \bS \bR_0(f,z) \bS\bw  - \langle \bu, \bw \rangle \accentset{\circ}{s}(z) = O_\prec(n^{-1/2})  . \label{nhy3}
\end{align}
Moreover, the convergence is uniform in $z$ on compact subsets $\mathcal{C} \subset \mathbb{C}$ such that $\inf_{z \in \mathcal{C}} \mathrm{Re}(z) > \lambda_+$. 
\end{theorem}

\begin{theorem} \label{lem:quad_formsX} 
    Let  $f$ be an odd polynomial and (\ref{asdf}) hold.
     For any $z \in \mathbb{C}^+$ and  deterministic  vectors $\bu, \bw \in \mathbb{S}^{p-1}$ such that $\|\bu\|_0 \vee \|\bw\|_0 \lesssim m$,
\begin{align} 
   &\bu^\top  \spa \bR_0(f,z) (\bK(f) - \bA(f) ) \bR_0(f,z)\bw  = O_\prec(n^{-1/2}) , \\
    &\bu^\top \spa \bS\bR_0(f,z)  (\bK(f) - \bA(f) )  \bR_0(f,z) \bw = O_\prec(n^{-1/2}) ,   \\ & \bu^\top \spa \bS \bR_0(f,z)  (\bK(f) - \bA(f))  \bR_0(f,z) \bS\bw   = O_\prec(n^{-1/2})  .
\end{align}
Moreover, the convergence is uniform in $z$ on compact subsets $\mathcal{C} \subset \mathbb{C}$ such that $\inf_{z \in \mathcal{C}} \mathrm{Re}(z) > \lambda_+$. 
\end{theorem}

{Proofs  are deferred to Appendix \ref{sec:appendix1}. 
Theorem  \ref{lem:quad_forms} strengthens earlier results of Cheng and Singer \cite{cheng2013spectrum} and Liao, Couillet, and Mahoney \cite{liao2020sparse}. In particular, combining the proof of Theorem~3.4 of Cheng and Singer with Theorem~1.6 of Fan and Montanari \cite{fan2019spectral} yields
\[
\frac{1}{p}\tr \bR_0(f,z)-s(z) = O_{\prec}(n^{-1/2}).
\]
Lemma 1 of Liao et al.\ computes the expectations of quantities appearing in (\ref{nhy1}) and~(\ref{nhy2}).
Neither reference provides bounds for bilinear forms of the resolvent at the rate stated above.

Our proofs are based on the martingale difference argument of Bai, Miao, and Pan \cite{bai2007asymptotics}, who establish concentration for quadratic forms of sample covariance resolvents. This method is standard in random matrix theory  (see Sections 8 and 10 of \cite{bai2007asymptotics}), although its extension to kernel matrices is involved and requires careful control of error terms. Furthermore, we are not aware of any variant of Theorem~\ref{lem:quad_formsX} in the literature. Its proof relies crucially on the sparsity of the test vectors \(\bu,\bw\), and it is the main technical input for the entrywise eigenvector estimate \eqref{asdfzxcv3}.}

\section{Proofs} \label{sec:proofs}
\subsection{Proof of Theorems \ref{thrm:poly1} and \ref{thrm:poly2}}
Throughout this section, we assume $f$ is an odd polynomial, $L \coloneqq \mathrm{deg}(f)$, and (\ref{asdf}) holds.  

We shall show that all terms of (\ref{eq:taylor_f_L_2}) are vanishing in operator norm except those with indices $(\ell, k)$ of the form $(0, k)$ or $(\ell, \ell)$. In contrast, in the setting where $\bv$ is distributed uniformly on $\S^{p-1}$, only those terms with $(\ell, k)$ equal to $(0 ,1)$ or $(\ell, \ell)$ are non negligible.

\begin{proof}[Proof of Theorem \ref{thrm:poly1}]

Rearranging (\ref{eq:taylor_f_L_2}) and using the identity $\bK_0(h_0) = n^{-1/2} (\ones_p \ones_p^\top - \bI_p) $, we obtain 
the following decomposition:
\begin{equation}\label{eq:decompo_K_f_A}
\begin{aligned}
         \bK (f) =&~ a_1 ( \bY - \diag (\bY)) + \sum_{\ell = 3}^L \frac{a_\ell}{\sqrt{\ell!}} n^{(\ell-1)/2}\big( \bX^{\odot \ell} - \diag( \bX^{\odot \ell} ) \big) + \bK_0 (f - a_1h_1) \\
    &~ + \sum_{\ell=1}^L \sum_{k= 1}^{L-\ell} \frac{a_{k+\ell} }{k!}\sqrt{\frac{(k+\ell)!}{\ell!}} \Big( \bK_0 (h_k) \odot (\sqrt{n} \bX)^{\odot \ell} \Big).
\end{aligned}
\end{equation}
Recalling that $\bA(f) \coloneqq a_1 (\bY- \bI_p) + \tau  \cdot \bv \bv^\top  + \bK_0(f - a_1 h_1)$, it suffices to demonstrate that 
\begin{enumerate}
    \item $\|\text{diag}(\bY) - \bI_p\|_2  \prec n^{-1/2}$,
    \item  For $\ell \in \{3, \ldots, L\}$, \[ \Big\|n^{(\ell-1)/2} \big( \bX^{\odot \ell} - \diag( \bX^{\odot \ell} ) \big)  - (\lambda-1)^\ell \beta_n^{-(\ell-1)}  \bv \bv^\top  \Big\|_2 \prec n^{-1/4} ,\] 
    \item For  $\ell \in [L]$ and $k \in [L-\ell]$, \[\big\| \bK_0(h_k) \odot (\sqrt{n} \bX)^{\odot \ell} \big\|_2 \prec n^{-1/4}. \]
    \end{enumerate}
We refer to these bounds as claims (1)--(3).
    
Expanding $\bX = \bSigma^{1/2} \bS \bSigma^{1/2} - \bS$ using $\bSigma^{1/2} = (\sqrt{\lambda} - 1) \bv \bv^\top  + \bI_p$, 
\begin{equation}
    \begin{aligned} \label{r8bc9} 
\bX = &~ (\sqrt{\lambda} - 1)^2 (\bv^\top \spa \bS \bv )\cdot \bv\bv^\top  + (\sqrt{\lambda} - 1) ( \bv \bv^\top \spa \bS + \bS \bv \bv^\top  ) \\
=&~ c \cdot \bv \bv^\top  + (\sqrt{\lambda} - 1) \big(\bv( \bS \bv - \bv)^\top + ( \bS \bv - \bv) \bv^\top  \big),
\end{aligned}
\end{equation}
where $c \coloneqq  (\sqrt{\lambda} - 1)^2 (\bv^\top \spa \bS \bv ) +2 (\sqrt{\lambda} - 1)$. Notice that $
|c - (\lambda -1)| \prec n^{-1/2} $ since $ \bv^\top \spa \bS \bv \sim n^{-1}\chi_n^2$.
Writing the coordinates of $\bw \coloneqq \bS \bv - \bv$ as 
\[
w_j = \sum_{i=1}^p  \bS_{ij} v_i  - v_j = \frac{v_j}{n}\sum_{i = 1}^n (z_{ij}^2 - 1) + \frac{1}{n} \sum_{i=1}^n z_{ij} (\bz_i^\top \bv - z_{ij} v_j) ,
\]
the independence of  $\{z_{ij}\}_{i \in [n]}$ and  $\{\bz_i^\top \bv - z_{ij} v_j\}_{i \in [n]}$ implies $w_j$ is an average of i.i.d.\ sub-exponential variables. Therefore, using Hoeffding's inequality and a union bound, 
\begin{equation}\label{gk91}
\|\bw\|_\infty= O_\prec(n^{-1/2}) . 
\end{equation}
Similarly bounding $\|\text{diag}(\bS) - \bI_p\|_2$ establishes claim (1):
\begin{equation}
\begin{aligned} \label{rkj7}
    \|\text{diag}(\bY) - \bI_p\|_2 &\leq \|\text{diag}(\bS) - \bI_p\|_2 + \|\text{diag}(\bX)\|_2  \\ & \leq \max_{j \in [p]} \big( |n^{-1} \bz_j^\top \bz_j - 1| + c  v_j^2 + 2(\sqrt{\lambda}-1)|v_j w_j| \big)\\ & \prec n^{-1/2}.
\end{aligned}    
\end{equation}

Next, consider 
the expansion 
\begin{align}
\bX^{\odot \ell} = \sum_{\substack{\ell_1, \ell_2, \ell_3 \in \{0,\ldots, L\}  \\ \ell_1 + \ell_2 + \ell_3 = \ell}} {{\ell}\choose{\ell_1,\ell_2,\ell_3}} c^{\ell_1} (\sqrt{\lambda}-1)^{\ell_2 + \ell_3} \big(\bv^{\odot (\ell_1 + \ell_2)} \odot \bw^{\odot \ell_3} \big) \big(\bv^{\odot (\ell_1 + \ell_3)} \odot \bw^{\odot \ell_2} \big)^\top , \label{kj3hg}
\end{align}
the first term of which  is $c^\ell m^{-(\ell-1)} \bv \bv^\top$ (corresponding to $\ell_1 = \ell$, $\ell_2 = \ell_3 = 0$).
Using (\ref{gk91}), for $\ell_1 \in [\ell]$ and $\ell_2 \in \{0\} \cup [\ell]$, we have the bounds
\begin{align*}
& \big\| \bv^{\odot \ell_1} \odot \bw^{\odot \ell_2} \big\|_2 \prec  \frac{1}{m^{(\ell_1-1)/2} \,n^{\ell_2/2}} , & \| \bw^{\odot \ell_1} \|_2 \prec \frac{1}{ n^{(\ell_1-1)/2}},
\end{align*}
implying that terms of  (\ref{kj3hg}) with $\ell_2 + \ell_3 > 0$ are $O_\prec(n^{-\ell/2 + 1/4})$. 
Thus,
\begin{equation}
    \begin{aligned}
     \big\|n^{(\ell-1)/2} \bX^{\odot \ell} - c^\ell \beta_n^{-(\ell-1)}  \bv \bv^\top  \big\|_2 &  \prec n^{-1/4} . 
\end{aligned}
\end{equation}
Claim (2) now follows from $|c - (\lambda -1)| \prec n^{-1/2} $ 
and  $\|\diag(\bX)\|_2 \prec n^{-1/2}$ (shown in (\ref{rkj7})).

Claim (3) is a consequence of Theorem 1.6 of \cite{fan2019spectral}. By (\ref{kj3hg}),  
\begin{align}
\big\| \bK_0(h_k) \odot \bX^{\odot \ell} \big\|_2  & \lesssim  \sum_{\substack{\ell_1, \ell_2, \ell_3 \in [\ell]  \\ \ell_1 + \ell_2 + \ell_3 = \ell}} c^{\ell_1} \Big\| \bK_0 (h_k) \odot \big( \big(\bv^{\odot (\ell_1 + \ell_2)} \odot \bw^{\odot \ell_3} \big) \big(\bv^{\odot (\ell_1 + \ell_3)} \odot \bw^{\odot \ell_2} \big)^\top\big) \Big\|_2  . \label{bcxg8}
\end{align}
Consider terms with $\ell_2, \ell_3 < \ell$. Using the inequality 
\[
\big\|\bA \odot (\bx\by^\top ) \big\|_2 = \big\|\diag(\bx) \bA \, \diag (\by)  \big\|_2 \leq \| \bx\|_\infty \|\by\|_\infty \|\bA\|_2 ,
\]
the sparsity of $\bv$, and  (\ref{gk91}), we obtain
\begin{align*} 
&~ \Big\| \bK_0 (h_k) \odot \big( \big(\bv^{\odot (\ell_1 + \ell_2)} \odot \bw^{\odot \ell_3} \big) \big(\bv^{\odot (\ell_1 + \ell_3)} \odot \bw^{\odot \ell_2} \big)^\top\big) \Big\|_2
\prec \frac{\big\| (\bK_0 (h_k))_{[1:m,1:m]} \big\|_2 }{m^{(\ell + \ell_1)/2} n^{(\ell_2 + \ell_3)/2}}  .
\end{align*}
For the remaining terms in (\ref{bcxg8}), where $\ell_2 = \ell$ or  $\ell_3 = \ell$, we have
\begin{align*} 
     \Big\| \bK_0 (h_k) \odot \big( \bv^{\odot \ell}  \big(\bw^{\odot \ell}  \big)^\top\big) \Big\|_2 = \Big\| \bK_0 (h_k) \odot \big( \bw^{\odot \ell}  \big(\bv^{\odot \ell}  \big)^\top\big) \Big\|_2 \prec \frac{1}{m^{\ell/2}n^{\ell/2}}   \| \bK_0 (h_k) \|_2  .
\end{align*}
Theorem 1.6 of \cite{fan2019spectral} yields
\begin{align} \label{qazwsx9}
    &  \| \bK_0 (h_k) \|_2 \prec\max\bigg(\sqrt{\frac{p}{n}}, \, \frac{p}{n}\bigg) , & \big\| (\bK_0 (h_k))_{[1:m,1:m]} \big\|_2  \prec \max\bigg(\frac{ \sqrt{\beta_n}}{n^{1/4}}, \, \frac{\beta_n}{\sqrt{n}}\bigg)  . 
\end{align}
Thus, by (\ref{asdf}), $\big\| \bK_0(h_k) \odot (\sqrt{n} \bX)^{\odot \ell} \big\|_2 \prec n^{-1/4}$, completing the proof.

\end{proof}

Let $\bR(f,z)$ denote the resolvent of $\bK(f)$, defined analogously to $\bR_0(f, z)$:
\begin{align*}
& \bR(f,z) \coloneqq (\bK(f) - z \bI_p)^{-1}  .
\end{align*}

\begin{lemma}\label{lem:max_eig2}
For any $\varepsilon, D > 0$, there exists $n_{\varepsilon, D} \in \mathbb{N}$ such that
\begin{align*}
     \hspace{4.3cm}\P\big(\|\bK_0(f)\|_2 > (1+\varepsilon) \lambda_+\big) \leq n^{-D} , \qquad \qquad \hspace{1cm} \forall \, n \geq n_{\varepsilon, D} . \end{align*} 
\end{lemma}
\noindent The proof of Lemma \ref{lem:max_eig2}, which is a slight extension of Section 6 in \cite{fan2019spectral}, is given in Appendix \ref{sec:C}.

Drawing upon \cite{benaych2011eigenvalues, bloe2014isotropic, feldman23, liao2020sparse}, we prove Theorem \ref{thrm:poly2} in two stages: eigenvalue results,
\begin{equation} \label{q6j}
    \begin{aligned} 
        & |\lambda_1 - \psi(s_+) | \prec n^{-1/2}, \hspace{1.5cm}& \psi'(s_+) > 0 , \\
        & \lambda_1 \xrightarrow{a.s.} \lambda_+, \hspace{1.5cm}& \psi'(s_+) \leq 0 ,
    \end{aligned}
\end{equation}
and eigenvector results, 
    \begin{align}
        & \langle \bu_1, \bv \rangle^2 = \theta^2(s_+)  + O_\prec(n^{-1/2}), & \psi'(s_+) > 0 ,\label{q7j}  \\
&         u_{1i} = \theta(s_+) v_i + O_\prec(n^{-1/2}), & \psi'(s_+) > 0 ,\label{q9j}  \\
        & \langle \bu_1, \bv \rangle \xrightarrow{a.s.} 0,& \psi'(s_+) \leq 0 .\label{q8j}
    \end{align}
For brevity, we often suppress the arguments of matrices such as $\bK_0(f)$ and $\bR_0(f,z)$.  

\begin{proof}[Proof of (\ref{q6j})] 

Assume $\psi'(s_+) > 0$ and define the contours
\[ \mathcal{C}_n \coloneqq \big\{z \in \mathbb{C} : |z - \psi(s_+)| = \delta_n \big\} ,\]
where  (1) $\delta_n \rightarrow 0 $ and (2) there exists $\varepsilon > 0$ such that $\delta_n \geq n^{-1/2 + \varepsilon}$. We shall prove  $\psi(s_+) > \lambda_+$ and that, with probability at least $1 - n^{-D}$ (for any $D > 0$), $\mathcal{C}_n$ encircles exactly one eigenvalue of $\bK$.  Moreover, if $\mathcal{C}$ is an an arbitrary, bounded contour satisfying $\inf_{z \in \mathcal{C}}\mathrm{Re}(z) > \lambda_+$ and $\mathcal{C} \hspace{.1em} \cap \hspace{.1em} \mathcal{C}_n = \emptyset$, an analogous argument shows that, with probability at least $1-n^{-D}$, $\mathcal{C}$ contains no eigenvalues of $\bK$. It follows that $|\lambda_1 - \psi(s_+)| \prec n^{-1/2}$. The proof that $\psi'(s_+) \leq 0$ implies $\lambda_1 \xrightarrow{a.s.}\lambda_+$ is similar and therefore omitted.  

From Corollary \ref{cor:stj_trans} and the discussion on page 19 of \cite{liao2020sparse},
$\psi(s_+) > \lambda_+$ if and only if  $\psi'(s_+) > 0$ (which holds by assumption) and  $s_+ \in (-1/(a_1\gamma),0)$. Since $s_+$  is negative and \begin{align} \label{9br7}
\frac{ \partial s_+}{\partial \lambda} = \frac{1}{2  \tau \gamma } \Big( -1 + \frac{a_1(\lambda+\gamma-1)+\tau}{\sqrt{(a_1(\lambda+\gamma-1) + \tau)^2 - 4 a_1 \tau \gamma}} \Big) > 0 , 
\end{align}
$s_+$ is increasing in $\lambda$ on $[1,\infty)$.\footnote{The derivative in \eqref{9br7} is symbolic and ignores the dependence of $\tau$ on $\lambda$.} Moreover, since $\tau \neq 0$, we necessarily have $\lambda > 1$. Therefore,
\begin{align*}
    s_+ & > \frac{-a_1 \gamma - \tau + \sqrt{(a_1 \gamma + \tau)^2 - 4 a_1 \tau \gamma}}{2 a_1 \tau \gamma} \\
    & = \frac{-a_1 \gamma - \tau + |a_1 \gamma - \tau|}{2 a_1 \tau \gamma} \\ & \geq - \frac{1}{a_1 \gamma}. 
\end{align*}

By Theorem \ref{thrm:poly1}, we may write $\bK(f) = \bV \spc \bLambda \bV^\top \spa + \bK_0(f) + {\bf \Delta}$, where 
\begin{align*} 
     \bV & \coloneqq \begin{bmatrix}
    \bv & \bS \bv
\end{bmatrix} , \\  {\bf \Lambda} &  \coloneqq  \begin{bmatrix}
    a_1(\sqrt{\lambda}-1)^2 + \tau & a_1(\sqrt{\lambda}-1) \\ a_1(\sqrt{\lambda}-1) & 0 
\end{bmatrix} , \\
 \bDelta & \coloneqq \bK - \bA + (c-(\lambda-1)) \bv \bv^\top  ,
\end{align*} 
and $\|\bDelta\|_2 \prec n^{-1/4}$ (recall $c$ is defined following (\ref{r8bc9})). Since $\psi(s_+)>\lambda_+$, Lemma~\ref{lem:max_eig2} implies that $\bR_0$ and $\bR_\bDelta \coloneqq \big(\bK_0+\bDelta-z\bI_p\big)^{-1}$ are holomorphic on and within $\mathcal{C}_n$ with probability $1-n^{-D}$ and
\begin{align} \label{poiu}
\sup_{z\in\mathcal{C}_n}\bigl(\|\bR_0\|_2 \vee \|\bR_\bDelta\|_2\bigr)\prec 1.
\end{align}
Here, a matrix-valued function is holomorphic if each entry is holomorphic, and it has a pole at $z$ if at least one entry has a pole at $z$. By the Woodbury identity,
\begin{align} \label{9br3}
\bR
=
\bR_\bDelta
-
\bR_\bDelta \bV
\bigl(\bLambda^{-1}+\bV^\top \bR_\bDelta \bV\bigr)^{-1}
\bV^\top \bR_\bDelta .
\end{align}
It follows that any pole of $\bR$ within $\mathcal{C}_n$ must be a zero of $\det \big(\bLambda^{-1}+\bV^\top \bR_\bDelta \bV\big)$.  
Therefore, if $z$ is an eigenvalue of $\bK$ encircled by $\mathcal{C}_n$, 
\begin{align} \label{9br4}
\det\bigl(\bLambda^{-1}+\bV^\top \bR_\bDelta \bV\bigr)=0.
\end{align}

Our approach is to relate solutions of (\ref{9br4}) to the roots of
\begin{align}\label{9br6}
    \xi(z) \coloneqq \det \bigg({\bf \Lambda}^{-1} +  \begin{bmatrix}
        s(z) & \breve s(z) \\ \breve s(z) & \accentset{\circ}{s}(z)
    \end{bmatrix}\bigg ) =  -\frac{1 + \tau s(z) + a_1 s(z) (\lambda  + \gamma+ \tau \gamma s(z) - 1 )}{a_1^2(1+a_1 \gamma s(z))(\sqrt{\lambda}-1)^2} 
\end{align}
(this is our {\it master equation}, in the parlance of \cite{benaych2011eigenvalues}). 
Specifically, we will prove the following claims:
\begin{enumerate}
    \item $\displaystyle \sup_{z \in \mathcal{C}_n} \Big| \det \big({\bf \Lambda}^{-1} + \hspace{-.03cm} \bV^\top \hspace{-.05cm}   \bR_\bDelta \hspace{-.04cm} \bV\big) 
 - \xi(z)  \Big| \prec n^{-1/2}$,
 \item $\displaystyle \inf_{z \in \mathcal{C}_n}| \xi(z)| \gtrsim \delta_n $,
 \item Within $\mathcal{C}_n$, $\psi(s_+)$ is the unique root of $\xi$ and has multiplicity one, and  $\xi$ is holomorphic.
\end{enumerate}
Claims (1) and (2) imply 
\[
 \sup_{z \in \mathcal{C}_n} \big| \det ({\bf \Lambda}^{-1} + \hspace{-.03cm} \bV^\top \hspace{-.05cm}   \bR_\bDelta \hspace{-.05cm} \bV) 
 - \xi(z)  \big|  <   \inf_{z \in \mathcal{C}_n}| \xi(z)|  ,
\]
with probability $1-n^{-D}$. Since both functions are holomorphic within $\mathcal{C}_n$, Rouch\'e's theorem  implies $\det ({\bf \Lambda}^{-1} + \hspace{-.03cm} \bV^\top \hspace{-.05cm}   \bR_\bDelta \hspace{-.05cm} \bV)$ and $\xi(z)$ have equal numbers of roots within $\mathcal{C}_n$. By Claim (3), this number is one, completing the proof. 

On $\mathcal{C}_n$,  Theorem \ref{lem:quad_formsX} implies $ \| \bV^\top \hspace{-.05cm}\bR_0 \bDelta \bR_0  \bV \|_2  \prec n^{-1/2}$. Moreover,  by (\ref{poiu}), $\|\bS\|_2 \prec 1$, and $\|\bDelta\|_2 \prec n^{-1/4}$, 
\[
 \sup_{z \in \mathcal{C}_n} \big\| \bV^\top \bR_0 \bDelta \bR_\bDelta \bDelta \bR_0 \bV \big\|_2 \prec \|\bDelta\|_2^2 \prec n^{-1/2}
.\]
Thus, using the identities $\bR_\bDelta^{-1} - \bR_0^{-1} = \bDelta$ and
\begin{equation}
    \begin{aligned}\label{poiu0}\bR_\bDelta &= \bR_0 - \bR_\bDelta \bDelta \bR_0 \\& =  \bR_0 - \bR_0 \bDelta \bR_0 + \bR_0 \bDelta \bR_\bDelta \bDelta \bR_0 , \end{aligned}
\end{equation}
we obtain
\begin{equation}
\begin{gathered} \label{poiu1}
        \sup_{z \in \mathcal{C}_n} \big\|\hspace{-.03cm}  \bV^\top ( \bR_\bDelta - \bR_0 )  \bV \hspace{.05cm}\big\|_2  \prec n^{-1/2} ,\\
        \sup_{z \in \mathcal{C}_n} \Big| \det \big({\bf \Lambda}^{-1} +  \hspace{-.03cm} \bV^\top \hspace{-.05cm}   \bR_\bDelta \hspace{-.04cm} \bV\big)  - \det \big({\bf \Lambda}^{-1} + \hspace{-.03cm}\bV^\top \hspace{-.05cm}   \bR_0  \bV\big)  \Big| \prec n^{-1/2} .      
\end{gathered}    
\end{equation}
Moreover, Theorem \ref{lem:quad_forms} implies
\begin{equation}
\begin{gathered}\label{poiu2}
          \sup_{z \in \mathcal{C}_n} \Big| \det \big({\bf \Lambda}^{-1} + \hspace{-.03cm}\bV^\top \hspace{-.05cm}   \bR_0  \bV\big)  - \xi(z) \Big| \prec n^{-1/2} .  
\end{gathered}
\end{equation}
Claim (1) follows from (\ref{poiu1}) and (\ref{poiu2}).

Since the numerator of $\xi$ is quadratic in $s(z)$, the roots of $\xi$ are $\psi(s_+)$ and $\psi(s_-)$, where 
\[
s_{-} \coloneqq \frac{-a_1(\lambda+\gamma-1) - \tau -\sqrt{(a_1(\lambda+\gamma-1) + \tau)^2 - 4 a_1 \tau \gamma}}{2 a_1 \tau \gamma}  .
\]
Similarly to (\ref{9br7}), $s_-$ is decreasing in $\lambda$ on $[1, \infty)$, implying 
\begin{align} \label{a2026}
    s_- & < \frac{-a_1 \gamma - \tau - \sqrt{(a_1 \gamma + \tau)^2 - 4 a_1 \tau \gamma}}{2 a_1 \tau \gamma}  \leq \frac{-a_1 \gamma - \tau}{2 a_1 \tau \gamma} \leq - \frac{1}{a_1 \gamma} .
\end{align}
 Claim (2) follows from 
\begin{align} \label{mybp4}
\xi(z) =  -\frac{(s(z) - s_+)(s(z)-s_-)}{a_1^2(\sqrt{\lambda}-1)^2(1+a_1 \gamma s(z))} \asymp s(z) - s_+ ,
\end{align}
where the equivalence holds uniformly on $\mathcal{C}_n$, 
together with the identity
\begin{equation*}
    \begin{aligned}
    s(z) - s_+ &= \int \bigg(\frac{1}{\lambda-z} - \frac{1}{\lambda - \psi(s_+)}\bigg) d\mu(z) = \int \frac{z - \psi(s_+)}{(\lambda-z)(\lambda-\psi(s_+))} d\mu(z) ,
\end{aligned}
\end{equation*}
which implies
\begin{align} \label{mybp3}
    & \inf_{z \in \mathcal{C}_n}|s(z) - s_+| \gtrsim \delta_n . 
\end{align}
Claim (3) follows from (\ref{a2026}), (\ref{mybp4}), and Corollary \ref{cor:stj_trans}.
\end{proof}

\begin{lemma} \label{lem:det_equiv}
    Define the function 
\begin{align*}
\zeta(z) & \coloneqq s(z) - \begin{bmatrix}        s(z) \\ \breve s(z) 
    \end{bmatrix}^\top \hspace{-.1cm} \bigg({\bf \Lambda}^{-1} +  \begin{bmatrix}
        s(z) & \breve s(z) \\ \breve s(z) & \accentset{\circ}{s}(z)
    \end{bmatrix}\bigg )^{-1} \begin{bmatrix}        s(z) \\ \breve s(z) 
    \end{bmatrix} \\
    &= \frac{ s(z) (1+a_1 \gamma s(z)) }{1+\tau s(z)+a_1 s(z)(\lambda + \gamma + \tau \gamma s(z) -1)} . 
\end{align*}
For any $z \in \mathbb{C}^+$ and deterministic vector $\bu \in \mathbb{S}^{p-1}$,
\[
\Big| \bu^\top \spb (\bR -  \bR_\bDelta) \bu -  \langle \bu, \bv \rangle^2 (\zeta(z)-s(z)) \Big| \prec \frac{1}{\sqrt{n}} |\langle \bu, \bv \rangle | + \frac{1}{n}.
\]
Moreover, the convergence is uniform in $z$ on compact subsets $\mathcal{C} \subset \mathbb{C}$ such that $\inf_{z \in \mathcal{C}} \mathrm{Re}(z) > \lambda_+$. 
\end{lemma}

\begin{proof}
The claim is a consequence of the following equations, which are derived from  Theorems \ref{lem:quad_forms} and \ref{lem:quad_formsX},  (\ref{9br3}), and (\ref{poiu0})--(\ref{poiu2}):
\begin{equation}
    \begin{aligned} \phantom{\Big|} 
& \bu^\top \spb (\bR -  \bR_\bDelta) \bu =  -\bu^\top \spa \bR_\bDelta \hspace{-.05cm} \bV \big({\bf \Lambda}^{-1} + \bV^\top \hspace{-.05cm}  \bR_\bDelta  \hspace{-.05cm}  \bV\big)^{-1} \bV^\top \hspace{-.05cm}   \bR_\bDelta  \bu ,\\ &
    \bigg\| \big({\bf \Lambda}^{-1} + \bV^\top \hspace{-.05cm}  \bR_\bDelta \hspace{-.04cm}  \bV\big)^{-1} - \bigg({\bf \Lambda}^{-1} +  \begin{bmatrix}
        s(z) & \breve s(z) \\ \breve s(z) & \accentset{\circ}{s}(z)
    \end{bmatrix}\bigg )^{-1} \bigg\|_2  \prec n^{-1/2} , 
\\ 
 \phantom{\Big|}    &\bV^\top \hspace{-.05cm}  \bR_\bDelta  \bu =   \bV^\top \hspace{-.05cm}  \bR_0  \bu -   \bV^\top \hspace{-.05cm}  \bR_0 \bDelta \bR_0  \bu +  \bV^\top \hspace{-.05cm}  \bR_0 \bDelta \bR_\bDelta \bDelta \bR_0  \bu \\
    & \hspace{1.45cm}= \langle \bu, \bv \rangle \begin{bmatrix}        s(z) \\ \breve s(z) 
    \end{bmatrix}  + O_\prec(n^{-1/2}).
    \end{aligned}
\end{equation}
\end{proof}

\begin{proof}[Proof of (\ref{q7j})] Let $\mathcal{C}$ denote a positively-oriented contour that encircles $\psi(s_+)$ and is bounded away from $\mathrm{supp}(\mu)$; as in the proof of (\ref{q6j}), 
with probability $1-n^{-D}$, $\mathcal{C}$ contains only the maximum eigenvalue of $\bK$. Cauchy's integral formula  and Lemma  \ref{lem:det_equiv} therefore yield
\begin{equation} \label{vfrcde}
    \begin{aligned}
   \langle \bu_1, \bv \rangle^2 &= -\frac{1}{2 \pi i} \oint_\mathcal{C} \bv^\top \spa \bR(z) \bv \hspace{.04em}dz\\
&      = -\frac{1}{2 \pi i} \oint_\mathcal{C} \zeta(z) dz + O_\prec(n^{-1/2}) 
\end{aligned}
\end{equation}
Here, the first equality is a particular case of (1.16) in \cite{kato} and the second uses the fact that $\bR_\bDelta$ and $s(z)$ are holomorphic on and within $\mathcal{C}$, implying
\[
\oint_\mathcal{C} \bv^\top \bR_\bDelta \bv \hspace{.1em} dz = \oint_\mathcal{C} s(z) \hspace{.05em} dz  = 0 . 
\]
Since $\psi(s_+)$ is the unique pole of $\zeta(z)$ in $\mathcal{C}$ (recall that $\xi(\psi(s_+)) = 0)$,
\begin{equation}
    \begin{aligned}
   \langle \bu_1, \bv \rangle^2 &=  - \lim_{z \rightarrow \psi(s_+)} (z-\psi(s_+)) \zeta(z) + O_\prec(n^{-1/2}). 
\end{aligned} 
\end{equation}

We compute the limit on the right-hand side using L'H\^{o}pital's rule, as in  Corollary 2 of \cite{liao2020sparse}:
\begin{equation}
    \begin{aligned}\label{333}
-    \lim_{z \rightarrow \psi(s_+)} (z- \psi(s_+)) \zeta(z) & =  \lim_{z \rightarrow \psi(s_+)} \frac{s(z)(1+a_1 \gamma s(z)) + (z-\psi(s_+))(1 + 2 a_1 \gamma s(z) )s'(z)}{(a_1(\lambda + \gamma - 1) + \tau)s'(z) + 2a_1 \tau \gamma s(z) s'(z)} . 
\end{aligned}
\end{equation}
Since $s(z)$ is continuous on $(\lambda_+,\infty)$, 
\begin{align*}
&    \lim_{z \rightarrow \psi(s_+)}  s(z) = s_+. 
\end{align*}
Together with $s_+ \geq -1/(a_1 \gamma)$,  this implies 
\begin{gather*}
 \lim_{z \rightarrow \psi(s_+)} \frac{(z-\psi(s_+))(1 + 2 a_1 \gamma s(z))}{a_1(\lambda + \gamma - 1) + \tau + 2a_1 \tau \gamma s(z) } = 0 . 
\end{gather*}
Moreover, differentiating $\psi(s(z)) = z$ (which holds on $(\lambda_+, \infty)$) gives $\psi'(s(z))s'(z) = 1$ and
\begin{align}
\lim_{z \rightarrow \psi(s_+)}  s'(z) = \frac{1}{\psi'(s_+)} > 0 .  \label{qwert2026}
\end{align}
Thus,
\begin{equation}
    \begin{aligned}\label{333_2026}
    \lim_{z \rightarrow \psi(s_+)} (z- \psi(s_+)) \zeta(z)  & =  \lim_{z \rightarrow \psi(s_+)} \frac{-s(z)(1+a_1 \gamma s(z))}{\big(a_1(\lambda + \gamma - 1) + \tau + 2a_1 \tau \gamma s(z) \big) s'(z)} .
\end{aligned}
\end{equation}

We complete the proof using the identity
\begin{equation}
    \begin{aligned} \label{334}
\psi'(s_+) = \frac{1}{s_+^2} - \frac{a_1^2\gamma}{(1+a_1 \gamma s_+)^2} - \gamma(\|f\|_\phi^2-a_1^2) , 
\end{aligned}
\end{equation}
which is obtained by differentiating (\ref{eq:stj_trans}). By (\ref{qwert2026})--(\ref{334}), 
\begin{equation}
    \begin{aligned} \label{vfrcde1}
    -\lim_{z \rightarrow \psi(s_+)} (z-\psi(s_+)) \zeta(z) 
     & = 
      \theta^2(s_+). 
\end{aligned}
\end{equation}
Finally, we note that (\ref{333}), $s_+ \in (-1/(a_1 \gamma), 0)$, and $s'(\psi(s_+)) > 0$ imply $\theta^2(s_+) > 0$.
\end{proof}

\begin{proof}[Proof of (\ref{q9j})] Assume $\langle \bu_1, \bv \rangle \geq 0$ without loss of generality and let $\mathcal{C}$ denote a contour as in the proof of (\ref{q7j}). Similarly to (\ref{vfrcde})--(\ref{vfrcde1}), 
\begin{equation}
    \begin{aligned}
    \langle \bu_1, \bv \rangle u_{1i} &= -\frac{1}{2 \pi i} \oint_\mathcal{C} \be_i^\top \spa \bR(z) \bv \hspace{.1em}dz   = -\frac{v_i}{2 \pi i} \oint_\mathcal{C} \zeta(z) dz + O_\prec(n^{-1/2})  \\
&= \theta^2(s_+) v_i + O_\prec(n^{-1/2}) .
\end{aligned}
\end{equation}
The claim now follows from   (\ref{q7j}) and the fact that the limit of $\theta(s_+)$ is strictly positive:
\begin{equation}
    \begin{gathered}
        \langle \bu_1, \bv \rangle - \theta(s_+) = \frac{\langle \bu_1, \bv \rangle^2 - \theta^2(s_+)}{\langle \bu_1, \bv \rangle + \theta(s_+)} =  O_\prec(n^{-1/2}) , 
        \\ u_{1i} = \frac{\theta^2(s_+) v_i}{\theta(s_+) + O_\prec(n^{-1/2}) }  + O_\prec(n^{-1/2}) = \theta(s_+) v_i + O_\prec(n^{-1/2}) . 
    \end{gathered}
\end{equation}
\end{proof}

\begin{lemma} \label{lem:uglyWoodbury}  Let $\bW \in \mathbb{R}^{p \times (p-1)}$ be a semi-orthogonal matrix such that $\bW\bW^\top + \bv \bv^\top  = \bI_p$. For $r \in \{1,2\}$, define
\[
Q_r(z) \coloneqq \bv^\top  \spa \bK \hspace{.1em} \bW \big(z \bI_{p-1} -\bW^\top \spa \bK \hspace{.1em} \bW\big)^{-r} \bW^\top \spa \bK \bv . 
\]
For any $z \in \mathbb{C}^+$,
\begin{align*}
   Q_1(z)
    = z + \frac{1}{\zeta(z)} -  a_1(\lambda-1) - \tau  + O_\prec(n^{-1/4}). 
\end{align*}    
Moreover, the convergence is uniform in $z$ on compact subsets of $\mathbb{C}^+$.
\end{lemma}

\begin{proof}

As the proof is similar to that of Lemma \ref{lem:det_equiv}, we only sketch the argument. By the Woodbury identity, we have
\[
  Q_1(z) = z + (\bv^\top \spa \bR \bv)^{-1} - \bv^\top \spa \bK \bv .  
\]
Theorem \ref{thrm:poly1}, Lemma \ref{lem:det_equiv}, and the bound $|\bv^\top \spb \bS \bv - 1| \prec n^{-1/2}$ yield $\bv^\top \spa \bR \bv = \zeta(z) + O_\prec(n^{-1/2})$ and 
\begin{align*}
\bv^\top \spa \bK \bv &= \bv^\top \spc\big(\bV \spc \bLambda \bV^\top + \bK_0 + \bDelta\big) \bv \\
&=   \bLambda_{11} + 2\bLambda_{12} + \bv^\top \spa \bK_0 \bv + O_\prec(n^{-1/4}) .
\end{align*}
The claim  now follows from $|\bv^\top \spa \bK_0 \bv| \prec n^{-1/2}$, which is proved similarly to (\ref{nhy1}). 

\end{proof}

\begin{proof}[Proof of (\ref{q8j})]
      Let $\bW \in \mathbb{R}^{p \times (p-1)}$ be a semi-orthogonal matrix such that $\bW\bW^\top + \bv \bv^\top  = \bI_p$, and let $Q_r(z)$ denote the corresponding quadratic form defined in Lemma \ref{lem:uglyWoodbury}. 
Similarly to Section 2 of \cite{JohnstoneCorrelation}, we have
\begin{align*}
\begin{bmatrix}    \bv^\top \spa \bK \bv & \bv^\top \spa \bK \hspace{.1em} \bW \\  \bW^\top  \spa \bK \bv &  \bW^\top  \spa \bK \hspace{.1em} \bW \end{bmatrix} \begin{bmatrix} \bv^\top  \bu_1 \\  \bW^\top \spb \bu_1 \end{bmatrix}
&= \lambda_1 \begin{bmatrix} \bv^\top  \bu_1 \\ \bW^\top \spb \bu_1 \end{bmatrix} .
\end{align*}
Hence, 
\[ \bW^\top \spb \bu_1 = \langle \bu_1, \bv \rangle \big(\lambda_1 \bI_{p-1} -\bW^\top \spa \bK \hspace{.1em} \bW\big)^{-1} \bW^\top \spa \bK \bv ,\]
where the inverse exists almost surely. Using the normalization condition $\langle \bu_1, \bv \rangle^2 + \|\bW^\top \spb \bu_1 \|_2^2 = 1$, we obtain
\begin{align} \label{319nrgk}
      \langle \bu_1, \bv \rangle^2  \big( 1 + Q_2(\lambda_1) \big) = 1 .  
\end{align}
We shall prove  
\begin{align} \label{320nrgk}
    Q_2(\lambda_1) \xrightarrow{a.s.} \infty ,
\end{align}
implying $\langle \bu_1, \bv \rangle^2 \xrightarrow{a.s.} 0 $ by (\ref{319nrgk}).

By Lemma \ref{lem:uglyWoodbury}, 
\begin{align*}
    Q_1(z) \xrightarrow{a.s.}
     z + \frac{1}{\zeta(z)} -  a_1(\lambda-1) - \tau , 
\end{align*}
where the convergence is uniform in $z$ on compact subsets of $\mathbb{C}^+$.  Since  $\lambda_1 \xrightarrow{a.s.} \lambda_+$ by (\ref{q7j}) and uniform convergence of an analytic sequence implies uniform convergence of the derivative,  
\begin{align} 
     Q_2(\lambda_1+i\eta)\xrightarrow{a.s.}  1 -  \frac{\zeta'(\lambda_+ + i \eta)}{\zeta^2(\lambda_+ + i \eta)}   , 
\end{align}
for any $\eta > 0$. 
Using $s(\lambda_+ +i \eta) \rightarrow  s_0$ and $|s'(\lambda_+ +i \eta)|  \rightarrow \infty$ as $\eta \rightarrow 0$, and
\[ 
 1 -  \frac{\zeta'(z)}{\zeta^2(z)} =  \frac{1 + a_1 \gamma s(z)(2 + a_1 (\lambda+\gamma-1)s(z))}{s^2(z)(1+a_1 \gamma s(z))^2} \cdot s'(z) , \]
we obtain the lower bound 
\begin{equation}\begin{aligned}
&~   \liminf_{n \rightarrow \infty} | Q_2(\lambda_1+i\eta)| 
 \geq
    \liminf_{\eta \rightarrow 0} \liminf_{n \rightarrow \infty} |Q_2(\lambda_++i\eta)| \stackrel{a.s.}{=}  \infty . 
\end{aligned}
\end{equation}
The claim now follows from
\[
Q_2(\lambda_1)
>
|Q_2(\lambda_1+i\eta)|
.\]
\end{proof}

\subsection{Proof of Theorems \ref{thrm:A} and \ref{thrm:B}}


Throughout this section, we assume $f$ satisfies the conditions of Theorem \ref{thrm:A}.

\begin{lemma} \label{lem:polynomial_approx} There exist odd polynomials $\{f_\ell\}_{\ell \in \mathbb{N}}$ such that $\|f - f_\ell\|_\phi \rightarrow 0$ and  
 \[
     \lim_{\ell \rightarrow \infty} \limsup_{n \rightarrow \infty} \|\bK(f) - \bK(f_\ell)\|_2 \stackrel{a.s.}{=} \lim_{\ell \rightarrow \infty} \limsup_{n \rightarrow \infty} \|\bK_0(f) - \bK_0(f_\ell)\|_2 \stackrel{a.s.}{=} 0.
 \]
\end{lemma}

\noindent Lemma \ref{lem:polynomial_approx}, whose proof is deferred to Appendix \ref{sec:poly approx}, extends \cite{fan2019spectral} in two respects: (1)  Theorems 1.4 and 1.6 of \cite{fan2019spectral} pertain only to the noise kernel matrix $\bK_0(f)$,  and (2) Theorem 1.4 assumes $f$ is continuously differentiable. 

\begin{proof}[Proof of Theorem \ref{thrm:A}]
Let $\{f_\ell\}_{\ell\in\mathbb N}$ be the polynomials  from Lemma \ref{lem:polynomial_approx}, and write $a_{\ell,k} \coloneqq \langle f_\ell, h_k\rangle_\phi$. 
By the triangle inequality,
\begin{equation}
    \begin{aligned} \label{8123g}
   \|\bK(f) - \bA(f)\|_2 \leq &~  \|\bK(f) - \bK(f_\ell)\|_2 + \|\bK(f_\ell) - \bA(f_\ell)\|_2 \\ &~ +  \|\bA(f_\ell) - \bA(f)\|_2 .
\end{aligned}
\end{equation}
Lemma \ref{lem:polynomial_approx} and Theorem \ref{thrm:poly1} give
\begin{align} \label{2026z}
    & \lim_{\ell \rightarrow \infty} \limsup_{n \rightarrow \infty}  \|\bK(f) - \bK(f_\ell)\|_2  \stackrel{a.s.}{=} \lim_{n \rightarrow \infty}   \|\bK(f_\ell) - \bA(f_\ell)\|_2 \stackrel{a.s.}{=} 0. 
\end{align}
It therefore remains to bound the third term of (\ref{8123g}):
\begin{equation}
\begin{aligned} 
      \|\bA(f) - \bA(f_\ell)\|_2 \leq &~ |a_1 - a_{\ell, 1}| \cdot \|\bY - \bI_p - \bK_0(h_1)\|_2 +  |\tau - \tau(f_\ell,\beta,\lambda)| \\ &~+ \|\bK_0(f) - \bK_0(f_\ell)\|_2 .
\end{aligned}    
\end{equation}
Since $\|f - f_\ell\|_\phi \rightarrow 0$, we have $|a_1 - a_{\ell,1}| \rightarrow 0$. Moreover, by the Cauchy-Schwarz inequality,
\begin{equation}
\begin{aligned}
   |\tau - \tau(f_\ell, \beta,\lambda)| & \leq  \bigg( \sum_{k=1}^\infty \Big(\frac{(\lambda-1)^{k}}{\sqrt{k!}\beta^{k-1}}\Big)^2 \cdot \sum_{k=1}^\infty (a_k - a_{\ell,k})^2 \bigg)^{1/2} \lesssim \|f - f_\ell\|_\phi \longrightarrow 0  \label{2026_fz}
\end{aligned}
\end{equation}
as $\ell \rightarrow \infty$. 
Applying the almost sure bound 
\[\|\bY - \bI_p - \bK_0(h_1)\|_2 \leq \|\bX\|_2 + \|\text{diag}(\bS) - \bI_p\|_2 \lesssim 1 , \]
we obtain
\begin{align}
    \label{2026_y}\lim_{\ell \rightarrow \infty} \limsup_{n \rightarrow \infty} \|\bA(f) - \bA(f_\ell)\|_2 \stackrel{a.s.}{=} 0 .
\end{align}
Combining (\ref{8123g}), (\ref{2026z}), and (\ref{2026_y}) completes the proof.  
\end{proof}

\begin{lemma}\label{lem:max_eig}
   The operator norm of the kernel matrix of noise satisfies 
   \[
   \|\bK_0(f)\|_2 \xrightarrow{a.s.} \lambda_+ . 
   \]
\end{lemma}
\begin{proof}
     Let $\{f_\ell\}_{\ell \in \mathbb{N}}$ be the polynomials from Lemma \ref{lem:polynomial_approx}, and let $\lambda_+(f_\ell)$ be the supremum of the support of the LSD of $\bK_0(f_\ell)$.  By Theorem 1.7 of \cite{fan2019spectral}, for any fixed $\ell$,
     \[ \|\bK_0(f_\ell)\|_2 \xrightarrow{a.s.} \lambda_+(f_\ell).\]
     Thus, almost surely,
\begin{equation}
    \begin{aligned} \label{2026aq}
        \limsup_{n \rightarrow \infty} \|\bK_0(f)\|_2 & \leq \lambda_+(f_\ell) + \limsup_{n \rightarrow \infty}  \|\bK_0(f) - \bK_0(f_\ell)\|_2 ,  \\
\liminf_{n \rightarrow \infty}  \|\bK_0(f)\|_2 & \geq \lambda_+(f_\ell) -  \limsup_{n \rightarrow \infty} \|\bK_0(f) - \bK_0(f_\ell)\|_2 .
    \end{aligned}
\end{equation}
Since $\|f - f_\ell\|_\phi \rightarrow 0$ as $\ell \rightarrow \infty$ and the solutions of $\psi'(s) = 0$ depend continuously on the coefficients in (\ref{eq:stj_trans2}), 
     \[ \lim_{\ell \rightarrow \infty} \lambda_+(f_\ell) = \lambda_+.
     \]
Taking limits as $\ell \rightarrow \infty$ in (\ref{2026aq}) and applying Lemma \ref{lem:polynomial_approx} completes the proof.
\end{proof}



Theorem \ref{thrm:B} follows from Theorems \ref{thrm:poly2} and \ref{thrm:A}, Lemmas \ref{lem:polynomial_approx} and \ref{lem:max_eig}, Weyl's inequality, and the Davis-Kahan theorem.

\subsection{Proofs for Section \ref{sec:2.3}}\label{sec:2.3proofs}

Recall $\psi(s)$ and $s_+$, defined in Corollary \ref{cor:stj_trans} and Theorem \ref{thrm:poly2}, respectively. In this section, for clarity, we write $[\psi'(s_+)](f,\lambda)$ to reflect that $s_+$ and $\psi'(s)$ each depend on $f$ and $\lambda$.  Lemma~\ref{lem:cR-poly-approx} is a consequence of Corollary \ref{thrm:SR} and Lemma \ref{propasdf}: 
\begin{lemma}\label{propasdf}
    If $[\psi'(s_+)](f,\lambda) > 0$, there exists $L \in \mathbb{N}$ such that $[\psi'(s_+)](f_\ell,\lambda) > 0$ for all $\ell \geq L$. 
\end{lemma}

\begin{proof}
As $\ell\rightarrow\infty$,
$\langle f_\ell,h_1\rangle_\phi\rightarrow a_1$,
$\|f_\ell\|_\phi\rightarrow\|f\|_\phi$, and, by (\ref{2026_fz}),
$\tau(f_\ell,\beta,\lambda)\rightarrow\tau(f,\beta,\lambda)$.
Since  $s_+$ and $\psi'(s_+)$ depend continuously on these quantities at $a_1$, $\|f\|_\phi,$ and $\tau$, it follows that
$[\psi'(s_+)](f_\ell,\lambda)\rightarrow
[\psi'(s_+)](f,\lambda)$.
\end{proof}

\begin{proof}[Proof of Lemma \ref{prop:non_dec}]
Since $s_+(\lambda,\tau)$ is increasing in each argument
(see (\ref{9br7})) and $\tau(\lambda)$ is non-decreasing, 
$s_+(\lambda,\tau(\lambda))$
is increasing in $\lambda$. Recall that $\psi'$ has a unique
root $s_0$ in $(-1/(a_1\gamma),0)$, and $\psi'(s) > 0$ if and 
only if $s > s_0$. Therefore, 
\[   \cR(f,\gamma,\beta,\cdot) = \big\{ \lambda \geq 1 :  \psi'(s_+(\lambda,\tau(\lambda)) > 0 \big\} = (\lambda_*(f,\gamma,\beta), \infty) ,\]
where
\[
\lambda_*(f,\gamma,\beta) \coloneqq \inf \big\{ \lambda \geq 1 :  \psi'(s_+(\lambda,\tau(\lambda)) > 0 \big\} . \]
\end{proof}

\begin{proof}[Proof of Corollary \ref{cor:opt_transform}]
Let $f_\lambda $ have Hermite coefficients $\langle f_\lambda,h_1 \rangle_\phi = 0$ and 
\[ \hspace{3.3cm}
    \langle f_\lambda, h_k \rangle_\phi = \begin{dcases}
        \frac{(\lambda-1)^k}{\sqrt{k!}\beta^{k -1}} & k \text{ odd},\\
        0 & k \text{ even},  
    \end{dcases} \hspace{2.6cm} k \neq 1 .
\]
Given any $f \in L_o^2(\phi)$, define
\[
{\breve f }_\lambda \coloneqq \langle f, h_1 \rangle_\phi h_1 + \frac{\sqrt{\|f\|_\phi^2 - \langle f, h_1 \rangle_\phi^2}}{\| f_\lambda\|_\phi} f_\lambda .
\]
The orthogonality of the Hermite polynomials implies  $\|\breve f_\lambda\|_\phi = \|f\|_\phi$ and  $\langle \breve f_\lambda, h_1\rangle_\phi = \langle f,h_1\rangle_\phi$. Moreover, by Cauchy-Schwarz inequality,
\begin{equation}
\begin{aligned} \label{348}
\tau(f,\beta,\lambda) &= \sum_{k=3}^\infty \frac{\langle f, h_k \rangle_\phi (\lambda-1)^k}{\sqrt{k!}\beta^{k -1}} \\ 
& \leq \sqrt{\|f\|_\phi^2 - \langle f, h_1 \rangle_\phi^2} \bigg( \sum_{k=3}^\infty \frac{ (\lambda-1)^{2k}}{k!\beta^{2(k -1)}} \bigg)^{1/2}\\
&= \sqrt{\|f\|_\phi^2 - \langle f, h_1 \rangle_\phi^2}  \|f_\lambda\|_\phi . 
\\& =  \tau(\breve f_\lambda,\beta,\lambda) . 
\end{aligned}   
\end{equation}
Similarly to the proof of Lemma \ref{prop:non_dec}, these properties imply 
\[ \hspace{4.5cm} [\psi'(s_+)](f,\lambda) \leq [\psi'(s_+)](\breve f_\lambda,\lambda ) , \hspace{2cm} \forall \lambda \in \cR(f,\gamma,\beta,\cdot) .  \] 

 Thus, since $\cR(f,\gamma,\beta,\cdot)$ is invariant under rescalings of $f$,
\begin{equation}
    \begin{aligned}
        \bigcup_{f \in L_o^2(\phi)} \cR(f,\gamma,\beta,\cdot) &= \bigcup_{a \geq 0} \bigcup_{\lambda > 1} \cR\big(a h_1 + f_\lambda,\gamma,\beta,\cdot\big) .
    \end{aligned}
\end{equation}
Consider $\{(a_k, \lambda_k)\}_{k \in \mathbb{N}}$ such that 
\[
\lim_{k \rightarrow \infty} \inf \cR\big( a_k h_1 + f_{\lambda_k},\gamma,\beta,\cdot\big) = \inf_{a \geq 0, \lambda \geq 1} \cR\big(a h_1 + f_\lambda,\gamma,\beta,\cdot\big) .
\]
Assuming $\{\| a_k h_1 + f_{\lambda_k}\|_\phi\}_{k \in \mathbb{N}} $ is bounded without loss of generality,  there exists a convergent subsequence $\{(a_{k_\ell}, \lambda_{k_\ell})\}_{\ell \in \mathbb{N}}$; let  $a_{k_\ell} \rightarrow a_* \geq 0$ and $\lambda_{k_\ell} \rightarrow \lambda_* > 1$ as $\ell \rightarrow \infty$, and define $f^* \coloneqq a_* h_1 + f_{\lambda_*}$.   As $[\psi'(s_+)](f,\lambda)$ is continuous in the Hermite coefficients of $f$, we conclude that
\begin{align*}    \cR\big(a_{k_\ell} h_1 + f_{\lambda_{k_\ell}},\gamma,\beta,\cdot\big) \rightarrow \cR(f^*, \gamma, \beta, \cdot).\end{align*}
Together with Lemma \ref{prop:non_dec}, this yields
\begin{equation}
    \begin{gathered}
         \bigcup_{f \in L_o^2(\phi)} \cR(f,\gamma,\beta,\cdot) = \cR(f^*,\gamma,\beta,\cdot) .
    \end{gathered}
\end{equation}
To complete the proof, we note that  (\ref{348}) implies
\[
\lambda_* = \inf \big\{ \lambda \geq 1 :  [\psi'(s_+)](f^*,\lambda) > 0 \big\} . \]

\end{proof}

\begin{proof}[Proof of Corollary \ref{prop_betainf}]

We first establish monotonicity of $\beta \mapsto \lambda_*(\gamma,\beta)$. Let $L^2_{o,+}(\phi)\subset L_o^2(\phi)$ denote the subset of kernel functions with non-negative Hermite coefficients. By Lemma \ref{prop:non_dec}
and Corollary \ref{cor:opt_transform},  
\begin{equation}
    \begin{aligned} \label{eq:hga7}
    (\lambda_*(\gamma,\beta),\infty) &= \bigcup_{f\in L^2_o(\phi)}\cR(f,\gamma,\beta, \cdot) = \bigcup_{f\in L^2_{o,+}(\phi)}\cR(f,\gamma,\beta, \cdot) \\
    &= \bigcup_{f\in L^2_{o,+}(\phi)}(\lambda_*(f,\gamma,\beta),\infty) . 
\end{aligned}
\end{equation}
For all $\lambda>1$ and $f\in L_{o,+}^2(\phi)$, $\tau(f,\beta,\lambda)$ is non-increasing in $\beta$, implying (similarly to the proof of Corollary \ref{cor:opt_transform}) that $\beta\mapsto \lambda_*(f,\gamma,\beta)$ is non-decreasing.  Consequently,  $\beta\mapsto \lambda_*(\gamma,\beta)$ is non-decreasing. 
Since $\lambda_*(\gamma,\beta)$ is upper-bounded by the BBP transition  (corresponding to the identity kernel), $\lim_{\beta\to\infty}\lambda_*(\gamma,\beta)$ exists and is at most $1+\sqrt{\gamma}$.

Fix $\lambda>\lim_{\beta\to\infty}\lambda_*(\gamma,\beta)$; it suffices to prove that $\lambda\ge 1+\sqrt{\gamma}$. By 
 (\ref{eq:hga7}), there exists a sequence of sparsity levels and kernels $\{(\beta_k,f_k)\}_{k\in \mathbb{N}}$ satisfying $\beta_k\to \infty$ and, for all $k \in \mathbb{N}$, $\lambda \in \cR(f_k,\gamma,\beta_k,\cdot)$, $f_k \in L^2_{o,+}(\phi)$, and $\|f_k\|_\phi=1$. For brevity, denote
\[
a_{1,k} \coloneqq \langle f_k, h_1 \rangle_\phi,\quad\quad\quad \tau_k \coloneqq \tau(f_k,\beta_k,\lambda),\quad\quad\quad s_{k} \coloneqq s_+(f_k,\beta_k,\lambda) .
\]
Note that (1) if there exists $k \in \mathbb{N}$ such that $\tau_k = 0$, the claim follows from $\cR(f_k,\gamma,\beta_k,\cdot) = (1+\sqrt{\gamma},\infty)$, (2) $\tau_k\to 0$ as $k \rightarrow \infty$ since $\|f_k\|_{\phi}=1$,  and (3) $a_{1,k} > 0 $ for all sufficiently large $k$. The third fact follows from 
Remark \ref{rem1}, which implies that if $a_{1,k} = 0$, $\lambda\in \cR(f_k,\gamma,\beta_k,\cdot)$ if and only if $\tau_k > \sqrt{\gamma}$ (which contradicts $\tau_k \rightarrow 0$). 
Henceforth, we assume without loss of generality that $a_{1,k}, \tau_k > 0$. 
From Corollary \ref{cor:stj_trans} and Theorem~\ref{thrm:poly2}, recall that $\psi'(s_k)>0$, $s_k \in (-1/(a_{1,k} \gamma),0)$, and
\begin{align*}
    s_{k} 
    &= \frac{1}{2a_{1,k}\gamma\tau_k}\left(  - a_{1,k}(\lambda+\gamma-1)-\tau_k + \sqrt{(a_{1,k}(\lambda+\gamma-1)+\tau_k)^2-4a_{1,k}\gamma\tau_k}\right), \\
    \psi'(s_{k}) &= \frac{1}{s_{k}^2} - \frac{a_{1,k}^2 \gamma}{(1+a_{1,k}\gamma s_{k})^2}-\gamma(1-a_{1,k}^2) . 
    \end{align*}

Assume first that $a_{1,k}$ is bounded away from zero. Taking a subsequence if necessary, let $a_{1,k} \rightarrow a_1 \in (0,1]$ as $k \rightarrow \infty$. Then,
\begin{align*}
\lim_{k \to \infty} s_{k} &= - \frac{1}{a_1(\lambda+\gamma-1)},     \\
\lim_{k \to \infty} s_{k}^2 \psi'(s_{k}) &= 1 - \frac{\gamma}{(\lambda-1)^2} - \frac{\gamma(1-a_1^2)}{a_1^2(\lambda+\gamma-1)^2} \le  1 -\frac{\gamma}{(\lambda-1)^2} .
\end{align*}
Since $s_k^2 \psi'(s_k) \geq 0$, this implies $\lambda \geq 1 + \sqrt{\gamma}$. In fact, as the inequality is strict unless $a_1 = 1$ (corresponding to the identity kernel), we conclude that PCA is the unique optimal kernel in the limit $\beta \rightarrow \infty$.

It remains to prove $a_{1,k} \asymp 1$. 
Assume towards a contradiction that $a_{1,k} \to 0$. There are three cases to consider: (1) $a_{1,k}/\tau_k\to 0$, (2) $a_{1,k}/\tau_k\rightarrow \infty$, and (3) $a_{1,k}/\tau_k \asymp 1$. In each case, we will argue that $s_k \rightarrow -\infty$, implying 
\[
\liminf_{k \rightarrow \infty} \psi'(s_k) \leq \lim_{k \rightarrow \infty} \frac{1}{s_k^2} - \gamma(1-a_{1,k}^2) = -\gamma 
\]
and contradicting $\psi'(s_k) > 0$. 

\begin{enumerate}
    \item If $a_{1,k}/\tau_k \rightarrow 0$,  we have
\begin{align*}
   s_k 
    &= \frac{1}{2 a_{1,k}\gamma} \Bigg(  - (\lambda+\gamma-1)\frac{a_{1,k}}{\tau_k}  - 1 + \sqrt{1+2(\lambda+\gamma-1)\frac{a_{1,k}}{\tau_k} + (\lambda+\gamma-1)^2 \frac{a_{1,k}^2}{\tau_k^2} - \frac{4\gamma a_{1,k}}{\tau_k}} 
    \Bigg) \\
    &= -\frac{1}{\tau_k} + O\bigg( \frac{a_{1,k}}{\tau_k^2}\bigg).
\end{align*}
This yields $ s_k \rightarrow -\infty$ since $\tau_k > 0$ tends to zero.

\item If $a_{1,k}/\tau_k \to \infty$, we have
\begin{align*}
    a_{1,k} s_k &= \frac{a_{1,k} (\lambda+\gamma-1)+ \tau_k}{2\gamma \tau_k} \left( 
\sqrt{1 - \frac{4 a_{1,k}\gamma\tau_k}{(a_{1,k}(\lambda+\gamma-1)+\tau_k)^2}} - 1
    \right) \\
    &= -\frac{a_{1,k} (\lambda+\gamma-1)+ \tau_k}{2\gamma \tau_k} \bigg(\frac{2\gamma \tau_k}{a_{1,k}(\lambda+\gamma-1)^2} + O \bigg( \frac{\tau_k^2}{a_{1,k}^2}\bigg) \bigg) . 
\end{align*}
This yields $a_{1,k} s_k \to -1/(\lambda+\gamma-1)$ and $s_k \rightarrow -\infty$.

\item If $a_{1,k}/\tau_k \asymp 1$, similarly to the previous case, we have $a_{1,k}s_k \asymp 1$ and $s_k \rightarrow -\infty$. 

\end{enumerate}


\end{proof}

\section*{Acknowledgements}

The authors are grateful to David Donoho and Boaz Nadler for helpful discussions and comments. 

\nocite{*} 
\bibliographystyle{abbrv} 
\bibliography{bibliography.bib}

\appendix

\clearpage

\section{Proof of Theorems \ref{lem:quad_forms} and  \ref{lem:quad_formsX}}\label{sec:appendix1}

Throughout this appendix, $f$ and $g$ are odd polynomials and $L \coloneqq \mathrm{deg}(f) \vee \mathrm{deg}(g)$. For simplicity, the proofs given below assume $\bu = \bw$ and $z \in \mathbb{C}^+$, implying $\|\bR_0(f,z)\|_2 \leq \Im(z)^{-1} < \infty$. 
{Given 
a compact set $\mathcal{C}\subset \mathbb{C}$ such that $\inf_{z \in \mathcal{C}} \mathrm{Re}(z) > \lambda_+$, the extension from pointwise to uniform convergence follows by a standard discretization and union-bound argument. }
{The generalizations of (\ref{nhy1}) and (\ref{nhy3}) to the asymmetric case $\bu \neq \bw$ follows by applying symmetric results to $\bu+\bw$. For example,
\begin{align*}
    (\bu+\bw)^\top\bR_0(f,z)(\bu+\bw) &= \|\bu+\bw\|_2^2 s(z) + O_\prec(n^{-1/2}) \\
    &= 2 s(z) + 2 \bu^\top \bR_0(f,z)\bw + O_\prec(n^{-1/2}) . 
\end{align*}
This yields 
\[ \bu^\top \bR_0(f,z) \bw = \Big(\frac{1}{2}\|\bu+\bw\|_2^2-1\Big)s(z) + O_\prec(n^{-1/2}) = \langle \bu, \bw \rangle s(z) + O_\prec(n^{-1/2}) .\] The generalization of (\ref{nhy2}) follows from slight modifications of arguments in Sections \ref{sec:quadforms1} and \ref{sec:quadforms2}. }

Theorem \ref{lem:quad_forms} is an immediate consequence of Lemmas \ref{lem:quad_forms_E} and \ref{lem:quad_forms_E2}:

\begin{lemma}\label{lem:quad_forms_E}
 For any  $z \in \C^+$ and deterministic  vectors $\bu, \bw \in \S^{p-1}$,
\begin{align} 
   &\bu^\top  \bR_0(f,z) \bw - \E [\bu^\top  \bR_0(f, z) \bw]  = O_{\prec} (n^{-1/2}) , \label{ghjk1}\\ 
   &\bu^\top  \bS \bR_0(f,z) \bw - \E \big[ \bu^\top \bS  \bR_0(f, z) \bw\big] = O_{\prec} (n^{-1/2}), \label{ghjk2}\\ 
  & \bu^\top \bS\bR_0(f,z) \bS  \bw - \E  \big[\bu^\top \bS  \bR_0(f, z) \bS \bw\big] = O_{\prec} (n^{-1/2}) .  \label{ghjk3} 
\end{align}
\end{lemma}

 \begin{lemma}\label{lem:quad_forms_E2} Let $p/n = \gamma + O(n^{-1/2})$. For any  $z \in \C^+$, $\eps >0$, and deterministic vectors $\bu, \bw \in \S^{p-1}$,
\begin{align} 
  & \E[\bu^\top  \bR_0(f,z) \bw ] - \langle \bu, \bw \rangle s(z)  = O (n^{-1/2+\eps}), \label{eq:expectation_quad_form_m} \\ 
  & \E [ \bu^\top  \bS \bR_0(f,z) \bw ] - \langle \bu, \bw \rangle \breve s(z)  = O (n^{-1/2+\eps}), \label{eq:expectation_quad_form_s1}\\ 
  & \E[\bu^\top \bS\bR_0(f,z) \bS  \bw ]  - \langle \bu, \bw \rangle \accentset{\circ}{s}(z) = O (n^{-1/2+\eps}).   \label{eq:expectation_quad_form_s2}
\end{align}
 \end{lemma}

Theorem \ref{lem:quad_formsX} is a corollary of Lemmas \ref{lem:quad_forms_1000}--\ref{lem:ygvuhb} and the proof of Theorem \ref{thrm:poly1}. Recall \[
\begin{aligned}
\bX = \bSigma^{1/2} \bS \bSigma^{1/2} - \bS = c \cdot \bv \bv^\top + (\sqrt{\lambda} - 1) \big(\bv (\bS \bv - \bv)^\top + (\bS \bv - \bv) \bv^\top\big),
\end{aligned}
\]
where $c \coloneqq (\sqrt{\lambda} - 1)^2 (\bv^\top \bS \bv) + 2 (\sqrt{\lambda} - 1)$, and define $\bK_{0}^{(m)}(g) \coloneqq  \bK_0(g) \odot (\sqrt{n} \bX)$;  the superscript of $\bK_{0}^{(m)}(g)$ signifies that only the first $m$ rows and columns of this matrix are non-zero. 

\begin{lemma} \label{lem:quad_forms_1000}
For any $z \in \mathbb{C}^+$ and  deterministic  vectors $\bu, \bw \in \mathbb{S}^{p-1}$ such that $\|\bu\|_0 \vee \|\bw\|_0 \lesssim m$,
   \begin{align} 
   &\bu^\top  \bR_0(f,z) \bK_0^{(m)}(g) \bR_0(f,z)\bw - \E \Big[\bu^\top  \bR_0(f, z)  \bK_0^{(m)}(g) \bR_0(f,z) \bw\Big]  = O_{\prec} (n^{-1/2}) , \label{qwerty1}\\ 
   &\bu^\top  \bS \bR_0(f,z)  \bK_0^{(m)}(g) \bR_0(f,z) \bw - \E \Big[ \bu^\top \bS  \bR_0(f, z)  \bK_0^{(m)}(g) \bR_0(f,z) \bw\Big] = O_{\prec} (n^{-1/2}), \label{qwerty2}\\ 
  & \bu^\top \bS\bR_0(f,z)  \bK_0^{(m)}(g) \bR_0(f,z) \bS  \bw - \E  \Big[\bu^\top \bS  \bR_0(f, z)  \bK_0^{(m)}(g) \bR_0(f,z) \bS \bw\Big] = O_{\prec} (n^{-1/2}) .  \label{qwerty3} 
\end{align}
 \end{lemma}

 \begin{lemma} \label{lem:quad_forms_1000E}    For any $z \in \mathbb{C}^+$ and  deterministic  vectors $\bu, \bw \in \mathbb{S}^{p-1}$ such that $\|\bu\|_0 \vee \|\bw\|_0 \lesssim m$,
    \begin{align}
   & \E \Big[\bu^\top  \bR_0(f, z)  \bK_0^{(m)}(g) \bR_0(f,z) \bw\Big]  = O(n^{-1/2+\varepsilon}) , \label{qwerty4}\\ 
   & \E \Big[ \bu^\top \bS  \bR_0(f, z)  \bK_0^{(m)}(g) \bR_0(f,z) \bw\Big] =O(n^{-1/2+\varepsilon}), \label{qwerty5}\\ 
  &  \E  \Big[\bu^\top \bS  \bR_0(f, z)  \bK_0^{(m)}(g) \bR_0(f,z) \bS \bw\Big] = O(n^{-1/2+\varepsilon}).  \label{qwerty6} 
    \end{align}
 \end{lemma}

 \begin{lemma} \label{lem:ygvuhb}    For any $z \in \mathbb{C}^+$ and  deterministic  vectors $\bu \in \mathbb{S}^{p-1}$,
    \begin{align}
   & \bu^\top \bR_0(f,z) (\bv \odot (\bS \bv - \bv))    = O_\prec(n^{-1/2}) ,\\
   & \bu^\top \bS\bR_0(f,z) (\bv \odot (\bS \bv - \bv))    = O_\prec(n^{-1/2}) .\end{align}
 \end{lemma}

 The proof of Lemma \ref{lem:ygvuhb} is similar to that of Lemma \ref{lem:quad_forms_E} and is omitted. Henceforth, for brevity, we shall often suppress the arguments of matrices such as $\bK_0(f)$ and $\bR_0(f,z)$.

\subsection{Proof of Lemma \ref{lem:quad_forms_E}} \label{sec:quadforms1}

We shall prove (\ref{ghjk3}); the proofs of (\ref{ghjk1}) and (\ref{ghjk2}) are similar and omitted. 

Our approach is to express $\bu^\top \bS \bR_0 \bS \bu - \E [\bu^\top \bS \bR_0 \bS \bu]$ as the sum of a martingale difference sequence and then apply the Burkholder inequality (Lemma \ref{lem:burkholder}); this method, applied to sample covariance matrices, is standard in random matrix theory (see Sections 8 and 10 of \cite{bai2007asymptotics}), although its extension to kernel matrices is involved.  

For $j \in [n]$, let $\bS_{-j} \coloneqq \bS - n^{-1} \bz_j \bz_j^\top$, $\bK_{0,-j}$ denote the kernel matrix of $\bS_{-j}$ (defined analogously to $\bK_0$), and  $\bR_{0,-j} \coloneqq ( \bK_{0,-j}  - z \bI_p)^{-1}$.  Define  $\mathcal{F}_0 \coloneqq \emptyset$, $\mathcal{F}_j \coloneqq \sigma(\{\bz_{1}, \ldots, \bz_j\})$, the conditional expectation $\E_j(\cdot) \coloneqq \E(\cdot | \mathcal{F}_j)$, and $\alpha_j \coloneqq  (\E_j - \E_{j-1} )  \bu^\top \bS \bR_0 \bS \bu$. Then, $\{\alpha_j\}_{j \in [n]}$ is a martingale difference sequence with respect to $\{\mathcal{F}_j \}_{j \in [n]}$ and we have the decomposition
\begin{equation} \label{eq:MDS}
\begin{aligned}
    \bu^\top \bS \bR_0 \bS \bu  - \E[\bu^\top \bS \bR_0 \bS \bu] = \sum_{j = 1}^n (\E_j - \E_{j-1} )  \bu^\top \bS \bR_0 \bS \bu = \sum_{j =1}^n \alpha_j.
    \end{aligned}    
\end{equation}
We further expand $\alpha_j$ into four terms:
\[
\begin{aligned}
\alpha_j =&~  (\E_j - \E_{j-1} )  \Big[ \bu^\top \bS \bR_0 \bS \bu - \bu^\top \bS_{-j} \bR_{0,-j} \bS_{-j} \bu \Big]\\
=&~ \alpha_{j,1} + 2\alpha_{j,2} +2\alpha_{j,3}+\alpha_{j,4},
\end{aligned}
\]
where 
\begin{equation}\label{eq:definition_alpha_1_4}
\begin{aligned}
    \alpha_{j,1} \coloneqq &~ \frac{1}{n^2} (\E_j - \E_{j-1} ) \big[  (\bu^\top \bz_j)^2 \cdot \bz_j^\top \bR_0 \bz_j \big], \\
    \alpha_{j,2} \coloneqq&~ \frac{1}{n} (\E_j - \E_{j-1} ) \big[ \bu^\top \bz_j  \cdot  \bz_j^\top \bR_{0,-j} \bS_{-j} \bu \big], \\
    \alpha_{j,3} \coloneqq&~ \frac{1}{n} (\E_j - \E_{j-1} ) \big[\bu^\top \bz_j  \cdot  \bz_j^\top (\bR_0 - \bR_{0,-j}) \bS_{-j} \bu  \big] , \\
    \alpha_{j,4} \coloneqq&~ (\E_j - \E_{j-1} ) \big[ \bu^\top \bS_{-j} (\bR_0 - \bR_{0,-j}) \bS_{-j} \bu \big] .
\end{aligned}
\end{equation}

The proof of (\ref{ghjk3}) is a consequence of the Burkholder inequality applied to $\{\alpha_j\}_{j \in [n]}$ and the following bound on $\E_{j-1}|\alpha_j|^k$ (which we derive by bounding $\alpha_{j1}, \ldots, \alpha_{j4}$):
\begin{lemma}\label{lem:bound_alpha_j_martingale}
    For each $j \in [n]$, define the random variable 
        \begin{equation*}\label{eq:def_Gamma_2_infty}
   \Gamma_{j} \coloneqq 1 \vee \| \bS_{-j} \|_2 \vee \bigg( \sup_{\ell \in [L]} \| \bK_{0,-j} (h_\ell) \|_2 \bigg) \vee \bigg(\sup_{ \ell  \in [L]}  \sqrt{n} \| \bK_{0,-j} (h_\ell) \|_{\infty} \bigg).
    \end{equation*}
    Then, 
    \[     \E_{j-1} | \alpha_j|^k \lesssim \frac{(\log n)^{4Lk}}{n^k} \E_{j-1}\Gamma_{j}^{4k} .    \]
\end{lemma}

\begin{lemma}\label{lem:bound_Gamma_j}
    For all $j \in [n]$,  $\eps>0$, and $k \in \mathbb{N}$, 
    \begin{align}\label{eq:Gammas_bound}
    & \Gamma_{j} \prec 1, &&  \E \hspace{.09em} \Gamma_{j}^k \lesssim n^\eps.
    \end{align}
\end{lemma}

\begin{lemma}[Burkholder inequality, Lemma 2.13 of \cite{bai2010spectral}]\label{lem:burkholder} Let $\{ \alpha_j\}_{j = 1}^\infty$ be a martingale difference sequence with respect to the filtration $\{\cF_j\}_{j=1}^\infty$, and let $\E_j(\cdot) = \E(\cdot | \cF_j)$. Then, for  $k \geq 2$, 
\[
\E  \bigg| \sum_{j=1}^n \alpha_j \bigg|^k  \lesssim \E \bigg[ \bigg( \sum_{j=1}^n \E_{j-1} | \alpha_j|^2\bigg)^{k/2} \bigg] + \sum_{j=1}^n \E|\alpha_j|^k 
\]    
(the implied constant depends only on $k$).
\end{lemma}

\begin{proof}[Proof of (\ref{ghjk3})] 
By  Lemma \ref{lem:bound_alpha_j_martingale} and Jensen's inequality, 
\begin{equation}
    \begin{aligned} \label{eq:vfr4}
 \sum_{j = 1}^n \E|\alpha_j|^k & \lesssim \frac{(\log n)^{4Lk}}{n^{k-1}} \E \hspace{.09em} \Gamma_{1}^{4k},   \\
  \E  \bigg[ \bigg( \sum_{j = 1}^n \E_{j-1} |\alpha_j|^2 \bigg)^{k/2} \bigg] & \lesssim  \frac{(\log n)^{4Lk}}{n^k} \E \bigg[ \bigg( \sum_{j = 1}^n \E_{j-1} \Gamma_{j}^{8} \bigg)^{k/2} \bigg] \\ 
&  \lesssim   \frac{(\log n)^{4Lk}}{n^{k/2}} \E \bigg[ \bigg(\sup_{j \in [n]}   \E_{j-1} \Gamma_{j}^{8}\bigg)^{k/2} \bigg].
    \end{aligned}
\end{equation}
Now, for $j \in [n]$, let  $\widetilde \Gamma_{j}$ denote $\Gamma_{j}$ with columns $\bz_{1}, \ldots, \bz_{j-1}$ of $\bZ$ replaced by i.i.d.\ copies  $\tilde \bz_{1}, \ldots, \tilde \bz_{j-1}$, enabling us to write $\E_{j-1}\Gamma_j^8 = \E\big[\widetilde \Gamma_j^8 \big| \tilde \bz_1, \ldots, \tilde \bz_n\big] $ and yielding the bound
\begin{equation}
    \begin{aligned}
  \E  \bigg[ \bigg( \sum_{j = 1}^n \E_{j-1} |\alpha_j|^2 \bigg)^{k/2} \bigg] & \lesssim   \frac{(\log n)^{4Lk}}{n^{k/2}} \E \bigg[ \bigg(\sup_{j \in [n]}   \E\big[ \widetilde \Gamma_{j}^{8} \big| \tilde \bz_1, \ldots, \tilde \bz_n \big] \bigg)^{k/2} \bigg] \\
  & \leq \frac{(\log n)^{4Lk}}{n^{k/2}} \E \bigg[  \sup_{j \in [n]}   \widetilde \Gamma_{j}^{4k} \bigg] . 
    \end{aligned} \end{equation}
Using Lemma \ref{lem:bound_Gamma_j} and a union bound, for $\varepsilon, D > 0$, we have  
\begin{equation}
    \begin{aligned}
         \E \bigg[  \sup_{j \in [n]}   \widetilde \Gamma_{j}^{4k} \bigg] & \lesssim n^\varepsilon + \E \bigg[  \sup_{j \in [n]}   \widetilde \Gamma_{j}^{4k} \cdot {\bf 1}\bigg( \sup_{j \in [n]}   \widetilde \Gamma_{j}^{4k} > n^\varepsilon \bigg) \bigg] \\
         &\lesssim  n^\varepsilon + \bigg( \E \bigg[  \sup_{j \in [n]}   \widetilde \Gamma_{j}^{8k} \bigg] \cdot \P\bigg( \sup_{j \in [n]}   \widetilde \Gamma_{j}^{4k} > n^\varepsilon \bigg) \bigg)^{1/2} \\
         & \lesssim n^\varepsilon + n^{1+\varepsilon -D/2} .
    \end{aligned} \label{eq:vfr5}
\end{equation}

Thus, applying the Burkholder inequality (Lemma \ref{lem:burkholder}) to (\ref{eq:MDS}) with (\ref{eq:vfr4})--(\ref{eq:vfr5}), we obtain 
\[\E \big| \bu^\top \bS \bR_0 \bS \bu  - \E[\bu^\top \bS \bR_0 \bS \bu] \big|^k \lesssim n^{1 + \varepsilon - k} + n^{\varepsilon - k/2} . \] 
Since $k \geq 2$ is arbitrary, this implies (\ref{ghjk3}) by Markov's inequality.
\end{proof}


\begin{lemma}\label{lem:decomp_K_j_martingale}
    The kernel matrix has the decomposition  
    \[
    \bK_0 (f) = \bK_{0,-j} (f) + \bDelta_{1,j} + \bDelta_{2,j} + \bDelta_{3,j},
    \]
    where
    \begin{align*}&
        \bDelta_{1,j} \coloneqq \frac{a_1}{n}\bz_j \bz_j^\top,  & \bDelta_{2,j} \coloneqq \frac{1}{\sqrt{n}} \bK_{0,-j} (f' - a_1) \odot (\bz_j \bz_j^\top) ,
    \end{align*}
    and the operator norm of  $\bDelta_3$ satisfies 
    \begin{align*}
        \| \bDelta_{3,j} \|_2 \lesssim  \frac{1}{n} ( 1 \vee \| \bz_{j} \|_\infty^{2L} )  \cdot  \Gamma_{j}.
    \end{align*}
    \end{lemma}

The following concentration inequality for Gaussian quadratic forms, which is a special case of Lemma B.26 from \cite{bai2007asymptotics}, is key to the proof of  Lemma \ref{lem:bound_alpha_j_martingale}:
\begin{lemma}\label{lem:HW}
Let  $\bA \in \mathbb{R}^{p \times p}$ be independent of $\bz_j$. For $k \geq 2$,
\begin{equation*}
\begin{aligned}\big|(\E_j - \E_{j-1}) \bz_j^\top \bA \bz_j\big|^k  & = \big | \E_{j-1} (\bz_j^\top \bA \bz_j - \tr\bA)\big|^k  \leq  \E_{j-1}  \big | \E_{\bz_j} (\bz_j^\top \bA \bz_j - \tr\bA)\big|^k  \\ & \lesssim \E_{j-1} \|\bA\|_F^{k}. \end{aligned}    
\end{equation*}
\end{lemma}

\begin{proof}[Proof of Lemma \ref{lem:bound_alpha_j_martingale}]
The claim follows immediately  from the bounds
\begin{align*}
      & \hspace{2.5cm} \E_{j-1} | \alpha_{j,r}|^k \lesssim \frac{(\log n)^{4Lk}}{n^k} \E_{j-1} \Gamma_{j}^{4k} , & r \in [4] ,  
\end{align*}
which we prove below: 
{\setlist[enumerate,1]{left=0pt}
\begin{enumerate}
    \item By Jensen's inequality and  $\| \bR_0 \|_2 \leq \Im(z)^{-1} \lesssim 1$, 
\begin{equation} \label{dfgh23}
    \begin{aligned}
    \E_{j-1}  |\alpha_{j,1}|^k \lesssim \frac{1}{n^{2k}} \E_{j-1} \big[ (\bz_j^\top \bR_0  \bz_j  )^k | \bu^\top \bz_j|^{2k} \big] \lesssim \frac{1}{n^{2k}} \E_{j-1} \big[ \|\bz_j\|_2^{2k} | \bu^\top \bz_j|^{2k} \big] \lesssim \frac{1}{n^k}.
\end{aligned}\end{equation}
\item Let $\E_{\bz_j}(\cdot)$ denote expectation with respect to $\bz_j$. Since $\E_{j-1}(\cdot) = \E_j\E_{\bz_j}(\cdot)$, 
\begin{equation}
\begin{aligned}
    \E_{j-1} | \alpha_{j,2}|^k \lesssim&~ \frac{1}{n^k} \E_{j} \Big[ \E_{\bz_j}\big| \bz_j^\top \bu \cdot \bz_j^\top \bR_{0,-j} \bS_{-j} \bu \big|^k  \Big]
    \lesssim \frac{1}{n^k} \E_j  \| \bS_{-j} \|_2^k \leq \E_j  \Gamma_{j}^k .
\end{aligned}
\end{equation}
\item  By Lemma \ref{lem:decomp_K_j_martingale} and  $\|\bz_j\|_2 \leq \sqrt{n} \|\bz_j\|_\infty$, 
\begin{equation}
    \begin{aligned} \label{eq:alpha_j3}
    \bz_j^\top (\bR_0 - \bR_{0,-j}) \bS_{-j} \bu =&~ -\bz_j^\top \bR_0 (\bDelta_{1,j} + \bDelta_{2,j} + \bDelta_{3,j})\bR_{0,-j} \bS_{-j} \bu, \\
          \big|\bz_j^\top \bR_0 \bDelta_{1,j} \bR_{0,-j} \bS_{-j} \bu \big| \lesssim &~  \frac{1}{n} \| \bz_j \|_2^2 | \bz_j^\top\bR_{0,-j} \bS_{-j} \bu | \leq  \|\bz_j\|_\infty^2 | \bz_j^\top\bR_{0,-j} \bS_{-j} \bu | , \\
          \big|\bz_j^\top \bR_0 \bDelta_{2,j} \bR_{0,-j} \bS_{-j} \bu \big| = &~   \frac{1}{\sqrt{n}} \Big|\bz_j^\top \bR_0 \diag (\bz_j) \bK_{0,-j} (f' - a_1) \diag (\bz_j) \bR_{0,-j} \bS_{-j} \bu\Big|\\
    \lesssim &~ \frac{1}{\sqrt{n}} \| \bz_j\|_2 \| \bz_j \|_\infty^2 \big\|\bK_{0,-j} (f' - a_1)\big\|_2  \| \bS_{-j} \|_2    \lesssim  \|\bz_j\|_\infty^3 \Gamma_j^2, \\
    \big| \bz_j^\top \bR_0 \bDelta_{3,j} \bR_{0,-j} \bS_{-j} \bu \big| \lesssim &~  \|\bz_j\|_2 \| \bDelta_{3,j} \|_2  \| \bS_{-j} \|_2  \lesssim \frac{1}{\sqrt{n}} ( 1 \vee \| \bz_{j} \|_\infty^{2L+1} )  \Gamma_{j}^2 .
    \end{aligned}
\end{equation}
Moreover, as $\bR_{0,-j}$ and $\bS_{-j}$ are independent of $\bz_j$,
$\E_{\bz_j} |\bz_j^\top\bR_{0,-j} \bS_{-j} \bu|^k \lesssim \|\bR_{0,-j} \bS_{-j} \bu\|_2^k \lesssim \|\bS_{-j}\|_2^k$. Together, these bounds imply 
    \[
    \begin{aligned}
        \E_{j-1}| \alpha_{j,3}|^k \lesssim&~ \frac{1}{n^k} \E_{j-1} \Big| \bu^\top \bz_j \cdot \bz_j^\top (\bR_0 - \bR_{0,-j}) \bS_{-j} \bu \Big|^k \lesssim \frac{(\log n)^{2Lk} }{n^k} \E_{j}  \Gamma_{j}^{2k} .
    \end{aligned}
    \]
\item Since $\bR_0 - \bR_{0,-j} =  -\bR_{0}(\bK_0-\bK_{0,-j})\bR_{0,-j}$, 
\begin{equation}\label{eq:decompo_alpha_4j}
\begin{aligned}
   \hspace{-.5cm} \bu^\top \bS_{-j} (\bR_0 - \bR_{0,-j}) \bS_{-j} \bu = &~ -\bu^\top \bS_{-j} \bR_0 (\bK_{0}-\bK_{0,-j}) \bR_{0,-j} \bS_{-j} \bu \\
    =&~  -\bu^\top \bS_{-j} \bR_{0,-j} (\bK_{0}-\bK_{0,-j} ) \bR_{0,-j} \bS_{-j} \bu \\
    &~ + \bu^\top \bS_{-j} \bR_{0,-j} (\bK_{0} - \bK_{0,-j}) \bR_{0}  (\bK_{0} - \bK_{0,-j}) \bR_{0,-j} \bS_{-j} \bu.
\end{aligned}
\end{equation}
We bound $\bu^\top \bS_{-j} \bR_{0,-j} (\bK_{0}-\bK_{0,-j} ) \bR_{0,-j} \bS_{-j} \bu$ similarly to (\ref{eq:alpha_j3}):
\begin{equation} \label{eq:alpha_j3_2}
    \begin{aligned}
\E_{\bz_j} \big| \bu^\top \bS_{-j} \bR_{0,-j} \bDelta_{1,j}\bR_{0,-j} \bS_{-j} \bu\big|^k  & \lesssim \frac{1}{n^k} \E_{\bz_j} \big| \bz_j^\top \bR_{0,-j} \bS_{-j} \bu\big|^{2k}  \lesssim\frac{1}{n^k} \Gamma_{j}^{2k} , \\
\E_{\bz_j} \big|\bu^\top \bS_{-j} \bR_{0,-j} \bDelta_{3,j}\bR_{0,-j} \bS_{-j} \bu\big|^k & \lesssim  \E_{\bz_j} \| \bS_{-j}\|_2^{2k}  \| \bDelta_{3,j}\|_2^k \lesssim \frac{(\log n)^{2Lk}}{n^k} \Gamma_{j}^{3k}. 
    \end{aligned}
\end{equation}
To analyze the corresponding term involving $\bDelta_{2,j}$, we use the identity
\begin{equation*}
\begin{aligned}
     \bu^\top \bS_{-j} \bR_{0,-j} \bDelta_{2,j}\bR_{0,-j} \bS_{-j} \bu  = &~ \frac{1}{\sqrt{n}}\bu^\top \bS_{-j} \bR_{0,-j} \big( \bK_{0,-j} (f'- a_1) \odot (\bz_j \bz_j^\top) \big)\bR_{0,-j} \bS_{-j} \bu \\
    = &~ \frac{1}{\sqrt{n}} \bz_j^\top \Big( \bK_{0,-j} (f'- a_1) \odot  \big( \bR_{0,-j} \bS_{-j} \bu \bu^\top   \bS_{-j}\bR_{0,-j}\big) \Big)  \bz_j 
\end{aligned}
\end{equation*}
(this follows from $\text{diag}(\bz_j) \bR_{0,-j} \bS_{-j} \bu = \text{diag}(\bR_{0,-j} \bS_{-j} \bu) \bz_j$) and concentration of Gaussian quadratic forms (Lemma \ref{lem:HW}):
\begin{equation}
\begin{aligned} \label{dfgh11}
    &~ \Big|(\E_j - \E_{j-1})  \bu^\top \bS_{-j} \bR_{0,-j} \bDelta_{2,j}\bR_{0,-j} \bS_{-j} \bu \Big|^k \\  
    = &~ \frac{1}{n^{k/2}}  \Big| (\E_j - \E_{j-1}) \Big[ \bz_j^\top \Big( \bK_{0,-j} (f'- a_1) \odot  \big( \bR_{0,-j} \bS_{-j} \bu \bu^\top   \bS_{-j}\bR_{0,-j}\big) \Big)  \bz_j \Big] \Big|^k \\
    \lesssim&~ \frac{1}{n^{k/2}} \E_{j-1} \Big\|  \bK_{0,-j} (f'- a_1) \odot  \big( \bR_{0,-j} \bS_{-j} \bu \bu^\top   \bS_{-j}\bR_{0,-j}\big) \Big\|_F^{k} \\ 
    \lesssim &~ \frac{1}{n^k}  \E_{j-1}  \Big[ \|\bS_{-j}\|_2^{2k} \big\| \sqrt{n}  \bK_{0,-j}(f' - a_1) \big\|_\infty^{k}  \Big] \\ 
    \lesssim&~ \frac{1}{n^{k}} \E_{j-1} \Gamma_{j}^{3k} .
\end{aligned}
\end{equation}
 To obtain the second inequality, we used the following bound: for $\bA \in \mathbb{R}^{p \times p}$ and $\bx \in \mathbb{R}^p$, 
 \[\| \bA \odot (\bx \bx^\top) \|_F \leq \| \bA\|_\infty \| \bx \bx^\top  \|_F = \|\bA\|_\infty \|\bx\|_2^2,\] where  $\|\bA\|_\infty \coloneqq \max_{i,j \in [p]} |\bA_{ij}|$.

\hspace{1.4em}  The second term on the right-hand side of \eqref{eq:decompo_alpha_4j} satisfies
\begin{equation}
    \begin{aligned}
    |\bu^\top \bS_{-j} \bR_{0,-j} (\bK_{0} - \bK_{0,-j}) \bR_{0}  (\bK_{0} - \bK_{0,-j}) \bR_{0,-j} \bS_{-j} \bu | 
    & \lesssim 
    \delta_{1,j} + \delta_{2,j} + \delta_{3,j} , \label{dfgh21}
\end{aligned}
\end{equation}
where
\begin{align*}
    \delta_{1,j} & \coloneqq  \frac{1}{n^2} \| \bz_j \|_2^2 \, \big|\bz_j^\top \bR_{-j} \bS_{-j} \bu \big|^2 ,  \\
    \delta_{2,j} & \coloneqq \frac{1}{n}  \| \bz_j\|_2 \, \big|\bz_j^\top \bR_{-j}\bS_{-j} \bu \big| \,   \|  \bS_{-j}\|_2 \, \big( \| \bDelta_{2,j} \|_2  + \| \bDelta_{3,j} \|_2 \big), \\
    \delta_{3,j} & \coloneqq       \| \bS_{-j}\|_2^2 \big(  \| \bDelta_{2,j}\|_2^2+ \| \bDelta_{3,j}\|_2^2 \big) . 
\end{align*}

By arguments similar to (\ref{eq:alpha_j3})--(\ref{eq:alpha_j3_2}),  we have 
\begin{align} &  \label{dfgh22}
    \E_{\bz_{j}} |\delta_{1,j}|^k  \lesssim \frac{1}{n^k} \Gamma_{j}^{2k}, & \E_{\bz_{j}}  |\delta_{2,j}|^k  \lesssim \frac{(\log n)^{2Lk}}{n^k} \Gamma_{j}^{3k}, && \E_{\bz_{j}} |\delta_{3,j}|^k \lesssim \frac{(\log n)^{4Lk}}{n^k} \Gamma_{j}^{4k}.
\end{align}
Combining the above bounds completes the proof:
\[
\E_{j-1}  | \alpha_{j,4}|^k \lesssim \frac{(\log n)^{4Lk}}{n^k} \E_{j-1} \Gamma_{j}^{4k} .
\]
\end{enumerate} }
\end{proof}

\begin{proof}[Proof of Lemma \ref{lem:bound_Gamma_j}]
 By Theorem 1.6 of \cite{fan2019spectral}, 
 $\| \bK_{0,-j} (h_\ell) \|_2 \prec 1$ for $\ell \in [L]$ 
 and $\| \bS_{-j} \|_2 \prec 1$. Moreover, 
    \[
  \sup_{\ell \in [L]} \sqrt{n} \| \bK_{0,-j} (h_\ell) \|_\infty = \sup_{\ell \in [L]} \sup_{k_1 \neq k_2 \in [n] \setminus \{j\}} \big| h_\ell \big(n^{-1/2} \bz_{k_1}^\top \bz_{k_2}\big) \big| ,
    \]
    and each term satisfies $| h_\ell ( n^{-1/2}\bz_{k_1}^\top \bz_{k_2}  ) | \prec 1$ 
    by Lemma 4.1 of \cite{cheng2013spectrum} and Markov's inequality. Thus, using a union bound, we conclude that $ \Gamma_{j} \prec 1$. 

    To prove the second point of the lemma, we use the bound
    \begin{equation}
    \begin{aligned} \label{dfgh2-1}
        \E  \Gamma_{j}^{2k}  & \lesssim 1 + \E \| \bS_{-j} \|_2^{2k}  + \sum_{\ell =1}^L \E \Big[ \| \bK_{0,-j} (h_\ell) \|_2^{2k} + \| \sqrt{n} \bK_{0,-j} (h_\ell) \|_\infty^{2k} \Big] \lesssim n^{2k}
    \end{aligned}
    \end{equation}
    (which follows from $\| \bS_{-j} \|_2 \leq \| \bS_{-j} \|_F$ and $ \| \bK_{0,-j} (h_\ell) \|_2 \leq  \| \bK_{0,-j} (h_\ell) \|_F$) and the fact that for any $\varepsilon, D > 0$, $\P(|\Gamma_j| > n^{\varepsilon}) \leq n^{-D}$ for sufficiently large $n$:
    \begin{equation}
    \begin{aligned} \label{dfgh2}
        \E  \hspace{.09em} \Gamma_j^k & \leq \E   \big[ \Gamma_j^k  \hspace{.09em} {\bf 1}(\Gamma_j \leq n^\varepsilon) \big] + \E\big[ \Gamma_j^k  \hspace{.09em} {\bf 1}(\Gamma_j > n^\varepsilon) \big] \\
        & \leq n^\varepsilon + \big( \E  \hspace{.09em} \Gamma_j^{2k}  \cdot \P(\Gamma_j > n^{\varepsilon}) \big)^{1/2} \\
        & \lesssim n^\varepsilon + n^{k-D/2}.
    \end{aligned}   
    \end{equation}
    Taking $D \geq 2k$ completes the proof.
    
\end{proof}

\begin{proof}[Proof of Lemma \ref{lem:decomp_K_j_martingale}]
Similarly to (\ref{eq:decompo_K_f_A}), 
\begin{equation}
    \begin{aligned}
    \bK_0 (f) =&~ \bK_{0,-j} (f) + \frac{1}{\sqrt{n}}\bK_{0,-j} (f') \odot (\bz_j\bz_j^\top)  \\
    &~+ \sum_{\ell = 0}^L \sum_{k=2}^{L-\ell} \frac{a_{\ell+k}}{k!\, n^{k/2}} \sqrt{\frac{(\ell +k)!}{\ell!}} \Big[ \bK_{0,-j} (h_\ell) \odot (\bz_j \bz_j^\top)^{\odot k} \Big], \\
    \frac{1}{\sqrt{n}} \bK_{0,-j} (f' ) \odot (\bz_j\bz_j^\top) =&~ \frac{a_1}{n} \bz_j \bz_j^\top + \frac{1}{\sqrt{n}} \bK_{0,-j} (f' -a_1) \odot (\bz_j\bz_j^\top) - \frac{a_1}{n} \diag ( \bz_j^{\odot 2}) \\
=&~ \bDelta_{1,j} + \bDelta_{2,j} - \frac{a_1}{n} \diag ( \bz_j^{\odot 2}) .
\end{aligned}
\end{equation}
Thus, 
 \begin{align*}
\bDelta_{3,j}  &  \coloneqq \bK_0(f) - \bK_{0,-j}(f) - \bDelta_{1,j} - \bDelta_{2,j} \\
& = - \frac{a_1}{n} \diag ( \bz_j^{\odot 2}) + \sum_{\ell = 0}^L \sum_{k=2}^{L-\ell} \frac{a_{\ell+k}}{k!\, n^{k/2}} \sqrt{\frac{(\ell +k)!}{\ell!}} \Big[ \bK_{0,-j} (h_\ell) \odot (\bz_j \bz_j^\top)^{\odot k} \Big] . 
\end{align*}
The lemma now follows from  $\| \diag(\bz_j^{\odot 2})\|_2 = \|\bz_j\|_\infty^2$, our assumption that  $f$ is an odd function (implying $a_2 = 0$), and the bounds
\begin{equation}
    \begin{aligned}
  \frac{1}{n^{k/2}} \big\| \bK_{0,-j}(h_0) \odot (\bz_j \bz_j^\top)^{\odot k} \big\|_2  \leq & \frac{2}{n^{(k-1)/2}} \|\bz_j^{\odot k} \|_\infty^2 \leq \frac{2}{n^{(k-1)/2}} \| \bz_{j}\|_\infty^{2k},  \\
  \frac{1}{n^{k/2}}  \big\| \bK_{0,-j} (h_\ell) \odot (\bz_j \bz_j^\top)^{\odot k} \big\|_2 \leq& \frac{1}{n} \big\| \diag (\bz_j^{\odot k}) \bK_{0,-j} (h_\ell) \diag (\bz_j^{\odot k}) \big\|_2 \\
\leq& \frac{1}{n} \| \bK_{0,-j} (h_\ell)\|_2 \cdot \| \bz_{j} \|_\infty^{2k}.
\end{aligned} \label{eq:b12} 
\end{equation}
%
\end{proof}

\subsection{Proof of Lemma \ref{lem:quad_forms_E2}} \label{sec:quadforms2}

\begin{lemma} \label{lem:quad_forms_E3}
    Under the assumptions of Lemma \ref{lem:quad_forms_E2}, 
    \begin{align*}
 &       \E(\bR_0)_{11} - s(z) = O(n^{-1/2 + \varepsilon}) , & \E(\bS\bR_0)_{11} - \breve s(z) = O(n^{-1/2 + \varepsilon})  .  
    \end{align*}
    Moreover, the convergence is uniform in $z$ on compact subsets disjoint from $\mathrm{supp}(\mu)$.
\end{lemma}

\begin{proof}
    The first claim follows from the proof of Theorem 3.4 of \cite{cheng2013spectrum} and Theorem 1.6 of \cite{fan2019spectral}. The proof of the second claim is similar to that of Lemma 1 of \cite{liao2020sparse}.
\end{proof}

For $j \in [p]$, let $\bz_{(j)}$ denote the $j$-th column of $\bZ$ and $\bZ_{(-j)}$ contain the remaining columns. 
Then, 
\begin{equation}
    \begin{aligned}\label{eq:block_inverse_R_0}
(\bR_{0})_{11} &=  -\Big( z + n^{-1} f ( n^{-1/2} \bz_{(1)}^\top \bZ_{(-1)}^\top ) \big((\bK_{0})_{[2:p, 2:p]} - z\bI_{p-1}\big)^{-1}  f ( n^{-1/2} \bZ_{(-1)} \bz_{(1)}) \Big)^{-1}  ,\\ 
   (\bR_{0})_{12} &=  - \frac{1}{\sqrt{n}} (\bR_0)_{11} \cdot    f ( n^{-1/2} \bz_{(1)}^\top \bZ_{(-1)}^\top ) \big((\bK_{0})_{[2:p, 2:p]} - z\bI_{p-1}\big)^{-1}  \be_2 ,
   \end{aligned}
\end{equation}
where $f$ is applied elementwise.

\begin{proof}[Proof of (\ref{eq:expectation_quad_form_m})]
       By exchangeability, \begin{align} \label{dfgh}
    \E [ \bu^\top \bR_0 \bu] &=  \E(\bR_0)_{11} + ( \< \bu, \ones_p\>^2 - 1) \E (\bR_0)_{12} .
   \end{align} 
   Since $f$ is an odd function,  (\ref{eq:block_inverse_R_0}) implies $(\bR_0)_{12}$ is odd with respect to $\bz_{(1)}$; therefore, by the symmetry of the Gaussian distribution,
    \[ \E (\bR_0)_{12} = \E \big[ \E\big[(\bR_0)_{12} | \bZ_{(-1)}\big] \big] = 0 . \]
    The claim now follows Lemma  \ref{lem:quad_forms_E3}.
\end{proof}

\begin{proof}[Proof of (\ref{eq:expectation_quad_form_s1})]
The proof is similar to that of (\ref{eq:expectation_quad_form_m}):
\begin{equation}
    \begin{gathered}
\E[ \bu^\top \bS \bR_0 \bu ] = \E  (\bS \bR_0 )_{11} + ( \< \bu, \ones_p\>^2 - 1) \E (\bS \bR_0)_{12} , \\
(\bS \bR_0)_{12} = \frac{1}{n}\sum_{j =1}^p \langle \bz_{(1)}, \bz_{(j)} \rangle (\bR_0)_{j2} = \frac{1}{n} \langle \bz_{(1)}, \bz_{(2)} \rangle 
 (\bR_0)_{22} + \frac{1}{n} \sum_{j \neq 2} \langle \bz_{(1)}, \bz_{(j)} \rangle  (\bR_0)_{j2}.
\end{gathered} 
\end{equation}
Since  $(\bR_0)_{22}$ and $(\bR_0)_{j2}$ are even and odd functions of $\bz_{(2)}$, respectively, 
\[ \E (\bS \bR_0)_{12}  = \E \big[ \E \big[  (\bS \bR_0)_{12} | \bZ_{(-2)}\big]  \big] = 0 . \]
\end{proof}

\begin{proof}[Proof of (\ref{eq:expectation_quad_form_s2})]
 By exchangeability,
 \begin{equation}
\begin{aligned} \label{dfgh3}
    \E[ \bu^\top \bS \bR_0 \bS \bu]  = &~ \frac{1}{n^2}\sum_{i=1}^n \E[ ( \bu^\top  \bz_i)^2 \cdot \bz_i^\top \bR_0 \bz_i  ] + \frac{1}{n}\sum_{i=1}^n \E[\bu^\top \bz_i  \cdot \bz_i^\top  \bR_0  \bS_{-i} \bu ]  \\
   = &~  \frac{1}{n} \E[ (\bu^\top  \bz_1)^2 \cdot \bz_1^\top \bR_0 \bz_1 ] +  \E[  \bu^\top \bz_1 \cdot \bz_1^\top \bR_0  \bS_{-1} \bu] \\ = &~ \frac{1}{n} \E[\bz_1^\top \bR_0 \bz_1] +
    \frac{1}{n} \E\Big[ (\bu^\top  \bz_1)^2  \big( \bz_1^\top \bR_0 \bz_1 -  \E[\bz_1^\top \bR_0 \bz_1] \big) \Big]  \\
    &~ + \E[  \bu^\top \bR_{0,-1} \bS_{-1} \bu] +  \E\Big[  \bu^\top \bz_1 \cdot \bz_1^\top (\bR_0 - \bR_{0,-1})  \bS_{-1} \bu \Big] ,
\end{aligned}
\end{equation}
where the final equality holds as $\E (\bu^\top \bz_1)^2 = 1$ and $\E [ \bu^\top \bz_1 \cdot \bz_1^\top \bR_{0,-1} \bS_{-1} \bu] = \E[  \bu^\top \bR_{0,-1} \bS_{-1} \bu] $. 
We will consider each of the terms on the right-hand side (\ref{dfgh3}), beginning with the first and third: using \eqref{eq:expectation_quad_form_s1}, 
\begin{equation}\label{eq:exp_zRz}
\begin{aligned}
&\frac{1}{n} \E[\bz_1^\top \bR_0 \bz_1] = \frac{p}{n} \cdot \frac{1}{p} \E \hspace{.05cm} \tr(\bS\bR_0) =  \gamma \breve s(z)  + O(n^{-1/2+\eps}) , \\
& \E[  \bu^\top \bR_{0,-1} \bS_{-1} \bu]  =   \breve s(z)  + O(n^{-1/2+\eps}) . 
 \end{aligned}
\end{equation}
The second term is negligible: an argument similar to the proof of Lemma \ref{lem:quad_forms_E} yields 
    \begin{align}\label{eq:stoch_dom_zRz}
& \frac{1}{n}\bz_1^\top \bR_0 \bz_1 - \frac{1}{n}\E[\bz_1^\top \bR_0 \bz_1]  = O_\prec (n^{-1/2}) ,  && \frac{1}{n^2} \E \big| \bz_1^\top \bR_0 \bz_1 -\E[\bz_1^\top \bR_0 \bz_1] \big|^2  = O (n^{-1+ \varepsilon})   , 
\end{align}
for any $\varepsilon > 0$, implying
\begin{equation}\label{eq:s2_exp_term1}
\begin{aligned}
     \frac{1}{n} \E\Big[ (\bu^\top  \bz_1)^2 \big( \bz_1^\top \bR_0 \bz_1 -  \E[\bz_1^\top \bR_0 \bz_1] \big) \Big] \lesssim &~ \frac{1}{n} \Big(\E \big| \bz_1^\top \bR_0 \bz_1 - \E[\bz_1^\top \bR_0 \bz_1] \big|^2\Big)^{1/2}  \\
     = &~  O (n^{-1/2+\eps}).
\end{aligned}
\end{equation} 

We expand the fourth term of (\ref{dfgh3}) using  Lemma \ref{lem:decomp_K_j_martingale}:
\begin{equation}
    \begin{aligned} \label{dfgh4} \hspace{-.2em}
\E \Big[ \bz_1^\top (\bR_0 - \bR_{0,-1}) \bS_{-1} \bu \bu^\top \bz_1 \Big] =&~   -\E \Big[   \bz_1^\top \bR_0 (\bDelta_{1,1}+ \bDelta_{2,1} + \bDelta_{3,1}) \bR_{0,-1} \bS_{-1} \bu \bu^\top \bz_1\Big].
\end{aligned}
\end{equation}
Denoting $\eta_1 \coloneqq \bz_1^\top \bR_{0,-1} \bS_{-1} \bu \bu^\top \bz_1 -  \bu^\top \bR_{0,-1} \bS_{-1} \bu$ , we have $\E|\eta_1|^2 = O(n^\varepsilon)$ by Lemma \ref{lem:HW} and $\|\bS_{-1}\|_2 \prec 1$.  Using this bound, the Cauchy-Schwarz inequality,  (\ref{eq:exp_zRz}), and (\ref{eq:stoch_dom_zRz}), we find that the  component of (\ref{dfgh4}) involving $\bDelta_{1,1}$ is approximately $-a_1 \gamma \breve s^2(z)$: 
\begin{equation*}
\begin{aligned}
     \E \Big[\bz_1^\top \bR_0 \bDelta_{1,1}\bR_{0,-1} \bS_{-1} \bu \bu^\top \bz_1\Big]  =  &~ \frac{a_1}{n} \E \Big[  \bz_1^\top \bR_0 \bz_1 \cdot \bz_1^\top \bR_{0,-1} \bS_{-1} \bu \bu^\top \bz_1  \Big] \\
      = &~  a_1\E \Big[\big(n^{-1} \bz_1^\top \bR_0 \bz_1 - \gamma \breve s(z) \big) \bu^\top \bR_{0,-1} \bS_{-1} \bu  \Big] + a_1 \gamma \breve s(z) \E \hspace{.09em}\eta_1 \\ 
      &~ +a_1\E \big[\big(n^{-1} \bz_1^\top \bR_0 \bz_1 - \gamma \breve s(z) \big) \eta_1   \big] + a_1 \gamma \breve s(z) \E \big[ \bu^\top \bR_{0,-1} \bS_{-1}\bu\big] \\
      =&~ a_1 \gamma \breve s^2(z) + O(n^{-1/2 + \varepsilon}).
\end{aligned}
\end{equation*}
Furthermore, the final line of (\ref{eq:alpha_j3}) implies 
\begin{equation}
\begin{aligned}
  \E \Big[\bz_1^\top \bR_0 \bDelta_{3,1}\bR_{0,-1} \bS_{-1} \bu \bu^\top \bz_1\Big] &  \lesssim  \Big(\E \big|\bz_1^\top \bR_0 \bDelta_{3,1}\bR_{0,-1} \bS_{-1} \bu\big|^2 \Big)^{1/2} \\
  &\lesssim  \frac{1}{\sqrt{n}} \big( (\log n)^{4L+2} \hspace{.09em}\E \hspace{.09em} \Gamma_1^4  \big)^{1/2}
  \\  &= O(n^{-1/2 + \varepsilon}) .    
\end{aligned}    
\end{equation}

It remains to prove that the component of (\ref{dfgh4}) involving $\bDelta_{2,1}$ is negligible,
\begin{align} \label{qazwsx}
    \E \Big[ \bz_1^\top \bR_{0} \bDelta_{2,1}\bR_{0,-1} \bS_{-1} \bu \bu^\top \bz_1\Big] = O(n^{-1/2 + \varepsilon}) , 
\end{align} which we address below. 
From (\ref{dfgh4})--(\ref{qazwsx}), we obtain 
\[
  \E \Big[   \bz_1^\top (\bR_0 - \bR_{0,-1})  \bS_{-1} \bu \bu^\top \bz_1\Big]  = -a_1 \gamma \breve s^2(z) + O(n^{-1/2 + \varepsilon}) . 
\]
Together with (\ref{dfgh3})--(\ref{eq:s2_exp_term1}), this equation yields
\[
\E[\bu^\top \bS \bR_0 \bS \bu ] =\gamma \breve s(z)  + \breve s(z) - a_1 \gamma \breve s^2(z)  + O(n^{-1/2+\eps}) , 
\]
completing the proof of the lemma. 
\end{proof}
  
\begin{proof}[Proof of (\ref{qazwsx})]
By Lemma \ref{lem:decomp_K_j_martingale},
\begin{equation*}
    \begin{aligned}
        \E \Big[\bz_1^\top \bR_{0} \bDelta_{2,1}\bR_{0,-1} \bS_{-1} \bu \bu^\top \bz_1 \Big]
        = &~  \E \Big[\bz_1^\top \bR_{0,-1} \bDelta_{2,1}\bR_{0,-1} \bS_{-1} \bu \bu^\top \bz_1 \Big] \\
        &~ -  \E \Big[\bz_1^\top \bR_{0,-1}(\bDelta_{1,1}+\bDelta_{2,1}+\bDelta_{3,1}) \bR_{0} \bDelta_{2,1}\bR_{0,-1} \bS_{-1} \bu \bu^\top \bz_1 \Big] . 
    \end{aligned}
\end{equation*}
Using the bounds  $\|\bDelta_{2,1}\|_2 \leq n^{-1/2} \|\bz_1\|_\infty^2 \Gamma_1$ and  $\|\bDelta_{3,1}\|_2 \leq n^{-1}(1\vee \|\bz_1\|_\infty^{2L}) \Gamma_1$ and the Cauchy-Schwarz inequality, we find that the terms of this expansion involving $\bDelta_{2,1}$ and $\bDelta_{3,1}$ are negligible: 
\begin{equation}
    \begin{aligned}
        \E \Big[\bz_1^\top \bR_{0,-1}\bDelta_{2,1} \bR_{0} \bDelta_{2,1}\bR_{0,-1} \bS_{-1} \bu \bu^\top \bz_1 \Big] & \lesssim \Big(\E \Big|\bz_1^\top \bR_{0,-1}\bDelta_{2,1} \bR_{0} \bDelta_{2,1}\bR_{0,-1} \bS_{-1} \bu \Big|^2 \Big)^{1/2}\\
        & \lesssim \frac{1}{\sqrt{n}}\big(\E \big[\|\bz_1\|_\infty^{10} \hspace{.09em}\Gamma_1^6 \big]\big)^{1/2}  = O(n^{-1/2+\varepsilon}) , \\
           \E \Big[\bz_1^\top \bR_{0,-1}\bDelta_{3,1} \bR_{0} \bDelta_{2,1}\bR_{0,-1} \bS_{-1} \bu \bu^\top \bz_1 \Big] & \lesssim \Big(\E \Big|\bz_1^\top \bR_{0,-1}\bDelta_{3,1} \bR_{0} \bDelta_{2,1}\bR_{0,-1} \bS_{-1} \bu \Big|^2 \Big)^{1/2}\\
        & \lesssim \frac{1}{n}\big(\E \big[\|\bz_1\|_\infty^{4L+6} \hspace{.09em}\Gamma_1^6 \big]\big)^{1/2}  = O(n^{-1/2+\varepsilon}) .
    \end{aligned}
\end{equation}
However, the corresponding term containing $\bDelta_{1,1}$ is significant: by (\ref{eq:exp_zRz}) and (\ref{eq:stoch_dom_zRz}),
\begin{align*}
   \E \Big[ \bz_1^\top \bR_{0,-1} \bDelta_{1,1} \bR_0 \bDelta_{2,1}\bR_{0,-1} \bS_{-1} \bu \bu^\top \bz_1  \Big] = a_1 \gamma s(z) \E \Big[ \bz_1^\top  \bR_0 \bDelta_{2,1}\bR_{0,-1} \bS_{-1} \bu \bu^\top \bz_1 \Big] + O(n^{-1/2+\varepsilon}). 
\end{align*}
Thus,
\begin{equation}
    \begin{aligned} \label{q1w5}
        &~ \E \Big[\bz_1^\top \bR_{0} \bDelta_{2,1}\bR_{0,-1} \bS_{-1} \bu \bu^\top \bz_1 \Big] \\
        = &~ \frac{1}{1+a_1 \gamma s(z)}  \E \Big[\bz_1^\top \bR_{0,-1} \bDelta_{2,1}\bR_{0,-1} \bS_{-1} \bu \bu^\top \bz_1 \Big] + O(n^{-1/2+\varepsilon}). 
    \end{aligned}
\end{equation}

We proceed by expanding the expectation on the right-hand side as
\begin{equation}
    \begin{aligned} \label{dfgh5}
        &~ \E \Big[\bz_1^\top \bR_{0,-1} \bDelta_{2,1}\bR_{0,-1} \bS_{-1} \bu \bu^\top \bz_1 \Big] \\ = &~ \frac{1}{\sqrt{n}}\sum_{i,j,k,\ell =1}^p \E \Big[z_{1i} z_{1j} z_{1k}z_{1\ell} (\bR_{0,-1})_{ik} (\bK_{0,-1}(f'-a_1))_{k\ell} (\bR_{0,-1}\bS_{-1} \bu \bu^\top)_{\ell j} \Big] .
    \end{aligned}
\end{equation}
Now, for $\bz_1 \sim \mathcal{N}(0, \bI_p)$ and an array 
 $(a_{ijk\ell} :  i,j,k,\ell \in [p] )$,
\begin{align*}
\sum_{i,j,k,\ell=1}^p  a_{ijk\ell} \E [z_{1i} z_{1j} z_{1k}z_{1\ell} ] &= \sum_{i,j=1}^p (a_{iijj} + a_{ijij} + a_{ijji} ) .
\end{align*} Applying this identity to 
(\ref{dfgh5}),  we obtain
\begin{equation*}
    \begin{aligned} 
        \sqrt{n} \E \Big[ \bz_1^\top \bR_{0,-1} \bDelta_{2,1}\bR_{0,-1} \bS_{-1} \bu\bu^\top \bz_1  \Big] = &~ \sum_{i,j = 1}^p \E \Big[ (\bR_{0,-1})_{ij}\big(\bK_{0,-1}(f'-a_1)\big)_{jj} (\bR_{0,-1}\bS_{-1} \bu \bu^\top)_{ii}\Big]\\ 
        &~ + \sum_{i,j=1}^p\E \Big[  (\bR_{0,-1})_{ii} \big(\bK_{0,-1}(f'-a_1)\big)_{ij} (\bR_{0,-1}\bS_{-1} \bu \bu^\top )_{jj}\Big] \\
        &~ + \sum_{i,j=1}^p \E \Big[ (\bR_{0,-1})_{ij} \big(\bK_{0,-1}(f'-a_1)\big)_{ij} (\bR_{0,-1}\bS_{-1} \bu \bu^\top )_{ij} \Big],
    \end{aligned}
\end{equation*}
with the first summation equal to zero as $\text{diag}( \bK_{0,-1}(f'-a_1)) = 0$.

Since $\bK_{0,-1} (f' -a_1)_{ij}$ and $(\bR_{0,-1})_{ii}$ are even functions of $\bz_{(i)}$ and $\bz_{(j)}$  (recall that $f'-a_1$ is even), and  $(\bR_{0,-1}\bS_{-1})_{j\ell}$ is an odd function of $\bz_{(\ell)}$ unless $j = \ell$,
\begin{equation}
    \begin{aligned} \label{q1w2}
          &~ \sum_{i,j=1}^p\E \Big[  (\bR_{0,-1})_{ii} \big(\bK_{0,-1}(f'-a_1)\big)_{ij} (\bR_{0,-1}\bS_{-1} \bu \bu^\top )_{jj}\Big] \\
         = &~ \sum_{i,j,\ell=1}^p u_j u_\ell  \E \Big[ (\bR_{0,-1})_{ii} \big(\bK_{0,-1}(f'-a_1)\big)_{ij} (\bR_{0,-1}\bS_{-1})_{j\ell}  \Big] \\
         = &~  \sum_{i,j=1}^p u_j^2 \E \Big[ (\bR_{0,-1})_{ii} \big(\bK_{0,-1}(f'-a_1)\big)_{ij} (\bR_{0,-1} \bS_{-1})_{jj}  \Big] \\
         = &~ (p-1) \hspace{.09em}\E \Big[ (\bR_{0,-1})_{11} \big(\bK_{0,-1}(f'-a_1)\big)_{12} (\bR_{0,-1} \bS_{-1})_{22}  \Big] .
    \end{aligned}
\end{equation}
Recall that $|(\bR_{0,-1})_{ii}  - s(z)| \prec n^{-1/2}$ and $|(\bR_{0,-1} \bS_{-1})_{ii}  - \gamma \breve s(z)| \prec n^{-1/2}$ by Lemma \ref{lem:quad_forms_E}, (\ref{eq:expectation_quad_form_m}), and (\ref{eq:expectation_quad_form_s1}); the corresponding bounds 
\begin{align}
    &~ \E |(\bR_{0,-1})_{ii}  - s(z)|^k \lesssim n^{-k/2 + \varepsilon} , && \E \big|(\bR_{0,-1} \bS_{-1})_{ii}  - \gamma \breve s(z)\big|^k \lesssim n^{-k/2 + \varepsilon} ,
\end{align}
are established through an analogous argument to the proof of Lemma \ref{lem:bound_Gamma_j}.
Thus,
\begin{equation}
    \begin{aligned}
      &~ \E \Big| \big( (\bR_{0,-1})_{11} - s(z) \big) \big(\bK_{0,-1}(f'-a_1)\big)_{12} (\bR_{0,-1} \bS_{-1})_{22}  \Big| \\
      \lesssim &~ \frac{1}{n^{1/2-\varepsilon}} \Big( \E  \Big| \big(\bK_{0,-1}(f'-a_1)\big)_{12} (\bR_{0,-1} \bS_{-1})_{22}  \Big|^2 \Big)^{1/2} \\
       \lesssim &~ \frac{1}{n^{1-\varepsilon}} \big( \E \hspace{.09em} \Gamma_1^4 \big)^{1/2} \lesssim  \frac{1}{n^{1-\varepsilon}} , \\
       &~   \E \Big| \big(\bK_{0,-1}(f'-a_1)\big)_{12} \big((\bR_{0,-1} \bS_{-1})_{22} - \gamma \breve s(z) \big)\Big|  \\
         \lesssim &~ \frac{1}{n^{1/2-\varepsilon}} \Big( \E  \big| \big(\bK_{0,-1}(f'-a_1)\big)_{12}  \big|^2 \Big)^{1/2}  \lesssim  \frac{1}{n^{1-\varepsilon}} .
    \end{aligned}
\end{equation}

Applying these bounds to (\ref{q1w2}),  we obtain 
\begin{equation}
\begin{aligned} \label{q1w3}
     &~ \sum_{i,j=1}^p\E \Big[  (\bR_{0,-1})_{ii} \big(\bK_{0,-1}(f'-a_1)\big)_{ij} (\bR_{0,-1}\bS_{-1} \bu \bu^\top )_{jj}\Big] \\ = &~ \gamma s(z) \breve s(z) (p-1) \cdot \E \big( \bK_{0,-1}(f'-a_1)\big)_{12} + O(n^\varepsilon) = O(n^\varepsilon) , 
\end{aligned}
\end{equation}
where the second equality follows from Section 4.1 of \cite{cheng2013spectrum}: \[
\sqrt{n} \hspace{.09em} \E \big( \bK_{0,-1}(f'-a_1)\big)_{12} =  \langle f'-a_1 , h_0 \rangle_\phi + O(p^{-1}) =  O(p^{-1}) . \]
Similarly, 
\begin{equation}
    \begin{aligned}  \label{q1w4}
    &~     \bigg| \sum_{i,j=1}^p \E \Big[ (\bR_{0,-1})_{ij} \big(\bK_{0,-1}(f'-a_1)\big)_{ij} (\bR_{0,-1}\bS_{-1} \bu \bu^\top )_{ij} \Big] \bigg| \\
    \leq &~  (p-1) \hspace{.09em} \E \Big| (\bR_{0,-1})_{12} \big(\bK_{0,-1}(f'-a_1)\big)_{12} (\bR_{0,-1} \bS_{-1})_{12}  \Big| \leq O(n^{-1/2 + \varepsilon}) .
    \end{aligned}
\end{equation}
The claim follows from (\ref{q1w5}), (\ref{q1w3}), and (\ref{q1w4}).
\end{proof}

\subsection{Proof of Lemma \ref{lem:quad_forms_1000}}

We shall prove (\ref{qwerty1}); the proofs of (\ref{qwerty2}) and (\ref{qwerty3}) are similar and omitted.

\begin{lemma}\label{lem:gamma2}
 For each $j \in [n]$, define the random variables
 \begin{align*}
&      \bX_{-j}  \coloneqq \bSigma^{1/2} \bS_{-j} \bSigma^{1/2} - \bS_{-j} , &&
    \bK_{0,-j}^{(m)}  \coloneqq  \bK_{0,-j} \odot (\sqrt{n} \bX_{-j}) , \vspace{-.1in}
 \end{align*}
        \begin{align*}
  \ol \Gamma_{j}  \coloneqq \Gamma_j & \vee  \bigg(\sup_{\ell \in [L]} \sqrt{m} \big\| \bK_{0,-j}^{(m)}(h_\ell) \big\|_2 \bigg)  \vee \bigg(\sup_{\ell \in [L]} \sqrt{n}  \big\| \bK_{0,-j}^{(m)}(h_\ell) \big\|_\infty \bigg) \\ & \vee 
  \Big(\sup_{i \in [p]} \sqrt{n}(\bS_{-j}\bv - \bv)_i \Big) \vee \Big(\sup_{i \in [p]} (\bR_{0,-j})_{ii}\Big) \vee \Big(\sup_{i, k \in [p], i \neq k} \sqrt{n}(\bR_{0,-j})_{ik}\Big). 
    \end{align*}
    Then, for all
    $\eps>0$ and $k \in \mathbb{N}$,
    \begin{align*}
    & \ol \Gamma_{j} \prec 1, &&  \E \hspace{.09em}  \ol \Gamma_{j} \hspace{-.92em}\phantom{\Gamma}^k \lesssim n^\eps.
    \end{align*}
\end{lemma}

\begin{lemma} \label{lem:gamma3}
    For each $j \in [n]$,
    \[
    \big\|\bK_0^{(m)} - \bK_{0,-j}^{(m)} \big\|_2 \lesssim \frac{1}{\sqrt{n}} (1 + |\bv^\top \bz_j|) \cdot (1 \vee \|\bz_j\|_\infty^{2L+1}) \cdot \ol \Gamma_j \hspace{-.92em}\phantom{\Gamma}^2
    \]
\end{lemma}

\noindent Proofs of Lemmas \ref{lem:gamma2} and \ref{lem:gamma3} are provided at the end of this section. 

\begin{proof}[Proof of (\ref{qwerty1})]
    Analogous to Section \ref{sec:quadforms1}, we  express the left-hand side of (\ref{qwerty1}) as the sum of a martingale difference sequence: introducing
  \begin{equation}  \begin{aligned} \nonumber
 \ol \alpha_j & \coloneqq  (\E_j - \E_{j-1}) \Big( \bu^\top \bR_0 \bK_{0}^{(m)} \bR_0 \bu - \bu^\top \bR_{0,-j} \bK_{0,-j}^{(m)} \bR_{0,-j} \bu \Big) ,\\
\ol \alpha_{j,1} & \coloneqq  (\E_j - \E_{j-1}) \bu^\top (\bR_0 - \bR_{0,-j}) \bK_{0}^{(m)}(\bR_0 - \bR_{0,-j}) \bu,\\
    \ol \alpha_{j,2} & \coloneqq (\E_j - \E_{j-1}) \bu^\top \bR_{0,-j} \bK_{0}^{(m)} (\bR_0 - \bR_{0,-j}) \bu ,\\ 
    \ol \alpha_{j,3} & \coloneqq (\E_j - \E_{j-1})  \bu^\top \bR_{0,-j} \big(\bK_{0}^{(m)} - \bK_0 \odot(\sqrt{n} \bX_{-j})\big)\bR_{0,-j} \bu 
    \\ \ol  \alpha_{j,4} & \coloneqq (\E_j - \E_{j-1})     \bu^\top \bR_{0,-j} \big(\bK_0 \odot(\sqrt{n} \bX_{-j})- \bK_{0,-j}^{(m)} \big)\bR_{0,-j} \bu  ,
\end{aligned}
\end{equation}
we have the decomposition
\begin{align}
     \bu^\top \bR_0 \bK^{(m)} \bR_0 \bu  - \E \big[  \bu^\top \bR_0 \bK^{(m)} \bR_0 \bu \big] = \sum_{j=1}^n \ol \alpha_j = \sum_{j=1}^n (\ol \alpha_{j,1} + 2 \ol \alpha_{j,2} + \ol \alpha_{j,3} + \ol \alpha_{j,4}) . 
\end{align}
Given the proof of Lemma \ref{lem:quad_forms_E}, it suffices to establish
\begin{align} \label{dfgh24}
      & \hspace{2.5cm} \E_{j-1} | \ol \alpha_{j,r}|^k \lesssim \frac{(\log n)^{4k}}{n^k} \E_{j-1} \ol\Gamma_j\hspace{-.92em}\phantom{\Gamma}^{4k} , & r \in [4] .   
\end{align}
 For brevity, we bound only $\E_{j-1}|\ol \alpha_{j,2}|^k$ and $\E_{j-1}|\ol \alpha_{j,3}|^k$; the analysis of $\ol \alpha_{j,1}$ and $\ol \alpha_{j,4}$ does not differ substantively from the arguments presented below and in Section \ref{sec:quadforms1}.

{\setlist[enumerate,1]{left=0pt}
\begin{enumerate} \setcounter{enumi}{1}
\item We expand $\ol \alpha_{j,2}$ using the identity $\bR_0 = \bR_{0,-j} - \bR_{0} (\bK_0 - \bK_{0,-j}) \bR_{0,-j}$: 
\begin{equation}
\begin{aligned} \label{dfgh12}
   \ol \alpha_{j,2}   = &~ -(\E_j - \E_{j-1}) \bu^\top \bR_{0,-j} \bK_0^{(m)} \bR_0(\bK_0 - \bK_{0,-j}) \bR_{0,-j} \bu \\ = &~ -(\E_j - \E_{j-1}) (\ol \delta_{j,1} + \ol \delta_{j,2} - \ol \delta_{j,3}), \\
     \ol \delta_{j,1} \coloneqq &~  \bu^\top \bR_{0,-j} \bK_{0,-j}^{(m)} \bR_{0,-j}(\bK_0 - \bK_{0,-j}) \bR_{0,-j} \bu  ,\\
     \ol \delta_{j,2} \coloneqq &~   \bu^\top \bR_{0,-j} \big(\bK_0^{(m)} - \bK_{0,-j}^{(m)}\big) \bR_0(\bK_0 - \bK_{0,-j}) \bR_{0,-j} \bu, \\ 
  \ol \delta_{j,3} \coloneqq  &~   \bu^\top \bR_{0,-j} \bK_{0,-j}^{(m)} \bR_{0,-j}(\bK_0 - \bK_{0,-j}) \bR_0 (\bK_0 - \bK_{0,-j}) \bR_{0,-j} \bu . 
\end{aligned}     \end{equation}
The component of $\ol \alpha_{j,2}$ involving $\ol \delta_{j,1}$ is bounded similarly to 
(\ref{dfgh11}): denoting $\widetilde \bu_j \coloneqq \bR_{0,-j} \bu$,
\begin{equation} 
    \begin{aligned}
        \Big|(\E_j - \E_{j-1}) \widetilde \bu_j^\top \bK_{0,-j}^{(m)} \bR_{0,-j} \bDelta_{j,1} \widetilde \bu_j \Big|^k & =  \frac{1}{n^k} \Big|(\E_j - \E_{j-1}) \bz_j^\top \widetilde \bu_j \widetilde \bu_j^\top \bK_{0,-j}^{(m)} \bR_{0,-j} \bz_j \Big|^k \\
         & \lesssim\frac{1}{n^k} \E_{j-1} \big\| \widetilde \bu_j\widetilde \bu_j^\top \bK_{0,-j}^{(m)} \bR_{0,-j}  \big\|_F^k \\
       & \lesssim   \frac{1}{n^{5k/4}} \E_{j-1}   \ol \Gamma_{j} \hspace{-.92em}\phantom{\Gamma}^k , \\
       \Big|(\E_j - \E_{j-1}) \widetilde \bu_j^\top \bK_{0,-j}^{(m)} \bR_{0,-j} \bDelta_{j,2} \widetilde \bu_j \Big|^k & \lesssim\frac{1}{n^{k/2}} \E_{j-1} \Big\| \bK_{0,-j}(f'-a_1) \odot \big( \widetilde \bu_j \widetilde \bu_j^\top  \bK_{0,-j}^{(m)} \big)  \Big\|_F^k \\
       & \lesssim   \frac{1}{n^{5k/4}} \E_{j-1}   \ol \Gamma_{j} \hspace{-.92em}\phantom{\Gamma}^{2k} , \\
        \E_{j-1} \Big| \widetilde \bu_j^\top \bK_{0,-j}^{(m)} \bR_{0,-j} \bDelta_{j,3} \widetilde \bu_j \Big|^k & \lesssim \frac{(\log n)^{2Lk}}{n^{5k/4}} \E_{j-1}  \ol \Gamma_{j} \hspace{-.92em}\phantom{\Gamma}^{2k} .  
    \end{aligned}
\end{equation}
Together with Lemma \ref{lem:decomp_K_j_martingale}, these bounds yield
\begin{align} \label{xsw1}
\E_{j-1} |\ol \delta_{j,1}|^k = \E_{j-1} \Big|\widetilde \bu_j^\top \bK_{0,-j}^{(m)} \bR_{0,-j}(\bDelta_{j,1} + \bDelta_{j,2} + \bDelta_{j,3}) \widetilde \bu_j\Big|^k \lesssim \frac{(\log n)^{2Lk}}{n^{5k/4}} \E_{j-1}  \ol \Gamma_{j} \hspace{-.92em}\phantom{\Gamma}^{2k} .    
\end{align}

\hspace{1.4em} Now, we consider $\ol \delta_{j,2}$. By Lemma \ref{lem:gamma3} and the independence of $\bz_j$ and $(\hspace{.1em}\ol \Gamma_j, \widetilde \bu_j)$, 
\begin{equation*}
    \begin{aligned}
      \E_{\bz_j}  \Big| \widetilde \bu_j^\top \big(\bK_0^{(m)} - \bK_{0,-j}^{(m)}\big) \bR_0 \bDelta_{j,1}  \widetilde \bu_j\Big|^k & \lesssim \frac{1}{n^{k/2}} \E_{\bz_j} \Big|   \big\|\bK_0^{(m)} - \bK_{0,-j}^{(m)} \big\|_2 \|\bz_j\|_\infty \cdot \bz_j^\top \widetilde \bu_j\Big|^k  \\
      & \lesssim \frac{1}{n^{k}} \E_{\bz_j}\Big| (1+\|\bz_j\|_\infty^{2(L+1)}) \ol\Gamma_j\hspace{-.92em}\phantom{\Gamma}^{2}\cdot (1 + |\bz_j^\top \bv|) \cdot \bz_j^\top \widetilde \bu_j \Big|^k \\
      & \lesssim \frac{(\log n)^{2(L+1)k}}{n^{k}}  \ol\Gamma_j\hspace{-.92em}\phantom{\Gamma}^{2k} , \\
       \E_{\bz_j}  \Big| \widetilde \bu_j^\top \big(\bK_0^{(m)} - \bK_{0,-j}^{(m)}\big) \bR_0 \bDelta_{j,2}  \widetilde \bu_j\Big|^k & \lesssim \frac{1}{n^{k/2}} \E_{\bz_j} \Big|   \big\|\bK_0^{(m)} - \bK_{0,-j}^{(m)} \big\|_2 \big\| \bK_{0,-j}(f'-a_1) \big\|_2 \|\bz_j\|_\infty^2\Big|^k  \\
       & \lesssim \frac{(\log n)^{(2L+3)k}}{n^k}  \ol\Gamma_j\hspace{-.92em}\phantom{\Gamma}^{3k} . 
    \end{aligned}
\end{equation*}
Similarly bounding $|\widetilde \bu_j^\top \big(\bK_0^{(m)} - \bK_{0,-j}^{(m)}\big) \bR_0 \bDelta_{j,3}  \widetilde \bu_j| \lesssim  \big\|\bK_0^{(m)} - \bK_{0,-j}^{(m)} \big\|_2 \|\bDelta_{j,3}\|_2$, we obtain
\begin{align} \label{xsw2}
    \E_{j-1}|\ol \delta_{j,2}|^k \lesssim  \frac{(\log n)^{(4L+1)k}}{n^k}  \E_{j-1}\ol\Gamma_j\hspace{-.92em}\phantom{\Gamma}^{3k} .
\end{align}
\hspace{1.4em} The analysis of $\ol \delta_{j,3}$ is analogous to (\ref{dfgh21})--(\ref{dfgh22}). For example, 
\begin{align*}
    \big| \widetilde \bu_j^\top  \bK_{0,-j}^{(m)} \bR_{0,-j} \bDelta_{j,1} \bR_0 \bDelta_{j,1} \widetilde \bu_j \big| &= \frac{1}{n^2} \big| \widetilde \bu_j^\top  \bK_{0,-j}^{(m)} \bR_{0,-j} \bz_j\big|  |\bz_j^\top \widetilde \bu_j | |\bz_j^\top \bR_0 \bz_j| \\
     & \lesssim  \frac{1}{n} \big\| \bK_{0,-j}^{(m)} \bR_{0,-j} \bz_j\big\|_2  |\bz_j^\top \widetilde \bu_j |   \|\bz_j\|_\infty^2 , \\
     \E_{\bz_j} \big| \widetilde \bu_j^\top  \bK_{0,-j}^{(m)} \bR_{0,-j} \bDelta_{j,1} \bR_0 \bDelta_{j,1} \widetilde \bu_j \big|^k & \lesssim \frac{(\log n)^{2k}}{n^k}\big\|\bK_{0,-j}^{(m)} \bR_{0,-j}\big\|_2^k \lesssim \frac{(\log n)^{2k}}{n^k} \ol\Gamma_j\hspace{-.92em}\phantom{\Gamma}^{k} , \\
     \big| \widetilde \bu_j^\top  \bK_{0,-j}^{(m)} \bR_{0,-j} \bDelta_{j,1} \bR_0 \bDelta_{j,2} \widetilde \bu_j \big| &= \frac{1}{n} \big| \widetilde \bu_j^\top  \bK_{0,-j}^{(m)} \bR_{0,-j} \bz_j\big|  \big|\bz_j^\top \bR_0 \bDelta_{j,2} \widetilde \bu_j\big| \\
     & \lesssim  \frac{1}{n} \big\| \bK_{0,-j}^{(m)} \bR_{0,-j} \bz_j\big\|_2  \big\| \bK_{0,-j}(f'-a_1)  \big\|_2   \|\bz_j\|_\infty^3 , \\
     \E_{\bz_j} \big| \widetilde \bu_j^\top  \bK_{0,-j}^{(m)} \bR_{0,-j} \bDelta_{j,1} \bR_0 \bDelta_{j,2} \widetilde \bu_j \big|^k &  \lesssim \frac{(\log n)^{3k}}{n^k} \ol\Gamma_j\hspace{-.92em}\phantom{\Gamma}^{2k} .
\end{align*}
From these and related bounds, we have 
\begin{align} \label{xsw3} \E_{j-1} |\ol \delta_{j,3}|^k \lesssim \frac{(\log n)^{4Lk}}{n^k} \ol\Gamma_j\hspace{-.92em}\phantom{\Gamma}^{3k} .  \end{align}
Thus, by (\ref{xsw1}), (\ref{xsw2}), and (\ref{xsw3}), 
\begin{align} \E_{j-1} |\ol \alpha_{j,2}|^k \lesssim \frac{(\log n)^{4(L+1)k}}{n^k} \ol\Gamma_j\hspace{-.92em}\phantom{\Gamma}^{3k} .  \end{align}

\item  Since
$\bX - \bX_{-j} = (\sqrt{\lambda}-1)( \bv \bv^\top( \bS - \bS_{-j}) + ( \bS - \bS_{-j}) \bv \bv^\top),$ it suffices to consider
\begin{equation*}
\begin{aligned} 
         \widetilde \bu_j^\top \big( \bK_0 \odot (\sqrt{n} \bv \bv^\top (\bS - \bS_{-j})) \big)\widetilde \bu_j  
     = &~  \widetilde \bu_j^\top \big( \bK_{0,-j} \odot (\sqrt{n} \bv \bv^\top (\bS - \bS_{-j})) \big)\widetilde \bu_j \\
     &~ + \widetilde \bu_j^\top \big( (\bK_0 - \bK_{0,-j}) \odot (\sqrt{n} \bv \bv^\top (\bS - \bS_{-j})) \big)\widetilde \bu_j. 
\end{aligned}
\end{equation*}
We prove below a bound on the first term on the right-hand side; the second is handled in the standard manner by applying Lemma \ref{lem:decomp_K_j_martingale} to $\bK_0 - \bK_{0,-j}$, which yields
\begin{align}\label{xsw6}
    \E_{\bz_j} \Big| \widetilde \bu_j^\top \big( (\bK_0 - \bK_{0,-j}) \odot (\sqrt{n} \bv \bv^\top (\bS - \bS_{-j})) \big)\widetilde \bu_j \Big|^k \lesssim \frac{(\log n)^{2k}}{n^k} + \frac{(\log n)^{2Lk}}{n^{5k/4}} \ol\Gamma_j\hspace{-.92em}\phantom{\Gamma}^{k} .  
\end{align}

\hspace{1.4em} Similarly to (\ref{dfgh11}),
\begin{equation}
\begin{aligned} \label{xsw4}
&~ \Big|    (\E_j - \E_{j-1}) \widetilde \bu_j^\top \big(\bK_{0,-j} \odot (\sqrt{n} \bv \bv^\top (\bS - \bS_{-j})) \big) \widetilde \bu_j  \Big|^k \\
= &~  \frac{1}{n^{k/2}} \Big| (\E_j - \E_{j-1}) \bz_j^\top \bv \bv^\top \big(\bK_{0,-j} \odot (\widetilde \bu_j \widetilde \bu_j^\top) \big) \bz_j   \Big|^k \\ 
\lesssim &~ \frac{1}{n^{k/2}}\E_{j-1} \big\| \bv \bv^\top \big(\bK_{0,-j} \odot (\widetilde \bu_j \widetilde \bu_j^\top) \big) \big\|_F^k  \\
 \lesssim  &~\frac{1}{m^{k/2}n^{k}} \E_{j-1}  \|(\widetilde \bu_j)_{[1:m]} \|_1^k \hspace{.1em} \hspace{.1em} \Gamma_j^{k}  , 
\end{aligned}
\end{equation}
where the final inequality follows from 
\begin{equation*}
    \begin{aligned}
         \big\| \bv \bv^\top \big(\bK_{0,-j} \odot (\widetilde \bu_j \widetilde \bu_j^\top) \big) \big\|_F^2 &= \big|\bv^\top \big(\bK_{0,-j} \odot (\widetilde \bu_j \widetilde \bu_j^\top) \big)^2 \bv \big|\\
         & \leq  \sum_{i,k =1}^m \sum_{\ell=1}^p \Big|v_i v_k (\bK_{0,-j})_{i\ell} (\bK_{0,-j})_{\ell k} (\widetilde \bu_j)_i (\widetilde \bu_j)_k (\widetilde \bu_j)_\ell^2\Big| \\
         & \leq \frac{1}{mn}\sum_{i,k =1}^m \sum_{\ell=1}^p \big|(\widetilde \bu_j)_i (\widetilde \bu_j)_k (\widetilde \bu_j)_\ell^2\big| \cdot \Gamma_j^2 . 
    \end{aligned}
\end{equation*}
Now, to bound $\|(\widetilde \bu_j)_{[1:m]} \|_1$, we use $\|\bu\|_0 \lesssim m$ and Lemma \ref{lem:gamma2}: 
\begin{equation} \label{xsw5}
  \begin{aligned}
     \|(\widetilde \bu_j)_{[1:m]} \|_1 &= \sum_{i=1}^m |(\bR_{0,-j} \bu)_i| = \sum_{i=1}^m \bigg| (\bR_{0,-j})_{ii} u_i +  \sum_{k \in [p], k \neq i} (\bR_{0,-j})_{ik} u_k \bigg| \\
      & \leq \bigg(\sum_{i=1}^m  (\bR_{0,-j})_{ii}^2 \bigg)^{1/2} + \sum_{i=1}^m\bigg(\sum_{k \in \mathrm{supp}(\bu), \hspace{.1em} k \neq i} (\bR_{0,-j})_{ik}^2\bigg)^{1/2} \\ 
      & \lesssim  \sqrt{m} \,\ol \Gamma_j . 
\end{aligned}  
\end{equation}
From (\ref{xsw6})--(\ref{xsw5}), we conclude
\begin{align} \E_{j-1} |\ol \alpha_{j,3}|^k \lesssim \frac{(\log n)^{2Lk}}{n^k} \ol\Gamma_j\hspace{-.92em}\phantom{\Gamma}^{k} . 
\end{align}
\end{enumerate}}

\end{proof}
\begin{proof}[Proof of Lemma \ref{lem:gamma2}]
Recall the identity $\bA \odot (\bx\by^\top ) = \diag(\bx) \bA \, \diag (\by)$ and from the proof of Theorem \ref{thrm:poly1}   
\begin{align} \label{lk9jh} 
\bX_{-j} = c_{-j} \cdot \bv \bv^\top  + (\sqrt{\lambda} - 1) \big(\bv( \bS_{-j} \bv - \bv)^\top + ( \bS_{-j} \bv - \bv) \bv^\top  \big),
\end{align}
where $c_{-j}$ (defined analogously to $c$ in (\ref{r8bc9})) satisfies $|c_{-j}-1| \prec n^{-1/2}$. Used together with (\ref{gk91}) and (\ref{qazwsx9}), these expressions lead to
\begin{align*}
    \big\|\bK_{0,-j}(h_\ell) \odot (\sqrt{n} \bX)\big\|_2 &\lesssim \big\|(\bK_{0,-j}(h_\ell))_{[1:m,1:m]}\big\|_2 +   \sqrt{m} \Big(\sup_{i \in [p]} |(\bS_{-j}\bv-\bv)_i| \Big) \big\| \bK_{0,-j}(h_\ell)\big\|_2  \prec \frac{1}{n^{1/4}}.
\end{align*}
Thus, 
\[
\Big(\sup_{\ell \in [L]} \sqrt{m} \big\| \bK_{0,-j}^{(m)}(h_\ell) \big\|_2 \Big) \vee  \Big(\sup_{i \in [p]} \sqrt{n}(\bS_{-j}\bv - \bv)_i \Big) \prec 1 .
\]
Furthermore, by Theorem \ref{lem:quad_forms}, 
\[  \Big(\sup_{i \in [p]} (\bR_{0,-j})_{ii}\Big) \vee \Big(\sup_{i, k \in [p], i \neq k} \sqrt{n}(\bR_{0,-j})_{ik}\Big) \prec 1 .\]
The claim now follows from Lemma \ref{lem:bound_alpha_j_martingale} and an argument analogous to  (\ref{dfgh2-1})--(\ref{dfgh2}).
\end{proof}

\begin{proof}[Proof of Lemma \ref{lem:gamma3}]
   Using (\ref{lk9jh}), 
\begin{equation}
\begin{aligned}
    \bK_0^{(m)} -  \bK_{0,-j}^{(m)} & = \bK_0 \odot (\sqrt{n} \bX - \sqrt{n} \bX_{-j}) + (\bK_0 - \bK_{0,-j}) \odot (\sqrt{n} \bX_{-j}) \\
    & = \frac{1}{\sqrt{n}}(\sqrt{\lambda}-1) \bv^\top \bz_j \cdot \bK_0 \odot (\bv \bz_j^\top +  \bz_j \bv^\top) +  \bDelta_{j,1}^{(m)} + \bDelta_{j,2}^{(m)} + \bDelta_{j,3}^{(m)} , 
\end{aligned}
\end{equation}
where  $\bDelta_{j,3}^{(m)} = \bDelta_{j,3} \odot (\sqrt{n} \bX_{-j}) + (c-c_{-j})\bK_0 \odot (\sqrt{n} \bv \bv^\top)$ and
 \begin{align*}
      &        \bDelta_{j,1}^{(m)}  \coloneqq \frac{ \langle g, h_1 \rangle_\phi}{\sqrt{n}} (\bz_j \bz_j^\top) \odot  \bX_{-j} , &&
        \bDelta_{j,2}^{(m)}   \coloneqq \bK_{0,-j}(g' - \langle g, h_1 \rangle_\phi)  \odot( \bz_j \bz_j^\top) \odot \bX_{-j} .\end{align*}
Using  (\ref{gk91}) and (\ref{qazwsx9}) as in the proof of Lemma \ref{lem:gamma2},
        \begin{align*}
        & \big\| \bDelta_{j,2}^{(m)}\big\|_2 \lesssim \frac{1}{n^{3/4}} (1 \vee \|\bz_j\|_\infty^2) \cdot \ol\Gamma_j\hspace{-.92em}\phantom{\Gamma}^{2} , && \big\|\bDelta_{j,3}^{(m)}\big\|_2 \lesssim  \frac{1}{n^{5/4}} (1 \vee \|\bz_j\|_\infty^{2L}) \cdot\ol\Gamma_j\hspace{-.92em}\phantom{\Gamma}^{2} ,\end{align*}
\vspace{-.9em}
\begin{equation}
\begin{aligned}
    \big\|\bDelta_{j,1}^{(m)}\big\|_2 &\lesssim \frac{1}{\sqrt{n}} \Big(\big\|(\bz_j \bz_j^\top) \odot (\bv \bv^\top) \big\|_2 +  \big\|(\bz_j \bz_j^\top) \odot (\bv (\bS_{-j}\bv-\bv)^\top) \big\|_2 \Big)   
      \\ & \lesssim \frac{1}{\sqrt{n}}\|\bz_j\|_\infty^2 +\|\bz_j\|_\infty^2 \|\bS_{-j} \bv - \bv \|_\infty \prec \frac{1}{\sqrt{n}} \|\bz_j\|_\infty^2 \ol\Gamma_j .   
\end{aligned}  
\end{equation}
Similarly, by Lemma \ref{lem:decomp_K_j_martingale},
\begin{align*}
    \big\| \bK_0 \odot (\bv \bz_j^\top +  \bz_j \bv^\top)\big\|_2 \leq &~   2 \big\| \bK_{0,-j} \odot (\bv \bz_j^\top)\big\|_2 + 2 \big\| (\bK_{0} - \bK_{0,-j}) \odot (\bv \bz_j^\top)\big\|_2 \\
    \lesssim &~  \frac{\|\bz_j\|_\infty}{\sqrt{m}}  \Big( \ol \Gamma_j + \|\bDelta_{j,1}\|_2 + \frac{\|\bz_j\|_\infty^2}{\sqrt{n}} \ol \Gamma_j + \frac{(1 \vee \|\bz_j\|_\infty^{2L})}{n} \ol \Gamma_j \Big)  \\
    \lesssim &~ \frac{1 \vee \| \bz_j\|_\infty^{2L+1}}{\sqrt{m}} \ol \Gamma_j . 
\end{align*}

\end{proof}

\subsection{Proof of Lemma \ref{lem:quad_forms_1000E}}

We shall prove the bounds
\begin{align}
    \label{sdfg1} &
\E \Big[ \bu^\top \bR_0 \big(\bK_0(g) \odot (\sqrt{n} \bv \bv^\top) \big) \bR_0 \bu \Big]= O(n^{-1/2 + \varepsilon}) , \\& 
\label{sdfg3}
\E \Big[ \bu^\top \bR_0 \big(\bK_0(g) \odot (\sqrt{n} \bv \bv^\top \bS) \big) \bR_0 \bu \Big] = O(n^{-1/2 + \varepsilon}) , 
\end{align}
from which
(\ref{qwerty4}) immediately follows; the proofs of (\ref{qwerty5}) and (\ref{qwerty6}) are similar and omitted. For simplicity, we assume $\mathrm{supp}(\bu) \subseteq \mathrm{supp}(\bv) = [m]$.

\begin{proof}[Proof of (\ref{sdfg1})]

Expanding the expectation,
\begin{align} \label{sdfg4}
    \E \Big[ \bu^\top \bR_0 \big(\bK_0(g) \odot (\sqrt{n} \bv \bv^\top) \big) \bR_0 \bu \Big] = \frac{\sqrt{n}}{m} \sum_{i,j,k,\ell = 1}^m u_i u_j \E \big [(\bR_0)_{ik}(\bK_0(g))_{k\ell} (\bR_0)_{\ell j} \big] , 
\end{align}
 we claim only those terms on the right-hand side with $i = j$ have non-zero mean. To see this, suppose $i \neq j$, and recall that (1) $(\bR_0)_{ij}$ and $(\bK_0(g))_{ij}$ are odd functions of $\bz_{(i)}$ and $\bz_{(j)}$  and (2) $(\bR_0)_{ii}$ is an even function of $\bz_{(i)}$ by (\ref{eq:block_inverse_R_0}). Therefore,  unless $(i,j)= (k, \ell)$ or $(i,j) = (\ell,k)$,
 \begin{align*}   \E \big[ (\bR_0)_{ik}(\bK_0(g))_{k\ell} (\bR_0)_{\ell j} \big] = 0 .
 \end{align*}
In the former case, 
\begin{align*}
    \E \big [(\bR_0)_{ik}(\bK_0(g))_{k\ell} (\bR_0)_{\ell j} \big] = 
 \E \big[\E_{\bz_{(i)}} \big[ (\bR_0)_{ii}(\bK_0(g))_{ij} (\bR_0)_{jj} \big] \big]= 0,
\end{align*}
and in the latter,
\begin{align*}
    \E \big [(\bR_0)_{ik}(\bK_0(g))_{k\ell} (\bR_0)_{\ell j} \big] = 
 \E \big[\E_{\bz_{(i)}} \big[ (\bR_0)_{ij}^2(\bK_0(g))_{ij}\big] \big]= 0 . 
\end{align*}

Thus, by exchangeability and the fact that the diagonal of $\bK_0(g)$ is zero,
\begin{equation}
\begin{aligned}
    \E \Big[ \bu^\top \bR_0 \big(\bK_0(g) \odot (\sqrt{n} \bv \bv^\top) \big) \bR_0 \bu \Big] = &~  \frac{\sqrt{n}}{m} \sum_{\substack{i,k,\ell \in [m] \\ k \neq l}} u_i^2 \E \big [(\bR_0)_{ik}(\bK_0(g))_{k\ell} (\bR_0)_{\ell i} \big]  \\
     = &~  \frac{\sqrt{n} (m-1)(m-2)}{m}  \E \big [(\bR_0)_{12}(\bK_0(g))_{23} (\bR_0)_{31} \big] \\
     &~ + \frac{\sqrt{n} (m-1)}{m}\E \big [(\bR_0)_{11}(\bK_0(g))_{12} (\bR_0)_{21} \big]  . 
\end{aligned}    
\end{equation}
To complete the proof, we use the bounds $|(\bR_0)_{12} \vee (\bR_0)_{31} | \prec n^{-1/2}$ and $|(\bR_0)_{11}| \prec 1$ (Theorem \ref{lem:quad_forms}) as well as $|(\bK_0)_{23}| \prec n^{-1/2}$ (Lemma  \ref{lem:bound_Gamma_j}): by a similar argument to (\ref{dfgh2-1})--(\ref{dfgh2}), this implies $\E \big [(\bR_0)_{12}(\bK_0(g))_{23} (\bR_0)_{31} \big] = O( n^{-3/2+ \varepsilon}) $ and $\E \big [(\bR_0)_{11}(\bK_0(g))_{12} (\bR_0)_{21} \big] = O( n^{-1+ \varepsilon})$. Thus, 
\[  \E \Big[ \bu^\top \bR_0 \big(\bK_0(g) \odot (\sqrt{n} \bv \bv^\top) \big) \bR_0 \bu \Big] = O(n^{-1/2 + \varepsilon}) . \]
\end{proof}

\begin{proof}[Proof of (\ref{sdfg3})] Similarly to (\ref{sdfg4}),
\begin{equation} 
\begin{aligned}\label{sdfg2}
 \big( \bR_0 \big[\bK_0(g) \odot (\sqrt{n} \bv \bv^\top \bS) \big] \bR_0\big)_{ij} 
  &= \frac{1}{ \sqrt{n}} \sum_{k,q=1}^m\sum_{\ell=1}^p  v_k v_q  
 (\bR_0)_{ik} ( \bK_0 (g))_{k \ell} (\bR_0)_{\ell j}  \bz_{(\ell)}^\top \bz_{(q)}  .
\end{aligned}
\end{equation}  
If $i \neq j$, the expectation of $ (\bR_0)_{ik} ( \bK_0 (g))_{k \ell} (\bR_0)_{\ell j}  \bz_{(\ell)}^\top \bz_{(q)}$ is zero unless $(k,\ell, q) = (i,j,i)$ or $(k,\ell,q) = (j,i,j)$, which leads to
\begin{equation*} 
\begin{aligned}
\E \big( \bR_0 \big[\bK_0(g) \odot (\sqrt{n} \bv \bv^\top \bS) \big] \bR_0\big)_{ij} =  &~  \frac{1}{m \sqrt{n}}  
\E \Big[ (\bR_0)_{ii} (\bK_0(g))_{ij} (\bR_0)_{jj} \bz_{(i)}^\top \bz_{(j)}\Big] \\
&~ +   \frac{1}{m \sqrt{n}} \E \Big[ (\bR_0)_{ij}^2 (\bK_0(g))_{ij} \bz_{(i)}^\top \bz_{(j)}\Big] .
\end{aligned} 
\end{equation*} 
Since $|\bz_{(i)}^\top \bz_{(j)}| \prec \sqrt{n}$ and $\|\bu\|_1 \leq \sqrt{m}$, 
\begin{equation} 
\begin{aligned} \label{sdfg7}
& \E \big( \bR_0 \big[\bK_0(g) \odot (\sqrt{n} \bv \bv^\top \bS) \big] \bR_0\big)_{ij} = O(n^{-1+\varepsilon}) , \\
& \sum_{\substack{i,j \in [m] \\ i \neq j}} u_i u_j \E \big( \bR_0 \big[\bK_0(g) \odot (\sqrt{n} \bv \bv^\top \bS) \big] \bR_0\big)_{ij} = O(n^{-1/2+\varepsilon})  .
\end{aligned}
\end{equation} 

We now consider  $i = j$. Using the identity \[\E_{\bz_{(\ell)}} \Big[ (\bR_0)_{ik} ( \bK_0 (g))_{k \ell} (\bR_0)_{\ell i}  \bz_{(\ell)}^\top \bz_{(q)}\Big] = {\bf 1}_{\ell = q}  \E \Big[ (\bR_0)_{ik}  ( \bK_0 (g))_{k \ell} (\bR_0)_{\ell i} \|\bz_{(\ell)}\|_2^2  \Big] , \]
and exchangeability, we obtain
\begin{equation}
    \begin{aligned}
        \E \big( \bR_0 \big[\bK_0(g) \odot (\sqrt{n} \bv \bv^\top \bS) \big] \bR_0\big)_{ii} =  &~ \frac{1}{m\sqrt{n}} \sum_{k,\ell=1}^m  \E \Big[ (\bR_0)_{ik}  ( \bK_0 (g))_{k \ell} (\bR_0)_{\ell i} \|\bz_{(\ell)}\|_2^2  \Big] \\
        = &~ \frac{(m-1)(m-2)}{m\sqrt{n}} \E \Big[ (\bR_0)_{12}  ( \bK_0 (g))_{2 3} (\bR_0)_{3 1} \|\bz_{(3)}\|_2^2 \Big] \\
        &~ + \frac{m-1}{m \sqrt{n}} \E \Big[ (\bR_0)_{11}  ( \bK_0 (g))_{12} (\bR_0)_{21} \|\bz_{(1)}\|_2^2 \Big] \\
        &~ +  \frac{m-1}{m \sqrt{n}} \E \Big[ (\bR_0)_{11}  ( \bK_0 (g))_{12} (\bR_0)_{21} \|\bz_{(2)}\|_2^2 \Big] \\
        = &~ O(n^{-1/2 + \varepsilon}). 
    \end{aligned}
\end{equation}
Thus,
\begin{align} \label{sdfg6}
    \sum_{i=1}^m  u_i^2 \E \big( \bR_0 \big[\bK_0(g) \odot (\sqrt{n} \bv \bv^\top \bS) \big] \bR_0\big)_{ii} = O(n^{-1/2+\varepsilon}) .   
\end{align}
The claim follows from (\ref{sdfg7}) and (\ref{sdfg6}). 
\end{proof}

\clearpage

\section{Proof of Lemma \ref{lem:max_eig2}} \label{sec:C}
In \cite{fan2019spectral}, a decomposition of $\bK_0(f)$ is developed: 
\begin{equation}\label{eq:Zhou_QRS}
\bK_0(f) = \bQ(f) + \bR(f) + \bS(f)
\end{equation}
Here, we have adopted the notation of \cite{fan2019spectral}; $\bR$ and $\bS$ are distinct from the resolvent and sample covariance matrices that appear in the body of this paper. Proposition 5.2 and Lemma 6.1 in \cite{fan2019spectral} prove there exists $n_{\varepsilon, D} \in \mathbb{N}$ such that 
\begin{align}&\P(\|\bQ(f)\|_2 > (1+\varepsilon) \lambda_+) \leq n^{-D} , &&  \P(\|\bS(f)\|_2 > n^{-\varepsilon})  \leq n^{-D},  \label{hgf-1} \end{align} 
for all $n \geq n_{\varepsilon,D}$, while Lemma 6.3 proves $\|\bR(f)\|_2 \xrightarrow{a.s.} 0$. Specifically,  $\bR(f)$ is given by
\begin{gather*}
    \bR(f) \coloneqq \sum_{\ell=1}^L \frac{a_\ell}{\sqrt{\ell!}} {\ell \choose 2} \bR_\ell , \\
    (\bR_\ell)_{ik} \coloneqq \begin{dcases}
    \frac{1}{n^{(\ell+1)/2}} \sum_{\bj \in \cJ_{\ell -1}} \bigg[ (z_{ij_1}^2 z_{k j_1}^2 - 1) \prod_{t=2}^{\ell -1} z_{i j_t} z_{k j_t} \bigg] & i \neq k ,\\
    0 & i = k , 
\end{dcases}
\end{gather*}
where $\cJ_{\ell}$ is the set of sequences $\bj = (j_1, \ldots , j_{\ell} ) \in [p]^{\ell}$ without repetitions.
We will generalize the moment method of \cite{fan2019spectral} to prove
\begin{align} \label{hgf1}
\|\bR_\ell\|_2 \prec n^{-1/2}  
\end{align}
for odd indices $\ell$, implying $\|\bR(f)\|_2 \prec n^{-1/2}$.
Together with (\ref{hgf-1}), this establishes Lemma \ref{lem:max_eig2}.

\begin{proof}[Proof of (\ref{hgf1})]
Using the identities
\begin{gather*}    
\left(z_{ij_1}^2 z_{k j_1}^2 - 1 \right)  = (z_{ij_1}^2 - 1)(z_{kj_1}^2 - 1)  + (z_{ij_1}^2 - 1) + (z_{kj_1}^2 - 1),
\end{gather*}
\[
\begin{aligned}
\sum_{\bj \in \cJ_{\ell -1}} \bigg[ (z_{ij_1}^2 - 1) \prod_{t=2}^{\ell -1} z_{i j_t} z_{k j_t} \bigg] =&~ \bigg( \sum_{j=1}^p (z_{ij}^2 - 1) \bigg) \sum_{\bj \in \cJ_{\ell -2}}  \prod_{t=1}^{\ell -2} z_{i j_t} z_{k j_t} \\
&~- (\ell -2) \sum_{\bj \in \cJ_{\ell -2}} \bigg[(z_{ij_1}^2 - 1)z_{ij_1}z_{kj_1} \prod_{t=2}^{\ell -2} z_{i j_t} z_{k j_t} \bigg]
\end{aligned}
\]
(we adopt the convention that $\prod_{t=2}^1 = 1$), we expand $\bR_\ell$ as 
\begin{equation}\label{eq:decomposition_bR}
\bR_\ell = \widetilde \bR_\ell + \bD \bA_\ell + \bA_\ell \bD - \bB_\ell . 
\end{equation}
Here, $\bD$ is a diagonal matrix with
\[
(\bD)_{ii} \coloneqq \frac{1}{\sqrt{n}} \sum_{j=1}^p (z_{ij}^2 - 1)
\]
and $\widetilde \bR_\ell, \bA_\ell,$ and $\bB_\ell$ have zero diagonal and off-diagonal elements
\[
\begin{aligned}
(\widetilde \bR_\ell)_{ik} \coloneqq &~  \frac{1}{n^{(\ell+1)/2}} \sum_{\bj \in \cJ_{\ell -1}} \bigg[(z_{ij_1}^2 - 1)(z_{kj_1}^2 - 1) \prod_{t=2}^{\ell -1} z_{i j_t} z_{k j_t} \bigg], \\
(\bA_\ell)_{ik} \coloneqq &~ \frac{1}{n^{\ell/2}} \sum_{\bj \in \cJ_{\ell -2}}\prod_{t=1}^{\ell -2} z_{i j_t} z_{k j_t} , \\
(\bB_\ell)_{ik} \coloneqq &~  \frac{1}{n^{(\ell+1)/2}} \sum_{\bj \in \cJ_{\ell -2}} \bigg[(z_{ij_1}^2 + z_{kj_1}^2 - 2)z_{ij_1}z_{kj_1} \prod_{t=2}^{\ell -2} z_{i j_t} z_{k j_t} \bigg].
\end{aligned}
\]

We will bound the operator norms of the matrices appearing in \eqref{eq:decomposition_bR}. By the sub-exponentiality of the Chi-squared distribution and a union bound, we have $\| \bD \|_2 
\prec 1$.  
Moreover, Proposition 5.2 of \cite{fan2019spectral} (applied to $h_{\ell-2}$) yields $\sqrt{n}\| \bA_\ell \|_2\prec 1$. Thus, $\| \bD \bA_\ell + \bA_\ell \bD \|_2 \prec n^{-1/2}$. 

To bound $\Tilde \bR_\ell$, we use the moment method: for $k \geq 2$,
\[
\begin{aligned}
& \E  \hspace{.04cm} \tr( \widetilde \bR_\ell^k ) = \frac{1}{n^{k(\ell+1)/2}} \sum_{i_1 \neq  \ldots \neq i_k \in [n]} \sum_{\bj^{(1)}, \ldots, \bj^{(k)} \in \cJ_{\ell -1}} W(i_1,\ldots, i_k, \bj^{(1)}, \ldots, \bj^{(k)}) ,  \\
& W(i_1,\ldots, i_k, \bj^{(1)}, \ldots, \bj^{(k)}) 
 \coloneqq    \E \Bigg[ \prod_{s \in [k]} \bigg( (z_{i_sj^{(s)}_1}^2 - 1)(z_{i_{s+1}j_1^{(s)}}^2 - 1)  \prod_{t=2}^{\ell -1} z_{i_s j^{(s)}_t} z_{i_{s+1} j^{(s)}_t} \bigg)  \Bigg] ,
\end{aligned}\]
with the convention that $i_{k+1} = i_1$.
We follow the $k$-graph approach in \cite{fan2019spectral}. From Definitions 5.3--5.8, let $\cL$ denote a multi-labeling, $[\cL]$ its equivalence class, and $\Delta ([\cL]) \coloneqq  1+ k\ell/2 - m$ the {\it excess} of $[\cL]$, where $m$ is the number of distinct $n$ and $p$ labels in $[\cL]$. Since $\E \hspace{.05em} W(i_1,\ldots, i_k, \bj^{(1)}, \ldots, \bj^{(k)})  \lesssim 1$,
\[
\begin{aligned}
\E  \hspace{.04cm}  \tr( \widetilde \bR_\ell^k ) \lesssim&~ \frac{1}{n^{k(\ell+1)/2}} \sum_{s = 0}^\infty \sum_{[\cL]: \, \Delta([\cL]) = s} \big\vert[\cL] \big\vert  
\lesssim \frac{1}{n^{k(\ell+1)/2}} \sum_{s = 0}^\infty n^{1 + k\ell /2 - s} \lesssim n^{1 - k/2},
\end{aligned}
\]
where the second bound uses that the number of equivalent classes is bounded and $\big\vert[\cL] \big\vert \lesssim n^m$.


By Markov's inequality,
\[
\P ( \| \widetilde \bR_\ell \|_2 \geq n^{-1/2+ \eps} ) \leq \frac{1}{n^{-(1/2 - \eps)k}}  \E  \hspace{.04cm}  \tr( \widetilde \bR^k )\lesssim \frac{1}{n^{k\eps - 1}}.
\]
Taking $k$ to be sufficiently large, we deduce that $\|\widetilde \bR_\ell \|_2 \prec n^{-1/2}$.  The bound $\| \bB_\ell \|_2 \prec n^{-1/2}$ follows from a similar argument.
\end{proof}

\section{Proof of Lemma~\ref{lem:polynomial_approx}} \label{sec:poly approx}

The proof extends the argument of Theorem 1.6 of \cite{fan2019spectral}. We first develop a polynomial approximation of  $f'$; this approximation is uniform on closed intervals in which $f$ is twice differentiable. 
\begin{lemma}\label{lem:nice-poly-approx}
Suppose $f(x)$ is odd and there exist finitely many points $x_1< \ldots <x_k $ such that $f(x)$ is twice-differentiable on $(-\infty, x_1), (x_1, x_2), \ldots, (x_{k-1}, x_k), (x_k, \infty)$. Additionally, assume there exists a constant $c > 0$ such that 
\[   \hspace{4.5cm} \max\big(|f(x)|, \|f'(x)|,|f''(x)|\big) \leq c e^{c |x| }  , \hspace{1.5cm} x  \in \mathbb{R} \backslash\{x_1, \ldots, x_k\} . 
    \]   Given $\eps>0$, let $\Omega_\eps:=\bigcup_{1\le i \le k}(x_i-\eps/k,x_i+\eps/k)$, so that the Lebesgue measure or length of  $\Omega_\eps$ is at most  $\eps$. 
    There exists an odd polynomial $q_\eps$ such that the residual $\kappa(x):=f(x)-q_\eps(x)$ satisfies
    \begin{equation}    \label{eq:nice-residual}
    \begin{aligned}
        \hspace{3.5cm}&  |\kappa'(x)| \leq C \varepsilon e^{C|x|} , \hspace{1.5cm}&& x\in \R\setminus\Omega_\eps , \\
        \hspace{3.5cm}& |\kappa'(x)| \leq C  e^{C|x|} , && x\in \Omega_\eps \setminus \{x_1, \ldots, x_k\} ,
            \end{aligned}
    \end{equation}
    where $C > 0$ is a constant depending only on $c$.
\end{lemma}
\noindent The proof of Lemma \ref{lem:nice-poly-approx} is deferred to Section~\ref{sec:proof-lem:nice-poly-approx}. 
 We stress that the specific conditions that Theorem \ref{thrm:A} and Lemma \ref{lem:nice-poly-approx} place on $f$ are not the focus of this paper and are likely improvable---we have developed a minimal extension of \cite{fan2019spectral} that accommodates soft thresholding.

Lemma~\ref{lem:polynomial_approx} is an immediate consequence of  Lemma~\ref{lem:nice-poly-approx} and the following result:

\begin{theorem}
\label{thm:opnorm-eps-bound}
    Suppose that $\kappa(x)$ is odd, continuous everywhere, and differentiable at all but finitely many points, with the derivative satisfying \eqref{eq:nice-residual} wherever it exists. Denote by $\bK(\kappa)$ the corresponding transformed matrix \eqref{0}.
    Almost surely,
    \[
    \limsup_{n\to\infty} \|\bK(\kappa)\| \lesssim \eps^{1/4} \quad\textrm{as}\quad n,p,m\to\infty.
    \]
\end{theorem}  

\noindent The remainder of this section is devoted to proving Theorem~\ref{thm:opnorm-eps-bound} by adapting the proof of Theorem 1.6 in \cite{fan2019spectral}. 

For brevity, we will use the shorthand $\bK:=\bK(\kappa)$. We write $\bZ\bSigma^{1/2}/\sqrt{n}=\sqrt{\lambda-1}\bxi \bv^\top  + \bN$, where $\bN$ has i.i.d. $\normal(0,1/n)$ entries; and $\bxi$, which collects the factor loadings, has i.i.d. $\normal(0,1/n)$ entries. 
Accordingly, 
\begin{align}
     \bY 
     &= (\lambda-1)\|\bxi\|^2\bv\bv^\top  + \sqrt{\lambda-1}\bv(\bN^\top \bxi)^\top + \sqrt{\lambda-1}(\bN^\top \bxi)\bv^\top  +  \bN^\top \bN.
\end{align}
We have $\|\bv\|_0 \le m$, $\|\bv\|_{\infty}\lesssim m^{-1/2}$, $\|\bxi\|_2\lesssim 1$.  
We first bound the expectation of $\bK$, which is non-zero due to the presence of the spike. 
\begin{lemma}
\label{lem:opnorm-expt}
    Under the conditions of Theorem~\ref{thm:opnorm-eps-bound}, 
    \[
    \|\E [\bK ] \|_2 \lesssim  \eps.
    \]
\end{lemma}
The proof appears below in Section~\ref{sec:proof-lem:opnorm-expt}.
The core of the proof of Theorem 1.6 in \cite{fan2019spectral} is a delicate net argument. The idea, originally due to Lata{\l}a  \cite{latala2005some}, is to construct an explicit net of the sphere where the number of points that have large $\ell^\infty$ norm is small. 

Denote $\breve{\bK} \coloneqq \bK-\E[\bK]$. For $\bx,\by\in \R^{p}$, $\|\bx\|,\|\by\|\le 1$, let $F_{\bx,\by}(\bN) \coloneqq \bx^\top \breve{\bK}\by$, where we treat $\breve{\bK}$ as a function of $\bN$. A key step in \cite{fan2019spectral} is a bound on the gradient of $F_{\bx,\by}$.
\begin{lemma}\label{lem:Fxy-gradient}
    The following holds with probability one:
    \begin{align*}
        \|\nabla_{\bN} F_{\bx,\by}(\bN) \|^2 
        &\le {8\|\bN\|^2\|\by\|_{\infty}}\max_{1\le \ell \le p} \bigg(  \sum_{i\ne \ell} (\kappa'(\sqrt{n}\bY_{i\ell}) )^4 \bigg)^{1/2} \\
        &\quad + 8 (\lambda-1)\|\bxi\|^2 \|\bv\|_{\infty}^2 \max_{1\le \ell \le p} \sum_{i\in \mathrm{supp}(\bv)\setminus\{\ell\}} |\kappa'(\sqrt{n}\bY_{i\ell})|^2 .
    \end{align*}
\end{lemma}
\noindent The proof is given below in Section~\ref{sec:proof-lem:Fxy-gradient}.
As in \cite{fan2019spectral}, we restrict attention to a high-probability subset of matrices $\bN$ on which $F_{\bx,\by}(\bN)$ is Lipschitz. Fix $D>0$. For sufficiently large $C=C(D)$, let $\mathcal{G}\subseteq \R^{n\times p}$ be the subset of matrices $\bN$ satisfying 
\begin{enumerate}
    \item  $\|\bN\| \le {C}(1+\sqrt{\gamma})^2$ , 
    \item $\max_{1\le \ell \le p}\sum_{i=1,i\ne \ell}^p (\kappa'(\sqrt{n}\bY_{i\ell}))^4 \le p {C} \eps$ , 
    \item  $\max_{1\le \ell \le p}\sum_{i\in \mathrm{supp}(\bv)\setminus\{\ell\}} |\kappa'(\sqrt{n}\bY_{i\ell})|^2 \le m{C} \eps$.
\end{enumerate} 
\begin{lemma}\label{lem:bound_event_G}
    Under the conditions of Theorem~\ref{thm:opnorm-eps-bound},
    $\P (\bN\notin \mathcal{G}) \lesssim n^{-D}$.
\end{lemma}
The proof is essentially identical to that of Lemma 3.1 in \cite{fan2019spectral} and is  omitted.
Note that Lemma~\ref{lem:Fxy-gradient} implies that $F_{\bx,\by}$ is Lipschitz on $\cG$, with 
\[
L:= \|F_{\bx,\by}\big|_{\cG}\|^2_{\mathrm{Lip}} \lesssim \|y\|_{\infty}\sqrt{p\eps } + \|\xi\|^2\|\bv\|_{\infty}^2 m\eps \lesssim \|y\|_{\infty}\sqrt{p\eps }.
\]
Here, we used that $\|\bv\|_{\infty}\lesssim m^{-1/2}$ and $\eps<1$ (by assumption), and $\|\xi\|\lesssim 1$ (note that $p\|y\|_{\infty}^2\ge 1$, hence the first term is always the dominant one). 

We now show that $F_{\bx,\by}$ concentrates, using the Gaussian Lipschitz concentration inequality. To this end, recall a well-known result on Lipschitz extension.
\begin{theorem}
    [Kirszbraun] Let $F:\mathcal{G}\to \R^{d_2}$, $\mathcal{G}\subseteq \R^{d_1}$, be $L$-Lipschitz. Then $F$ has an $L$-Lipschitz extension to $\R^{d_1}$:
    \[
    \tilde{F}:\R^{d_1}\to \R^{d_2},\qquad \tilde{F}\big|_{\mathcal{G}}=F .
    \]
    When $d_2=1$ and $\mathcal{G}$ is compact, there is a simple construction:
    \[
    \tilde{F}(x) = \min_{y\in \mathcal{G}}(F(y)+L\|x-y\|) .
    \]
\end{theorem}
Let $\tilde{F}_{\bx,\by}$ be a Lipschitz extension of $F_{\bx,\by}\big|_{\mathcal{G}}$.
\begin{lemma}\label{lem:bound_exp_F_G}
Under the conditions of Theorem~\ref{thm:opnorm-eps-bound}, for $\tilde{D}=\tilde{D}(D)$ that can be made arbitrarily small,
    \[
    \left|\E[\tilde{F}_{\bx,\by}(\bN)]-\E[F_{\bx,\by}(\bN)]\right| \lesssim n^{-\tilde{D}}.
    \]
\end{lemma}
\begin{proof}
    For $\bN\notin\mathcal{G}$
    \[
    |F_{\bx,\by}(\bN)|\le \|\breve{\bK}\|_2 \le \|\breve{\bK}\|_{F},\qquad |\tilde{F}_{\bx,\by}(\bN)| \le F_{\bx,\by}(0) + L\|\bN\|_F. 
    \]
    We have 
    \begin{align*}
        |\E[F_{\bx,\by}(\bN)\mathbf{1}_{\bN\notin \mathcal{G}}]| 
        &\le (\E|F_{\bx,\by}(\bN)|^2)^{1/2} ( \P (\bN\notin \mathcal{G}) )^{1/2} \\
        &\le (\E\|\bK\|^2_F)^{1/2} ( \P (\bN\notin \mathcal{G}) )^{1/2} \lesssim p n^{-D/2} \lesssim n^{-\tilde{D}} .
    \end{align*}  
    Similarly one can bound $|\E[\tilde{F}_{\bx,\by}(\bN)\mathbf{1}_{\bN\notin \mathcal{G}}]| $.
\end{proof}

\begin{lemma}\label{lem:Gauss-Lip}
    For $t>0$,
    \[
    \P \left(\tilde{F}_{\bx,\by}(\bN)\gtrsim n^{-\tilde{D}} + \eps t \right) 
    \le 2e^{-\sqrt{p}t^2/\|\by\|_{\infty}} .
    \]
\end{lemma}
\begin{proof}
    Gaussian Lipschitz concentration.
\end{proof}

Finally, we apply the net argument of \cite{fan2019spectral}. As in Section 2 of \cite{fan2019spectral}, let
\begin{equation}
    D_2^p = \left\{ \by\in \R^p\; : \; \|\by\|\le 1,\; y_i^2\in \{0,1,2^{-1},\ldots,2^{-(K+3)}\}\right\},\qquad K:=\lceil \log_2 p \rceil .
\end{equation}
\begin{lemma}
    [Lemma 3.3 of \cite{fan2019spectral}] For $\bA$ symmetric, $\|\bA\|_2\le 10\max_{\by\in D^p_2} \by^\top \bA \by$.
    \label{lem:ZM-Lemma3.3}
\end{lemma}
For $l=0,1,\ldots,K+3$, define the projection $\pi_{l\setminus(l-1)}$ (resp. $\pi_l$) on dyadic scale $l$ (resp. $\le l$):
\begin{align}
    \pi_{l\setminus (l-1)}(\by)_i = y_i\mathbf{1}_{y_i^2=2^{-l}}, \qquad 
    \pi_{l}(\by)_i = y_i\mathbf{1}_{y_i^2\ge 2^{-l}} .
\end{align}
As shown in \cite[Eq. (7)]{fan2019spectral},
\begin{align}\label{eq:yKy}
    \by^\top \breve{\bK} \by = \sum_{l=0}^{K+3} \pi_l(\by)^\top \breve{\bK} \pi_{l\setminus(l-1)}(\by) + \sum_{l=0}^{K+3} \pi_{l\setminus(l-1)}(\by)^\top \breve{\bK} \pi_{l-1}(\by) .
\end{align}
\begin{lemma}
    [Lemma 3.4 of \cite{fan2019spectral}] There exists an absolute constant $C>0$ such that
    \begin{align}
    &   |\left\{\pi_l(\by)\;:\;\by\in D_2^p \right\}| \le \exp(C(K+4-l)2^l) , && l \in \{0, 1, \ldots, K+3\} . 
    \end{align}  
    \label{lem:FM-Lemma3.4}
\end{lemma}
For ${c}_0$ large enough, set
\begin{equation}
t_l^2 = {c}_0(K+4-l)2^{l/2}\frac{1}{\sqrt{p}}.
\end{equation}
As shown in the proof of Theorem 1.6 in \cite{fan2019spectral},
\begin{align}
    \sum_{l=0}^{K+3}t_l 
    &\le  c_0^{1/2} p^{-1/4}\sum_{l=0}^{K+3}(K+4-l)2^{l/4} 
    \le c_0^{1/2} p^{-1/4}\sum_{l=0}^{K+3}\sum_{j=0}^l 2^{j/4} \nonumber \\
    &\lesssim p^{-1/4}\sum_{l=0}^{K+3} 2^{l/4} 
    \lesssim p^{-1/4}2^{K/4}=p^{-1/4}2^{\frac{1}{4}\lceil \log_2 p \rceil} = O(1).
    \label{eq:sum-tl}
\end{align}

\begin{proof}[Proof of Theorem \ref{thm:opnorm-eps-bound}]
    By Lemma~\ref{lem:ZM-Lemma3.3}, \eqref{eq:yKy}, and \eqref{eq:sum-tl}, it suffices to show that 
    \[
    \sup_{y\in D_2^p}\pi_l(\by)^\top \breve{\bK} \pi_{l\setminus(l-1)}(\by),\; \sup_{y\in D_2^p} \pi_{l\setminus(l-1)}(\by)^\top \breve{\bK} \pi_{l-1}(\by) \lesssim n^{-\tilde{D}} + \eps t_l ,\qquad  \ell \in \{0,1,\ldots,K+3\},
    \]
    with high probability. Without loss of generality, we focus on terms \[\pi_l(\by)^\top \breve{\bK} \pi_{l\setminus(l-1)}(\by) =: F_{\pi_l(\by),\pi_{l\setminus(l-1)}(\by)}(\bN); \] the other terms are similar. Under the high-probability event $F_{\bx,\by}(\bN)=\tilde{F}_{\bx,\by}(\bN)$ for any $\bx,\by$, where $\tilde{F}$ is the Lipschitz extension of ${F}\big|_{\mathcal{G}}$. Consequently, as $\Pr(\bN\notin \cG)\lesssim n^{-D}$,
    \begin{equation}\label{eq:B.9}
        \P\bigg( \sup_{\by\in D_2^p} F_{\pi_l(\by),\pi_{l\setminus(l-1)}(\by)}(\bN) \gtrsim n^{-\tilde{D}} + \eps t_l \bigg) 
        \lesssim 
        \P\bigg( \sup_{\by\in D_2^p} \tilde{F}_{\pi_l(\by),\pi_{l\setminus(l-1)}(\by)}(\bN) \gtrsim n^{-\tilde{D}} + \eps t_l \bigg) + n^{-D} .
    \end{equation}
    Denote 
    \[
    \Pi_l = \{\pi_l(\by)\,:\,\by\in D_2^p\},\quad \Pi_{l\setminus(l-1)} = \{\pi_{l\setminus(l-1)}(\by)\,:\,\by\in D_2^p\},\quad 0\le l \le K+3 .
    \]
    Clearly, $\sup_{\by\in D_2^p} \tilde{F}_{\pi_l(\by),\pi_{l\setminus(l-1)}(\by)}(\bN) \le \sup_{\bx\in \Pi_{l\setminus(l-1)},\bz\in\Pi_l} \tilde{F}_{\bx,\bz}(\bN)$. By Lemma~\ref{lem:FM-Lemma3.4},
    $|\Pi_{l\setminus(l-1)}|\le |\Pi_l|\le \exp(C(K+4-l)2^l)$. Consequently, using a union bound, we bound the first term in the r.h.s. of \eqref{eq:B.9} as
    \begin{align*}
        \P\bigg(\sup_{\bx\in \Pi_{l\setminus(l-1)},\bz\in\Pi_l} \tilde{F}_{\bx,\bz}(\bN)\gtrsim n^{-\tilde{D}}+\eps t_l\bigg) 
        &\le |\Pi_{l\setminus(l-1)}||\Pi_l|\sup_{\bx\in \Pi_{l\setminus(l-1)},\bz\in\Pi_l} \P\Big( \tilde{F}_{\bx,\bz}(\bN)\gtrsim n^{-\tilde{D}}+\eps t_l \Big) \\
        &\lesssim e^{2C(K+4-l)2^l} \sup_{\bx\in \Pi_{l\setminus(l-1)},\bz\in\Pi_l} \exp\Big(-\frac{\sqrt{p}}{\|\bx\|_{\infty}}t_l^2\Big) .
    \end{align*}
    Using $\|\bx\|_\infty=2^{-l/2}$ and $t_l^2=c_0(K+4-l)2^{l/2}\frac{1}{\sqrt{p}}$, the above is $\lesssim e^{(2C-c_0)(K+4-l)2^l}$. 
    Choosing $c_0$ sufficiently large but constant, so that $c_1:=c_0-2C$ is large, we can guarantee that $e^{-c_1(K+4-l)2^l}\le e^{-c_1(\log_2 p + 4 -l)2^l}$ is smaller than $p^{-C}$ for any pre-specified $C>0$. Finally taking a union bound over all $0\le l \le K+3\lesssim \log p$,
    \[
    \P\bigg( \max_{0\le l \le K+3}\sup_{\by\in D_2^p} F_{\pi_l(\by),\pi_{l\setminus(l-1)}(\by)}(\bN) \gtrsim n^{-\tilde{D}} + \eps t_l \bigg) \le (K+3)p^{-C} + n^{-\tilde{D}} \asymp n^{-\tilde{D}'}.
    \]
    As explained before, Theorem~\ref{thm:opnorm-eps-bound} now follows using Lemma~\ref{lem:ZM-Lemma3.3}.
\end{proof}



\subsection{Proof of Lemma~\ref{lem:nice-poly-approx}}
\label{sec:proof-lem:nice-poly-approx}

    The proof is an adaption of that of Theorem 1.4 in \cite{fan2019spectral}, with an additional smoothing step. The latter is necessary because $f'$ is not assumed to be continuous as in \cite{fan2019spectral}. 

    For $\eta\in (0,1)$, let $T_{\eta}:L^2(\phi)\to L^2(\phi)$ be the Gaussian smoothing operator,
    \begin{equation}
        T_\eta[f](x) \coloneqq \frac{1}{\sqrt{2\pi} \eta}\int f(y)e^{-(x-y)^2/(2\eta^2)}dy.
    \end{equation}
    Note that under the conditions of Lemma~\ref{lem:nice-poly-approx}, $T_\eta[f]$ is odd and everywhere twice-differentiable, with $(T_\eta[f])'=T_\eta[f'],(T_\eta[f])''=T_\eta[f'']$ and $|T_\eta[f](x)|,|T_\eta[f'](x)|,|T_\eta[f''](x)|\lesssim e^{c_1|x|}$. 
    We shall show that (1) $T_\eta[f']$ approximates $f'$ well away from its discontinuity points; and (2) $T_\eta[f']$ is uniformly approximated by a polynomial. We note that the second step, which relies on the differentiability of $T_\eta[f']$, follows from \cite{fan2019spectral} immediately.

    We next show that $T_\eta[f']$ approximates $f'$ uniformly on $\R\setminus \Omega_\eps$. Note that on $\Omega_\eps\setminus \{x_i\}_{1\le i \le N}$, we already have $|f'(x)-T_\eta[f'](x)|\lesssim e^{c_1|x|}$. Define $\delta_\varepsilon \coloneqq \varepsilon/k$.     Let $x\in \R\setminus\Omega_\eps$. By definition, $f',f''$ exist on $I_{x,\eps}:=(x-\delta_\eps,x+\delta_\eps)$. By the mean value theorem, for $y\in I_{x,\eps}$,
    \begin{equation}
        \label{eq:smooth-aux1}
    |f'(y)-f'(x)|\le \max_{z\in (x,y)}|f''(z)||x-y|\lesssim e^{C\max\{|x|,|y|\}}|x-y|\lesssim |x-y| e^{C|x| + C|x-y|}.
    \end{equation}
    Decompose
    \begin{align}
        f'(x)-T_\eta[f'](x) = \E[(f'(x)-f'(x+\eta G))\mathbf{1}_{|G|\le \delta_\eps/\eta}] + \E[(f'(x)-f'(x+\eta G))\mathbf{1}_{|G|> \delta_\eps/\eta}] .
    \end{align}
    Using \eqref{eq:smooth-aux1}, the first term satisfies
    \begin{align*}
        |\E[(f'(x)-f'(x+\eta G))\mathbf{1}_{|G|\le \delta_\eps/\eta}]| 
        \lesssim 
        e^{C|x|}\E [e^{C\eta|G|}|\eta G| \mathbf{1}_{|G|\le \delta_\eps/\eta}]
        &\le
        e^{C|x|}\delta_\eps e^{C\delta_\eps} \E[\mathbf{1}_{|G|\le \delta_\eps/\eta}] \\
        &\lesssim e^{C|x|}\delta_\eps \asymp e^{C|x|}\eps,
    \end{align*}
    where we used that $\delta_\eps\asymp \eps \lesssim 1$. For the second term,
    \begin{align*}
        |\E[(f'(x)-f'(x+\eta G))\mathbf{1}_{|G|> \delta_\eps/\eta}]| 
        & \lesssim e^{2C|x|}\E[ e^{C\eta|G|}\mathbf{1}_{|G|> \delta_\eps/\eta} ] \\
        & \leq e^{2C|x|} \left( \E[e^{2C\eta|G|}] \right)^{1/2} \left( \E[\mathbf{1}_{|G|> \delta_\eps/\eta}] \right)^{1/2} \\
        &\lesssim e^{C_1|x|}e^{C_2\eta^2}e^{-C_3(\delta_\eps/\eta)^2} \lesssim e^{C_1|x|}e^{-C_4(\eps/\eta)^2} ,
    \end{align*}
    where the second inequality follows from the Cauchy-Schwarz inequality and the final two inequalities follow from $\eta\lesssim 1$,= and the Gaussian tail behavior. Choosing $\eta=\eta_\eps\asymp \eps/\sqrt{\log(1/\eps)}$, which is $\le 1$ for small $\eps>0$, implies that the above is $\lesssim e^{C_1|x|}\eps$. Combining our bounds for the two terms above, we deduce that for any $x\in \R\setminus\Omega_\eps$, $|f'(x)-T_{\eta_\eps}[f](x)| \lesssim e^{c_1|x|}\eps$. Combining with the coarse bound we had before for $x\in \Omega_\eps$,
    \begin{align}
        |f'(x)-T_{\eta_\eps}[f'](x)| \lesssim e^{C|x|}\eps + e^{C|x|}\mathbf{1}_{\Omega_\eps}(x).
    \end{align}
    for appropriate $C>0$.

    The final step is to approximate $T_{\eta_\eps}[f]$ by a polynomial uniformly. 
    This follows from the argument of \cite[Proof of Theorem 1.4]{fan2019spectral}, which yields a polynomial $q_\eps$ such that $\kappa_1(x) = T_\eta[f](x)-q_\eps(x)$ satisfies $|\kappa_1'(x)|\lesssim \eps  e^{C|x|}$ for all $x\in \R$. The  residual $\kappa(x)=f(x)-q_\eps(x)$ therefore satisfies
    \begin{align*}
        |\kappa'(x)| 
        &\le |f'(x)-(T_\eta[f])'(x)| + |\kappa_1'(x)| \\
        &\lesssim e^{C|x|}\eps + e^{C|x|}\mathbf{1}_{x\in \Omega_\eps} 
    \end{align*}
    for all $x\neq x_1, \ldots, x_k$, completing proof.

\subsection{Proof of Lemma~\ref{lem:opnorm-expt}}
\label{sec:proof-lem:opnorm-expt}

   Let $\bN_1,\ldots,\bN_p\in \R^n$  denote the columns of $\bN$.     Recall that $\bK(\kappa)_{i,i}=0$ and for $i\ne j$,
    \[
    (\bK(\kappa))_{ij} = \frac{1}{\sqrt{n}}\kappa(\sqrt{n}\bY_{ij})= \frac{1}{\sqrt{n}}\kappa\left(\sqrt{n}(\lambda-1)\|\xi\|^2 v_i v_j + \bA_{ij} \right) .
    \]
    where $\bA_{ij}=\sqrt{n} (\sqrt{\lambda-1} v_i \bN_j^\top\bxi + \sqrt{\lambda-1} v_j \bN_i^\top \bxi + \bN_i^\top \bN_j)$. 
    
    Recall that $\kappa$ is continuous and piecewise differentiable, hence the fundamental theorem of calculus holds. For brevity, denote $\eta_{i,j}:=\sqrt{n}(\lambda-1)\|\xi\|^2v_iv_j$, so that $\sqrt{n}\bY_{ij}=\bA_{ij}+\eta_{i,j}$.  Then
    \begin{equation}
    \begin{aligned}
        \kappa(\sqrt{n}\bY_{ij}) - \kappa(\bA_{ij}) &= \int_{[0,\eta_{i,j}]} \kappa'(\bA_{ij} + t)dt \\ &= \int_{[0,\eta_{i,j}]\cap (\Omega_\eps-\bA_{ij})} \kappa'(\bA_{ij} + t)dt + \int_{[0,\eta_{i,j}]\setminus (\Omega_\eps-\bA_{ij})} \kappa'(\bA_{ij} + t)dt ,
    \end{aligned}
    \end{equation}
    where $\Omega_\eps$ is the set from Lemma~\ref{lem:nice-poly-approx}. We can bound each term using \eqref{eq:nice-residual}. Indeed,
    \begin{align*}
        \left|\int_{[0,\eta_{i,j}]\cap (\Omega_\eps-\bA_{ij})} \kappa'(\bA_{ij} + t)dt\right|
        & \le |[0,\eta_{i,j}]\cap \Omega_\eps|\cdot \max_{t\in [0,\eta_{i,j}]\cap (\Omega_\eps-\bA_{ij})} |\kappa'(\bA_{ij} + t)|  \\&
        \lesssim \eps\cdot e^{c_1|\eta_{i,j}|}e^{c_1|\bA_{ij}|} .
        \lesssim e^{c_1|\bA_{ij}|}\eps ,
    \end{align*}
    where we used that $\eta_{i,j}\lesssim 1$ ($\lambda$ is fixed), since $|v_iv_j|\in \{0,1/m\}$.
    As for the other term, 
    \begin{align*}
        \left|\int_{[0,\eta_{i,j}]\setminus (\Omega_\eps-\bA_{ij})} \kappa'(\bA_{ij} + t)dt\right|
        \lesssim \int_{[0,\eta_{i,j}]}(e^{c_1|\bA_{ij} + t|}\eps) dt  \lesssim e^{c_1|\bA_{ij}|}\eps .
    \end{align*}
    It is straightforward to verify that $\E[e^{c_1|A_{i,j}|}]\lesssim 1$. For completeness' sake, let us verify this for the heavier-tailed term $\E[ e^{c_1 \sqrt{n}\bN_i^\top \bN_j}]$. Conditioned on $\bN_i$, $\sqrt{n}\bN_i^\top \bN_j\sim \normal(0,\|\bN_i\|^2)$. Thus, 
    \begin{align*}
        \E[ e^{c_1 \sqrt{n}\bN_i^\top \bN_j}] = \E[e^{ c_1^2\|\bN_i\|^2/2}] = \E[e^{c_1^2 \chi^2(n)/(2n)}] = (1-c_1^2/n)^{-n/2}\approx e^{c_1^2/2},
    \end{align*}
    where $\chi^2(n)$ denotes a chi-squared random variable with $n$ degrees of freedom (note that $\sqrt{n} \bN_i^\top \bN_j$ tends to a $\normal(0,1)$ random variable.)

    Now, note that $\E[\bA_{ij}]=0$, since $\bA_{ij}$ has a symmetric distribution. Thus,
    \begin{align*}
        |\E[\bK_{i,j}]| = \frac{1}{\sqrt{n}}\left| \E\kappa(\sqrt{n}\bY_{ij}) - \E\kappa(\bA_{ij}) \right| \lesssim \frac{\eps}{\sqrt{n}}\E[e^{c_1|\bA_{ij}|}] \lesssim \frac{\eps}{\sqrt{n}} .
    \end{align*}
    Also note that $\sqrt{n}\bY_{ij}$ differs from $\bA_{ij}$ only when $i,j\in \mathrm{supp}(v)$; accordingly, $\E[\bK]$ has at most $m^2\asymp n$ nonzero entries. Thus,
    \begin{align*}
        \|\E \bK\|_2 \le \|\E \bK \|_F \lesssim m\cdot \frac{1}{\sqrt{n}}\eps \lesssim \eps .
    \end{align*}

    

\subsection{Proof of Lemma~\ref{lem:Fxy-gradient}}
\label{sec:proof-lem:Fxy-gradient}

    As in the previous section, denote by $\bN_1,\ldots,\bN_p\in \R^n$ the columns of $\bN$.

   The gradient of an entry ($i\ne j$) with respect to $\bN_\ell$
   \begin{align*}
       \nabla_{\bN_\ell} (\bK_{i,j} )
        = &~ \nabla_{\bN_\ell} \frac{1}{\sqrt{n}}\kappa \Big(
       (\lambda-1)\|\bxi\|^2\sqrt{n} v_iv_j + \sqrt{\lambda-1} \sqrt{n} v_i \langle \bN_j,\bxi\rangle \\
       &~ + \sqrt{\lambda-1} \sqrt{n} v_j \langle \bN_i,\bxi\rangle + \sqrt{n}\langle\bN_i,\bN_j\rangle 
       \Big) \\
       = &~ \kappa'\left( \sqrt{n}\bY_{ij} \right) \Big(  
       \sqrt{\lambda-1}  v_i \bxi \mathbf{1}_{j=\ell} +
       \sqrt{\lambda-1}  v_j \bxi \mathbf{1}_{i=\ell} +
       \bN_i\mathbf{1}_{j=\ell} + \bN_j \mathbf{1}_{i=\ell}
       \Big) .
   \end{align*}
   Denote the following vectors $\ba_{ij}^{(\ell)},\bb_{ij}^{(\ell)}\in \R^n$, $1\le i,j,\ell \le p$,  so that $\nabla_{\bN_\ell}\bK_{i,j}=\ba_{ij}^{(\ell)}+\bb_{ij}^{(\ell)}$:
   \begin{align*}
       \ba_{ij}^{(\ell)} &= \sqrt{\lambda-1} \kappa'(\sqrt{n}\bY_{ij})\cdot \left(  v_i \bxi \mathbf{1}_{j=\ell} +   v_j \bxi \mathbf{1}_{i=\ell} \right), \\
       \bb_{ij}^{(\ell)} &= \kappa'(\sqrt{n}\bY_{ij})\cdot \left( \bN_i\mathbf{1}_{j=\ell} + \bN_j \mathbf{1}_{i=\ell} \right) .
   \end{align*}
   We have $\nabla_{\bN_\ell}F_{\bx,\by}(\bN) = \sum_{1\le i,j\le p}x_iy_i \nabla_{\bN_\ell} (\bK_{i,j} ) \in \R^n$. 
   The total gradient of $F_{\bx,\by}$ with respect to $\bN$ satisfies
   \begin{equation}
       \begin{aligned}
   \|\nabla_{\bN} F_{\bx,\by}(\bN) \|^2 
       &= \sum_{\ell=1}^p \|\nabla_{\bN_\ell} F_{\bx,\by}(\bN)\|^2 
       = \sum_{\ell=1}^p \bigg\|\sum_{1\le i,j\le p}x_iy_i \nabla_{\bN_\ell} (\bK_{i,j} ) \bigg\|^2  \nonumber \\
       &\le 
       2\sum_{\ell=1}^p \bigg\|\sum_{1\le i,j\le p}x_iy_i \ba_{ij}^{(\ell)} \bigg\|^2 
       +
       2\sum_{\ell=1}^p \bigg\|\sum_{1\le i,j\le p}x_iy_i \bb_{ij}^{(\ell)} \bigg\|^2 .
   \end{aligned}
   \end{equation}

   Note that the vectors $\bb_{ij}^{(\ell)}$ depend on the spike only through $\kappa'(\bY_{ij})$. To bound the second term, we apply (10) in \cite{fan2019spectral}: 
   \begin{align}
       \sum_{\ell=1}^p \bigg\|\sum_{1\le i,j \le p} x_i y_j \bb^{(\ell)}_{i,j} 
       \bigg\|^2 
       \le
       {4\|\bN\|^2\|\by\|_{\infty}}\max_{1\le \ell \le p} \bigg( \sum_{i\ne \ell} \big( \kappa'(\sqrt{n}\bY_{i\ell} \big)^4 \bigg)^{1/2}. 
   \end{align}
   As for the first term, by the definition of $\ba_{ij}^{(\ell)}$,
   \begin{align*}
       \sum_{\ell=1}^p \bigg\| \sum_{1\le i,j,\le p} x_iy_j \ba_{ij}^{(\ell)} \bigg\|^2
       = &~ (\lambda-1) \|\bxi\|^2\sum_{\ell=1}^p \bigg[
       y_\ell \sum_{i=1,i\ne \ell}^p \kappa'(\sqrt{n}\bY_{i\ell})x_i v_i  + x_\ell \sum_{j=1,j\ne \ell}^p \kappa'(\sqrt{n}\bY_{j\ell}) y_j v_j
       \bigg]^2 \\
        \leq  &~
       2(\lambda-1)\|\bxi\|^2 \sum_{\ell=1}^p y_\ell^2 \bigg( \sum_{i=1,i\ne \ell}^p \kappa'(\sqrt{n}\bY_{i\ell})x_i v_i\bigg)^2
       \\ &~ +
       2\tau^2\|\bxi\|^2 \sum_{\ell=1}^p x_\ell^2 \bigg( \sum_{j=1,i\ne \ell}^p \kappa'(\sqrt{n}\bY_{j\ell})y_j v_j\bigg)^2 .
   \end{align*}
   Using the Cauchy-Schwarz inequality and $\|\bx\|,\|\by\|\le 1$,
   \begin{align*}
       \sum_{\ell=1}^p y_\ell^2 \bigg( \sum_{i=1,i\ne \ell}^p \kappa'(\sqrt{n}\bY_{i\ell})x_i v_i\bigg)^2
       &\le \sum_{\ell=1}^p y_\ell^2 \bigg(\sum_{i=1,i\ne \ell}^p x_i^2\bigg)
       \bigg( \sum_{i=1,i\ne \ell}^p |\kappa'(\sqrt{n}\bY_{i\ell})|^2 v_i^2 \bigg)\\
       &\le \|\bv\|_{\infty}^2 \max_{1\le \ell \le p} \sum_{i\in \mathcal{S}\setminus\{\ell\}} |\kappa'(\sqrt{n}\bY_{i\ell})|^2. 
   \end{align*}
   Combining the above bounds yields the lemma.

\end{document}